\documentclass[11pt]{article}
\usepackage[utf8]{inputenc}

\usepackage[margin=2cm]{geometry}
\usepackage{mathtools}
\usepackage{pgf}
\usepackage{multimedia}
\usepackage{amsmath,amssymb,mathrsfs,epsfig}
\usepackage{fancybox}
\usepackage{cancel}
\usepackage{graphicx}
\usepackage[cmtip,arrow]{xy}
\usepackage{pb-diagram,pb-xy}
\usepackage{xcolor}
\usepackage{wrapfig}
\usepackage{amsthm}
\usepackage{comment}
\usepackage[small,compact]{titlesec}
\usepackage{bm}
\usepackage{tikz}
\usepackage{bbm}
\usepackage{algorithm,stmaryrd}
\usepackage{algorithmic}
\usepackage{clrscode}
\usepackage{url}

\usetikzlibrary{calc,angles,quotes,decorations.markings}
\usepackage{subcaption}

\newtheorem{remark}{Remark}
\newtheorem{definition}{Definition}
\newtheorem{lemma}{Lemma}
\newtheorem{question}{Question}
\newtheorem{corollary}{Corollary}
\newtheorem{theorem}{Theorem}

\DeclareMathOperator\supp{supp}

\newtheorem{proposition}{Proposition}

\newcommand{\squishlist}{
 \begin{list}{$\bullet$}
  { \setlength{\itemsep}{0pt}
     \setlength{\parsep}{3pt}
     \setlength{\topsep}{3pt}
     \setlength{\partopsep}{0pt}
     \setlength{\leftmargin}{1.5em}
     \setlength{\labelwidth}{1em}
     \setlength{\labelsep}{0.5em} } }
\newcommand{\squishlisttwo}{
 \begin{list}{$\bullet$}
  { \setlength{\itemsep}{0pt}
     \setlength{\parsep}{0pt}
    \setlength{\topsep}{0pt}
    \setlength{\partopsep}{0pt}
    \setlength{\leftmargin}{2em}
    \setlength{\labelwidth}{1.5em}
    \setlength{\labelsep}{0.5em} } }
\newcommand{\squishend}{
  \end{list}  }
  \DeclareMathOperator*{\argmax}{arg\,max}
  
\usepackage{cite}
\newcommand\RectTri[4][thick,green!50!black,text=black]{
\coordinate  (C) at #2;
\coordinate (B) at #3;
\coordinate  (A) at #4;

\draw[#1] 
  (C) -- 
  node[auto] {$a'_{\ell}$} (A) -- 
  node[auto] {$N'_{\ell}$} (B) --
  node[auto] {$t_{\ell}$} 
  (C)
  pic ["$\theta_{\ell}$",draw,black,thick,angle radius=2.2cm] {angle = B--C--A} 
  pic ["$\phi_{\ell}$",draw,black,thick,angle radius=1cm] {angle = C--A--B}
  ;
}

\begin{document}

\title{Phase Retrieval for $L^2([-\pi,\pi])$ via the  Provably Accurate and Noise Robust Numerical Inversion of Spectrogram Measurements}

\author{Mark Iwen\thanks{Michigan State University, Department of Mathematics, and the Department of Computational Mathematics, Science and Engineering (CMSE), \texttt{markiwen@math.msu.edu}.  Supported in part by NSF DMS 1912706.}, 
Michael Perlmutter\thanks{University of California, Los Angeles, Department of Mathematics, \texttt{perlmutter@math.ucla.edu}}, 
Nada Sissouno\thanks{Technical University of Munich, Faculty of Mathematics, Garching b. M\"unchen,-
\& Helmholtz Zentrum M\"unchen, ICT Information Communication Technology, Neuherberg, Germany \texttt{sissouno@ma.tum.de}. The author acknowledges partial funding from an Entrepreneurial Award in the Program "Gobal Challenges for Women in Math Science" funded by the Faculty of Mathematics at the Technical University of Munich.}, 
Aditya Viswanathan\thanks{University of Michigan -- Dearborn, Department of Mathematics \& Statistics, \texttt{adityavv@umich.edu}. Supported in part by NSF DMS 2012238.}} 

\maketitle

\begin{abstract}
In this paper, we focus on the approximation of smooth functions $f: [-\pi, \pi] \rightarrow \mathbb{C}$, up to an unresolvable global phase ambiguity, from a finite set of Short Time Fourier Transform (STFT) magnitude (i.e., spectrogram) measurements.  Two algorithms are developed for approximately inverting such measurements, each with theoretical error guarantees establishing their correctness.  A detailed numerical study also demonstrates that both algorithms work well in practice and have good numerical convergence behavior.

\end{abstract}

\vspace{1in}

\section{Introduction}\label{sec: intro}

We consider the approximate recovery of a smooth function $f: \mathbb{R} \rightarrow \mathbb{C}$ supported inside of a compact interval $I \subset \mathbbm{R}$
from a finite set of noisy spectrogram measurements of the form
\begin{equation*}
Y_{\omega,\ell} \coloneqq \left|\int_{-\infty}^{\infty}f(x)\tilde{m}\left(x-\frac{2\pi}{L}\ell\right)\mathbbm{e}^{-\mathbbm{i}x\omega}dx\right|^2+\eta_{\omega,\ell}.
\end{equation*}
Here $\tilde{m}: \mathbb{R} \rightarrow \mathbb{C}$ is a known mask, or window, and the $\eta_{\omega,\ell}$ are arbitrary additive measurement errors.  Without loss of generality, we will assume that $I \subseteq [-\pi,\pi]$ and seek to characterize how well the function $f$, with its domain restricted to $[-\pi,\pi]$, 
can be approximated using $d L$ measurements of this form for $d$ frequencies $\omega$ at each of $L$ shifts $\ell$.  
Toward that end, we present two algorithms which can provably approximate any such function $f$ (belonging to a general regularity class defined below in Definition~\ref{def: Fourier decay}) up to a global phase multiple using spectrogram measurements of this type  resulting from two different types of masks $\tilde{m}$.  As we shall see, both algorithms ultimately work by approximating finitely many Fourier series coefficients of $f$.

Inverse problems of this type appear in many applications including optics \cite{walther1963question}, astronomy \cite{Fienup1987}, and speech signal processing  \cite{griffin1984signal,balan2006signal} to name just a few.  In this paper we are primarily motivated by phaseless imaging applications such as ptychography \cite{R2008ptychography}, in which Fourier magnitude data is collected from overlapping shifts of a mask/probe (e.g., a pinhole) across a specimen and then used to recover the specimen's image.  Indeed, these types of phaseless imaging applications directly motivate the types of masks $\tilde{m}: \mathbb{R} \rightarrow \mathbb{C}$ considered below.  In particular, we consider two types of masks $\tilde{m}$ including both $(i)$ relatively low-degree trigonometric polynomial masks representing masking the sample $f$ with shifts of a periodic structure/grating, and $(ii)$ compactly supported masks representing the translation of, e.g., an aperture/pinhole across the sample during imaging.  Note that first type of periodic masks are reminicent of some of the Coded Diffraction Pattern type measurements for phase retrieval analyzed by Cand{\`e}s et al. in the discrete (i.e., finite-dimensional $f$ and $\tilde{m}$) setting \cite{candes2015phase,CANDES2015277}.   (See Section 1 of \cite{Perlmutter2020} for a related discussion.)  The second type of compactly supported masks, on the other hand, correspond more closely to standard ptychographic setups in which Fourier magnitude data is collected from small overlapping portions of a large sample $f$ in order to eventually recover its global image.

Although a number of algorithms exhibiting great empirical success were designed decades ago for phaseless imaging, e.g., \cite{fienup1982phase}, \cite{GSaxtonAltProj}, \cite{griffin1984signal}, the mathematical community has only recently begun to  undertake the challenge of designing measurement setups and corresponding recovery algorithms with provable accuracy and reconstruction guarantees.  The vast majority of those theoretical works have only addressed discrete (i.e., finite-dimensional) phase retrieval problems, (see  e.g., \cite{balan2006signal}, \cite{alexeev2014phase}, \cite{candes2015phase}, \cite{CANDES2015277}, \cite{gross2015improved}, \cite{FKMS20}) where the signal of interest and measurement masks are both discrete vectors and where the relevant measurement vectors are generally random and globally supported. 

In this paper, we develop a provably accurate numerical method\footnote{Numerical implementations of the methods proposed here are available at \url{https://bitbucket.org/charms/blockpr}.} for approximating smooth functions $f: \mathbb{R} \rightarrow \mathbb{C}$ from a finite set of Short-Time Fourier Transform (STFT) magnitude measurements. Though there has been general work concerning the uniqueness and stability of reconstruction from STFT magnitude measurements  in this setting (see, e.g., recent work by Alaifari, Cheng, Daubechies, and their collaborators \cite{Alaifari2020}, \cite{Cheng2020}), to the best of our knowledge, no prior work exists concerning the development or analysis of provably accurate numerical methods for actually carrying out such reconstructions from a finite set of such measurements.  Perhaps the closest prior work is that of Thakur \cite{thakur2011reconstruction}, who gives an algorithm for the reconstruction of real-valued bandlimited functions up to a global sign from the absolute values of their point samples, and that of Gr\"{o}chenig \cite{G2019}, who considers/surveys similar results in shift-invariant spaces.  Other related work includes that of Alaifari et al. \cite{alaifari2019stable}, which proves (among other things) that one can not hope to stably recover a periodic function up to a single global phase using a trigonometric polynomial mask of degree $\rho/2$, as done below, unless its Fourier series coefficients do not vanish on any $\rho$ consecutive integer frequencies in between two other frequencies with nonzero Fourier series coefficients.  This helps to motivate the function classes we consider recovering here. (In particular, if a function $f$ satisfies Definition~\ref{def: Fourier decay} below, then any strings of zero Fourier series coefficients in $\{ \hat{f}(n) \}_{n \in \mathbbm{Z}}$ longer than a certain finite length must be part of an infinite string of zero Fourier coefficients associated with all frequencies beyond a finite cutoff.)  We also refer the reader to \cite{IMP2019} and \cite{Cheng2020} for similar considerations in the discrete setting.

\subsection{Problem Setup and Main Results}
 
Let $\tilde{m},f:\mathbb{R}\rightarrow\mathbb{C}$ be  $C^k$-functions for some $k\geq 2$.  Let $d$ be an odd number, and let $K$ and $L$ divide $d$. Let 
$
    \mathcal{D}=\{-\frac{d-1}{2},\ldots,0,\ldots, \frac{d-1}{2}\},
$
and
let
$\mathbf{Y}=(Y_{\omega,\ell})_{\omega,\ell\in \mathcal{D}}$ be the $d\times d$ measurement matrix defined by 
\begin{equation}\label{eq:conti-meas2}
Y_{\omega,\ell} \coloneqq \left|\int_{\mathbb{R}} f(x)\tilde{m}\left(x-\frac{2\pi}{d}\ell\right)\mathbbm{e}^{-\mathbbm{i}x\omega}dx\right|^2+\eta_{\omega,\ell},
\end{equation}
where  ${\bm{ \eta}}=(\eta_{\omega,\ell})_{\omega,\ell\in\mathcal{D}}$ is an arbitrary additive noise matrix. The goal of this paper is to address the following question. 
\begin{question}
Under what conditions on $f$ and $\tilde{m}$ can we produce an efficient and noise robust algorithm which provably recovers $f$ from the $K\times L$ measurement matrix $\mathbf{Y}_{K,L}$ obtained by subsampling equispaced entries of $\mathbf{Y}$.
\end{question}

In order to partially answer this question, we will assume that $f$ satisfies a regularity assumption defined below in Definition \ref{def: Fourier decay} and also that  one of the following two assumptions hold:
\begin{enumerate}
\item \label{as: trigpoly}$f$ is compactly supported with $\text{supp}(f)\subseteq [-\pi,\pi]$ and $\tilde{m}$ is a trigonometric polynomial given  by
\begin{equation*}
    \tilde{m}(x)=\sum_{p=-\rho/2}^{\rho/2}\widehat{m}(p)\mathbbm{e}^{\mathbbm{i}px}
\end{equation*}
for some even number $\rho < d$ and some complex numbers  $\widehat{m}(-\rho/2),\ldots,\widehat{m}(0),\ldots,\widehat{m}(\rho/2)$. 
\item \label{as: compact}Both $f$ and $\tilde{m}$ are compactly supported with
$\text{supp}(f)\subseteq (-a,a)$ and   $\text{supp}(\tilde{m})\subseteq (-b,b)$
for some \newline $a$ and $b$ such that
$a+b\leq \pi$.
\end{enumerate}

We will introduce a four-step method which relies on recovering the Fourier coefficients of $f$. In our discretization step, we approximate the mask $\tilde{m}$ by a function with finitely many nonzero Fourier coefficients. Therefore, we effectively regard the mask as being compactly supported in the frequency domain.  As mentioned above, several previous works, including \cite{IMP2019}, \cite{alaifari2019stable}, and \cite{Cheng2020}, have noted that this implies that the recovery of $f$ is impossible if $f$ has many consecutive Fourier coefficients which are equal to zero followed by nonzero Fourier coefficients at higher frequencies. Moreover, if there are many consecutive small Fourier coefficients followed by larger coefficients at higher frequencies, the problem is inherently unstable.  
Therefore, we will remove such pathological functions from consideration by assuming that our function $f$ is a member of the following function class for a suitable choice of $\beta$. This choice of $\beta$ will depend on whether $f$ and $\tilde{m}$ satisfy Assumption \ref{as: trigpoly} or Assumption \ref{as: compact}, respectively.   

\begin{definition}\label{def: Fourier decay} Let $\beta$ be a positive integer and let $D_n = \max_{|m-n|<\beta/2} |\widehat{f}(m)|$.
We say that $f$ has \emph{$\beta$ Fourier decay} if $D_n\geq D_{n'}$ whenever $|n|\leq |n'|$.
\end{definition}

A useful property of this function class, which follows immediately from the definition, is summarized in the following remark.
 \begin{remark} \label{la:n-sequence}
Suppose $f$ has $\beta$ Fourier decay, and let $a, n \in \mathbbm{Z}$ with $|a|<|n|$.  
Then the string of $\beta-1$ consecutive integers centered around $a$ contains an integer $m$ such that  $|\widehat{f}(m)|\geq|\widehat{f}(n)|$.
\end{remark}


 We will show that functions satisfying Definition \ref{def: Fourier decay} can be reconstructed from $\mathbf{Y}$ using the following four-step approach:
 \begin{enumerate}
    \item Approximate the matrix of continuous measurements $\mathbf{Y}$, defined in terms of functions $f$ and $\tilde{m}$, by a matrix of discrete measurements $\mathbf{T}'$, defined in terms of corresponding vectors $\mathbf{x}$ and $\mathbf{z}$.
    \item Apply a discrete Wigner distribution deconvolution method \cite{Perlmutter2020} to recover a portion of the Fourier autocorrelation matrix \label{step: wigner} $\widehat{\mathbf{x}}\widehat{\mathbf{x}}^*$. 
    \item Recover $\widehat{\mathbf{x}}$, the discrete Fourier transform of $\mathbf{x}$,  via a greedy angular synchronization scheme along the lines of the one used in \cite{IVW16}.
    \item Estimate $f$ by a trigonometric polynomial with coefficients given by $\widehat{\mathbf{x}}$.
\end{enumerate}
The details of step \ref{step: wigner} are quite different depending on whether $f$ and $\tilde{m}$ satisfy Assumption \ref{as: trigpoly} or Assumption \ref{as: compact}. However, we emphasize that the other three steps of the process are identical in either case.
The result of this approach is two algorithms which allow for the reconstruction of $f$ under either Assumption \ref{as: trigpoly} or \ref{as: compact}, as well as two theorems providing theoretical guarantees. 
The following main results are variants of Corollaries \ref{cor: trigpoly} and \ref{cor: compact} presented in Section \ref{sec: convergence theorems}.

\begin{theorem}\label{thm: compact informal} Let $\mathcal{C}^k_{\rho/2}$ be the set of all compactly supported functions $f:\mathbbm{R} \rightarrow \mathbbm{C}$ with $\text{supp}(f)\subseteq [-\pi,\pi]$ that are $C^k$-smooth for some $k\geq 5$ and that have $\rho/2$ Fourier decay. Then, there exist degree $\rho/2$ trigonometric polynomial masks $\tilde{m}$  such that for all $f \in \mathcal{C}^k_{\rho/2}$, $K = d \geq 2 \rho + 6$, and $L$ dividing $d$ with $2+\rho \leq L \leq 2\rho$ the trigonometric polynomial $f_e(x)$ output by Algorithm \ref{Big Algorithm} is guaranteed to satisfy
\begin{align*}
\min_{\theta\in[0,2\pi]} \left\|\mathbbm{e}^{\mathbbm{i}\theta} f - f_e \right\|_{L^2([-\pi,\pi])}^2 
\leq   C_{f,m}\bigg(
\left(\frac{1}{d}\right)^{k-9/2} + \frac{d^3}{L^{1/2}}\|\mathbf{\bm{\eta}_{d,L}}\|_F \bigg),
\end{align*}
where $\mathbf{\bm{\eta}_{d,L}}$ is the $d\times L$ matrix obtained by subsampling equispaced entries of $\bm{\eta}$ and $C_{f,m}$ is a constant only depending on $f,\widetilde{m},$ and $k$.
\end{theorem}

\begin{proof} Apply Corollary \ref{cor: trigpoly} with $s = \lceil (d+1)/2 \rceil$ and $r = d - s -1 \geq d/2 - 2$. The assumption that $d \geq 2\rho + 6$,  implies that $\rho \leq r-1$.  Noting now that $\kappa\coloneqq L-\rho \geq 2$ and applying Proposition~\ref{prop:mu_condition} for choices of $\tilde{m}$ satisfying \eqref{eqn: a0big} with $\kappa$ replaced by $\rho$ (since  $\rho \geq \kappa$), we have that $\mu_1^{-1} \leq C_m d$ for a mask-dependent constant $C_m$.
\end{proof}

Theorem~\ref{thm: compact informal} guarantees the existence of periodic masks which allow the exact recovery of all sufficiently smooth $f$ as above as $d \rightarrow \infty$ in the noiseless case (i.e., when $\bm{\eta}=\bm{0}$).  In particular, it is shown that a single mask $\tilde{m}$ will work with all sufficiently large choices of $d$ as long as $d$ has a divisor in $[\rho+2, 2 \rho]$.  Furthermore, Theorem~\ref{thm: compact informal} demonstrates that Algorithm \ref{Big Algorithm} is robust to small amounts of arbitrary additive noise on its measurements for any fixed $d$.  We note here that the $d^3$ term in front of the noise term $\|\mathbf{\bm{\eta}_{d,L}}\|_F$ is almost certainly highly pessimistic, and the numerical results in Section~\ref{sec:Eval} indicate that the method performs well with noisy measurements in practice.  We expect that this $d^3$ dependence in our theory can be reduced, especially for more restricted classes of functions $f$ that are compatible with less naive angular synchronization approaches than the one utilized here. (See, for example, recent work on angular synchronization approaches for phase retrieval by Filbir et al. \cite{FMS20}.)

Focusing on the total number of STFT magnitude measurements \eqref{eq:conti-meas2} used by Algorithm \ref{Big Algorithm}, we can see that Theorem~\ref{thm: compact informal} guarantees that $KL \leq 2d\rho$ will suffice for accurate reconstruction when the mask $\tilde{m}$ is a trigonometric polynomial.  In particular, this is linear in $d$ for a fixed $\rho$.  As we shall see below, the situation appears more complicated when $\tilde{m}$ is compactly supported.  In particular, Theorem \ref{thm: ptycmainres} stated below requires $KL = d^2 / 3$ STFT magnitude measurements in that setting (and more generally, the argument we give here always requires $KL \geq C b d^2$, where $C$ is an absolute constant, and $b$ is the support size of the mask as per Assumption~\ref{as: compact}).

\begin{theorem}
\label{thm: ptycmainres}
Let $\mathcal{\tilde{C}}^k_{a,\beta}$ be the set of all compactly supported functions $f:\mathbbm{R} \rightarrow \mathbbm{C}$ with $\text{supp}(f)\subseteq (-a,a)$ for some $a \in (0, \pi-3/4)$ that are $C^k$-smooth for some $k\geq 4$ and have $\beta$ Fourier decay.  Let $b =3/4$, and then fix $d = L$ to be a multiple of three large enough that all of the following hold: $\beta < \lceil db / 2 \pi\rceil - 1/2$, $s = r = \lceil db / 2 \pi \rceil < d/8 - 1$, and $ 5 d / 21 < \delta = \lfloor db / \pi \rfloor < d/4$.  Finally, set $K = d/3$.  Then, for any compactly supported mask $\tilde{m}$ with $\text{supp}(\tilde{m})\subseteq (-b,b)$ and $\mu_2 > 0$ (see \eqref{eq: mu2} and \eqref{eq:def-x} for the definition of $\mu_2$) the trigonometric polynomial $f_e(x)$ output by Algorithm \ref{Big Algorithm Compact} is guaranteed to satisfy
\begin{align*}
\min_{\theta\in[0,2\pi]} &\left\|\mathbbm{e}^{\mathbbm{i}\theta} f - f_e \right\|_{L^2([-\pi,\pi])}^2 
\leq C_{f,m}\bigg(\frac{1}{\mu_2\sigma_{\min}(\mathbf{W})d^k}  + \frac{\|\mathbf{\bm{\eta}_{K,d}}\|_F}{\mu_2\sigma_{\min}(\mathbf{W})}
+ \left(\frac{1}{d}\right)^{2k-2}\bigg)
\end{align*}
for all $f \in \mathcal{\tilde{C}}^k_{a,\beta}$, where $C_{f,m}$ is a constant only depending on $f,\widetilde{m},$ and $k$. Here $\sigma_{\min}(\mathbf{W})$ denotes the smallest singular value of the $(2(d/3 - \lfloor 3d / 4\pi \rfloor)-1)\times \lceil db / 2 \pi \rceil $ partial Fourier matrix $\mathbf{W}$ defined in Section~\ref{sec: wddas2} and 
$\mathbf{\bm{\eta}_{K,d}}$ is the $K\times d$ matrix obtained by subsampling equispaced entries of $\bm{\eta}$.
\end{theorem}

\begin{proof} We first note that $\delta + (s+1)/2 < 5d/16 \leq K \leq 10d/21 < 2 \delta$. Next, we apply Corollary \ref{cor: compact} with $s, r, \delta,$ and all other parameters set as above.  Next, we observe that $\mathbf{W}$ will be full rank given that it is a Vandermonde matrix. Therefore, $\sigma_{\min}(\mathbf{W}) > 0$ will always hold.  Finally, we note that, for any choice of $d$ and $b \leq \pi - a$, Proposition \ref{prop: mu2} guarantees the existence of a smooth and compactly supported mask $\tilde{m}$ with $\mu_2 > 0$.
\end{proof}

Theorem \ref{thm: ptycmainres} demonstrates that sufficiently smooth functions $f$ can be approximated well for measurement setups and masks having $\mu_2$ and $\sigma_{\min}(\mathbf{W})$ not too small.  Furthermore, Proposition \ref{prop: mu2} demonstrates that masks exist for which $\mu_2$ scales polynomially in $d$ (independently of $f$ and $k$).  It remains an open problem, however, to find a single compactly supported mask $\tilde{m}$ which will provably allow recovery for all choices of $d$, as well as optimal constructions of such masks more generally.  Nonetheless, our numerical results in Section~\ref{sec:Eval} demonstrate that Algorithm \ref{Big Algorithm Compact} does indeed work well in practice for a fixed compactly supported mask and that the mask we evaluate has reasonable values of $\mu_2$ for the range of choices of $d$ evaluated there.

 \subsection{Notation}\label{sec: notationa}

We will denote matrices and vectors by bold letters.
 We will let $\mathbf{M}_j$ denote the $j$-th column of a matrix $\mathbf{M}$ and, if $\mathbf{x}$ and $\mathbf{y}$ are vectors, we will let 
 \begin{equation*}
    \frac{\mathbf{x}}{\mathbf{y}}
 \end{equation*}denote their componentwise quotient. For any odd number $n$, we will let 
\begin{equation*}
    [n]_c \coloneqq \left[\frac{1-n}{2},\frac{n-1}{2}\right]\cap\mathbb{Z}
\end{equation*}
be the set of $n$ consecutive integers centered at the origin. In a slight abuse of notation, if $n$ is even, we will define $[n]_c\coloneqq [n+1]_c$, so that in either case $[n]_c$ is the smallest set of at least $n$ consecutive integers centered about the origin. We will let $d$ be an odd number, let $K$ and $L$ divide $d$, and let 
\begin{equation*}
    \mathcal{D}\coloneqq[d]_c, \quad \mathcal{K}\coloneqq[K]_c, \quad\text{and}\quad\mathcal{L}\coloneqq[L]_c.
\end{equation*}For $\ell\in\mathbb{Z}$, we let $S_\ell:\mathbb{C}^d\rightarrow\mathbb{C}^d$ be the circular shift operator defined for $\mathbf{x}=(x_p)_{p\in\mathcal{D}}$ by  
\begin{equation*}(S_{\ell}\mathbf{x})_p=\mathbf{x}_{p+\ell},
\end{equation*}
where the addition $p+\ell$ is interpreted to mean the unique element of $\mathcal{D}$ which is equivalent to $p+\ell$ modulo $d$. 

If $K$ and $L$ are integers which divide $d$, and  $\mathbf{M}=(M_{k,\ell})_{k,\ell\in\mathcal{D}}$ is a $d\times d$ matrix, we will let $\mathbf{M_{K,L}}$ be the $K\times L$ matrix obtained by subsampling $\mathbf{M}$ at equally spaced entries. That is, for $k\in\mathcal{K}$ and $\ell\in\mathcal{L}$, we let
\begin{equation}\label{eqn: subsample}
    (\mathbf{M_{K,L}})_{k,\ell}=M_{k\frac{d}{K}, \ell\frac{d}{L}}.
\end{equation}
We let $\mathbf{F_d}$  be the $d\times d$  Fourier matrix with entries given by 
\begin{equation*}
(\mathbf{\mathbf{F_d}})_{j,k} =
\frac{1}{d}\mathbbm{e}^{\frac{-2\pi\mathbbm{i} j k}{d}}
\end{equation*}
for $j,k\in\mathcal{D}$, and similarly let $\mathbf{F_L}$ and $\mathbf{F_K}$ be the $L\times L$ and $K\times K$ Fourier matrices with indices in $\mathcal{L}$ and $\mathcal{K}$, respectively.  Finally, we will often use generic constants whose values change from line to line, but whose dependencies on other quantities are explicitly tracked and noted.  These constants will be denoted by capital $C$ and have subscripts that indicate the mathematical objects on which they depend.

\section{Discretization}\label{sec: discretize}

Let $\tilde{m},f:\mathbb{R}\rightarrow\mathbb{C}$ be  $C^k$-functions for some $k\geq 2$ such that 
$
\text{supp}(f)\subseteq [-\pi,\pi],
$  and assume that either Assumption \ref{as: trigpoly} or Assumption \ref{as: compact} holds. We will define $m$ to be a periodic function which coincides with $\tilde{m}$ on $[-\pi,\pi]$.  Specifically, we let
\begin{equation*}
    m(x)\coloneqq \begin{cases} \tilde{m}(x)&\text{ if Assumption \ref{as: trigpoly} holds,}
\\
\sum_{n\in\mathbb{Z}}\tilde{m}(x+2\pi n)&\text{ if Assumption \ref{as: compact} holds}.
        \end{cases}
\end{equation*}

As in Section \ref{sec: intro}, 
let $\mathcal{D}$ be the set of $d$ consecutive integers centered at the origin, and
define $\mathbf{Z}=(Z_{\omega,\ell})_{\omega,\ell\in\mathcal{D}}$ to be the $d\times d$ matrix with entries given by 
\begin{equation*}
Z_{\omega,\ell} \coloneqq \left|\int_{\mathbb{R}} f(x)\tilde{m}\left(x-\frac{2\pi}{d}\ell\right)\mathbbm{e}^{-\mathbbm{i}x\omega}dx\right|^2.
\end{equation*}
Our goal is to recover $f$ from the matrix $\mathbf{Y}=(Y_{\omega,\ell})_{\omega,\ell\in\mathcal{D}}$ of noisy measurements given by 
\begin{equation*}
Y_{\omega,\ell} \coloneqq Z_{\omega,\ell}+\eta_{\omega,\ell},
\end{equation*}
where  ${\bm{ \eta}}=(\eta_{\omega,\ell})_{\omega,\ell\in\mathcal{D}}$ is an arbitrary additive noise matrix.
 Since the support of $f$ is contained in $[-\pi,\pi]$, we note that
\begin{equation}
    Z_{\omega,\ell} 
= \left|\int_{-\pi}^{\pi}f(x)\tilde{m}\left(x-\frac{2\pi}{d}\ell\right)\mathbbm{e}^{-\mathbbm{i}x\omega}dx\right|^2.\label{eq: restrict to pi}
\end{equation}Furthermore, 
under either Assumption \ref{as: trigpoly} or Assumption \ref{as: compact}, we note that we may replace $\tilde{m}$ with $m$ in \eqref{eq: restrict to pi}, i.e., 
\begin{align}Z_{\omega,\ell} &= \left|\int_{-\pi}^{\pi}f(x){m}\left(x-\frac{2\pi}{d}\ell\right)\mathbbm{e}^{-\mathbbm{i}x\omega}dx\right|^2. \label{eq: same with periodic}
\end{align}
Under Assumption \ref{as: trigpoly}, this is immediate since 
$\tilde{m}(x)=m(x)$ 
by definition.  
Under Assumption 
\ref{as: compact}, we note that 
\begin{equation*}
    \text{supp}(\tilde{m}-{m}) \subseteq (-\infty,b-2\pi]\cup[2\pi-b,\infty)
\end{equation*}
and that  $\left|\frac{2\pi\ell}{d}\right|<\pi$ for all $\ell\in\mathcal{D}$.  
Therefore, we have that 
\begin{equation*}
    \tilde{m}\left(x-\frac{2\pi}{d}\ell\right)-m\left(x-\frac{2\pi}{d}\ell\right)=0\quad\text{for all } |x|<\pi-b.
\end{equation*}
As a result, the assumptions that  the  support of $f$ is contained in $(-a,a)$ and that $a<\pi-b$ imply that
\begin{equation*}
    \int_{-\pi}^{\pi}f(x)\left(\tilde{m}\left(x-\frac{2\pi}{d}\ell\right)-m\left(x-\frac{2\pi}{d}\ell\right)\right)\mathbbm{e}^{-\mathbbm{i}x\omega}dx=0
\end{equation*}
and so \eqref{eq: same with periodic} follows.

For any $C^2$-smooth function $g:  \mathbb{R} \rightarrow \mathbb{C}$, we will define  
\begin{equation*}
\widehat{g}(n) \coloneqq \frac{1}{2\pi}\int_{-\pi}^{\pi} g(x) \mathbbm{e}^{-\mathbbm{i}nx}dx
\end{equation*}
for all $n\in\mathbb{Z}$, 
and note that, if $g$ is $2\pi$-periodic, we may use Fourier series to write
\begin{equation}\label{eq: fourier series}
    g(x) = \sum_{n\in \mathbb{Z}} \widehat{g}(n)\mathbbm{e}^{\mathbbm{i}nx}.
\end{equation}
We also note that, if $g$ is not $2\pi$-periodic, but its support is contained in $(-\pi,\pi)$, then \eqref{eq: fourier series} still holds for all $x\in(-\pi,\pi)$ since we may view $\{\widehat{g}(n)\}_{n\in\mathbb{Z}}$ as the Fourier coefficients of the periodized version of $g$. 
For any set $\mathcal{A}\subseteq\mathbb{Z}$, we define $P_\mathcal{A}$ to be the Fourier projection operator given by  
\begin{equation}\label{eq:Fourierprojection}
P_\mathcal{A}g(x) \coloneqq  \sum_{n\in \mathcal{A}} \widehat{g}(n)\mathbbm{e}^{\mathbbm{i}nx}.
\end{equation}

Now, let $r$, $s$, and $d$ be odd numbers with $r+s<d.$  
Let $\mathcal{R}\coloneqq [r]_c$,  $\mathcal{S}\coloneqq [s]_c$, and $\mathcal{D}=[d]_c$ be the sets of $r$,  $s$, and $d$ consecutive integers centered at the origin.  Let $\mathbf{T}\coloneqq (T_{\omega,\ell})_{\omega,\ell\in\mathcal{D}}$ denote the matrix  of  measurements obtained by replacing $f$ with $P_\mathcal{S}f$ and ${m}$ with $P_\mathcal{R}{m}$ in \eqref{eq: same with periodic}, i.e., the matrix whose entries are given by 
\begin{align}\label{eq:truncated-integral}
    T_{\omega,\ell} &\coloneqq \left|\int_{-\pi}^{\pi}P_\mathcal{S}f(x)P_\mathcal{R}{m}\left(x-\frac{2\pi}{d}\ell\right)\mathbbm{e}^{-\mathbbm{i}x\omega}dx\right|^2.
\end{align}
If Assumption \ref{as: trigpoly} holds, we will assume that $r>\rho+1$ which implies $P_\mathcal{R}m(x)=m(x)$.

The following lemma  provides a bound on the $\ell^\infty$-norm of the error matrix 
$\mathbf{Z}-\mathbf{T}.$
 
\begin{lemma}\label{lem: bound on E}
Let $r$, $s$, and $d$ be odd numbers with $r+s<d,$  and let  $\tilde{m}:\mathbb{R}\rightarrow\mathbb{C}$ and $f:\mathbb{R}\rightarrow\mathbb{C}$ be $C^k$-smooth functions for some $k\geq 2$. 
Then, under Assumption \ref{as: trigpoly},  we have
\begin{equation*}\label{eq:E-boundnew}
\|\mathbf{Z}-\mathbf{T}\|_\infty \leq C_{f,m}\left(\frac{1}{s}\right)^{k-1},
\end{equation*}
 and, under Assumption \ref{as: compact}, we have
\begin{equation*}\label{eq:E-boundnewcompact}
\|\mathbf{Z}-\mathbf{T}\|_\infty \leq C_{f,m}\bigg(\left(\frac{1}{s}\right)^{k-1}+\left(\frac{1}{r}\right)^{k-1}\bigg).
\end{equation*} 
In either case, $C_{f,m} \in \mathbbm{R}^+$ is a generic constant that depends only on $f$, $\tilde{m}$, and $k$ (and, in particular, is independent of $s$, $r$ and $d$).
\end{lemma}

To prove Lemma \ref{lem: bound on E}, we need the following auxiliary lemma.  Note in particular, it can be applied both to $2\pi$-periodic functions and to functions whose support is contained in $(-\pi,\pi).$ 
\begin{lemma} \label{lem:Intbound}
Let $k\geq 2,$ and let $g:\mathbb{R}\rightarrow\mathbb{C}$ be a 
$C^k$-smooth function such that \eqref{eq: fourier series} holds for all $x\in(-\pi,\pi)$.  Let $n \geq 3$ be an odd number,  let  $\mathcal{N} \coloneqq [n]_c,$ and let $\mathcal{A}$ be any subset of $\mathbb{Z}.$  Then, there exists a constant $C_g$ depending only on $g$ and $k$ such that 
\[ \|P_{\mathcal{A}}g\|_{L^\infty([-\pi,\pi])}\leq C_{g} \quad\textrm{and}\quad 
\|g - P_{\mathcal{N}}g\|_{L^\infty([-\pi,\pi])} \leq C_{g} \left( \frac{1}{n} \right)^{k-1},
\]
where $P_{\mathcal{A}}$ and $P_{\mathcal{N}}$ are the Fourier projection operators defined as in \eqref{eq:Fourierprojection}.
\end{lemma}
For a proof of Lemma \ref{lem:Intbound}, please see Appendix \ref{app: proof of discretization lemmas}.\\

\begin{proof}[The Proof of Lemma \ref{lem: bound on E}]

We note that the measurements given in \eqref{eq: same with periodic} and \eqref{eq:truncated-integral} may be written as \begin{equation*}
    Z_{\omega,\ell}=|M_{\omega,\ell}|^2 \quad\text{and}\quad T_{\omega,\ell}=|U_{\omega,\ell}|^2 ,
\end{equation*}
where
\begin{equation*}
M_{\omega,\ell}\coloneqq \int_{-\pi}^{\pi} f(x){m}\left(x-\frac{2\pi}{d}\ell\right)\mathbbm{e}^{-\mathbbm{i}x\omega}dx 
\quad\text{and}\quad 
    U_{\omega,\ell}\coloneqq \int_{-\pi}^{\pi}P_\mathcal{S}f(x)P_\mathcal{R}{m}\left(x-\frac{2\pi}{d}\ell\right)\mathbbm{e}^{-\mathbbm{i}x\omega}dx.
\end{equation*}
Lemma \ref{lem:Intbound} implies
\begin{equation*}
    \|P_\mathcal{R}{m}\|_{L^\infty([-\pi,\pi])}\leq C_{{m}}\quad \text{and}\quad
\|P_\mathcal{S}f\|_{L^\infty([-\pi,\pi])} \leq  C_f.
\end{equation*} 
Therefore,  
\begin{equation*}\label{eqn: Tboundnew}
|U_{\omega,\ell}| \leq 2\pi \|P_\mathcal{R}{m}\|_{L^\infty([-\pi,\pi])}
\|P_\mathcal{S}f\|_{L^\infty([-\pi,\pi])}\leq C_{f,{m}}.
\end{equation*}
Next, letting $\tilde{\ell}=2\pi\ell/d$, we note that \begin{align*}
M_{\omega,\ell}-U_{\omega,\ell} &=
\int_{-\pi}^\pi \big(f(x)-P_\mathcal{S}f(x)\big){m}(x-\tilde\ell)\mathbbm{e}^{-\mathbbm{i} \omega x}dx
+\int_{-\pi}^\pi P_\mathcal{S}f(x)\left({m}(x-\tilde{\ell})-P_\mathcal{R}{m}(x-\tilde{\ell})\right)\mathbbm{e}^{-\mathbbm{i} \omega x}dx.
\end{align*} Therefore, by  Lemma~\ref{lem:Intbound}
and the triangle inequality,  we get
\begin{align*}
|M_{\omega,\ell}-U_{\omega,\ell}| &\leq C_{f,m}\bigg(\left(\frac{1}{s}\right)^{k-1}+\|m-P_\mathcal{R}m\|_{L^\infty([-\pi,\pi])}\bigg).
\end{align*}
Thus, we may use the difference of squares formula to see 
\begin{align*}\nonumber
|Z_{\omega,\ell}-T_{\omega,\ell}| &
= (|M_{\omega,\ell}|+|U_{\omega,\ell}|)||M_{\omega,\ell}|-|U_{\omega,\ell}||
\leq (2|U_{\omega,\ell}|+|M_{\omega,\ell}-U_{\omega,\ell}|)|M_{\omega,\ell}-U_{\omega,\ell}|\\
&\leq C_{f,m}\bigg(1+\left(\frac{1}{s}\right)^{k-1}+\|m-P_\mathcal{R}m\|_{L^\infty([-\pi,\pi])}\bigg)\bigg(\left(\frac{1}{s}\right)^{k-1}+\|m-P_\mathcal{R}m\|_{L^\infty([-\pi,\pi])}\bigg).
\end{align*}
Under Assumption \ref{as: trigpoly}, we have $\|m-P_\mathcal{R}m\|_{L^\infty([-\pi,\pi])}=0,$ and thus,
\begin{equation*}
|Z_{\omega,\ell}-T_{\omega,\ell}| \leq C_{f,m}\bigg(1+\left(\frac{1}{s}\right)^{k-1}\bigg)\left(\frac{1}{s}\right)^{k-1}\leq C_{f,m}\left(\frac{1}{s}\right)^{k-1} .
\end{equation*}
Likewise, under Assumption \ref{as: compact}, Lemma \ref{lem:Intbound} implies  $\|m-P_\mathcal{R}m\|_{L^\infty([-\pi,\pi])}\leq C_m \left(\frac{1}{r}\right)^{k-1},$ and so 

\begin{align*}
|Z_{\omega,\ell}-T_{\omega,\ell}| &\leq C_{f,m}\bigg(1+\left(\frac{1}{s}\right)^{k-1}+\left(\frac{1}{r}\right)^{k-1}\bigg)\bigg(\left(\frac{1}{s}\right)^{k-1}+\left(\frac{1}{r}\right)^{k-1}\bigg)\\
&\leq C_{f,m}\bigg(\left(\frac{1}{s}\right)^{k-1}+\left(\frac{1}{r}\right)^{k-1}\bigg). 
\end{align*}

\end{proof} 

Algorithms \ref{Big Algorithm} and \ref{Big Algorithm Compact} rely on discretizing the integrals used in the definitions of our measurements. Towards this end, we define three vectors $\mathbf{x}\coloneqq (x_p)_{p\in\mathcal{D}},$ $\mathbf{y}\coloneqq (y_p)_{p\in\mathcal{D}},$ and $\mathbf{z}\coloneqq (z_p)_{p\in\mathcal{D}}$ by

\begin{align}\label{eq:def-x}
x_{p} &\coloneqq P_\mathcal{S}f\left(\frac{2\pi p}{d}\right),\quad 
y_{p} \coloneqq P_\mathcal{R}m\left(\frac{2\pi p}{d}\right),\quad\text{and}\quad z_p=m\left(\frac{2\pi p}{d}\right).
\end{align}
We note that under Assumption \ref{as: trigpoly}, we have $P_\mathcal{R}m(x)=m(x)$ and therefore $\mathbf{y}=\mathbf{z}.$ Under Assumption \ref{as: compact}, we have that  $\text{supp}(m)\cap [-\pi,\pi]\subseteq(-b,b).$ Therefore,   $\text{supp}(\mathbf{z})\subseteq[\delta+1]_c$, where $\delta \coloneqq \lfloor\frac{b}{\pi}d\rfloor.$
The following lemma shows that the integral used in the definition of $\mathbf{T}$ can be rewritten as a discrete sum. Please see Appendix \ref{app: proof of discretization lemmas} for a proof.

\begin{lemma}\label{lem: discretesum}
Let  $\mathbf{x}=(x_p)_{p\in\mathcal{D}}$ and  $\mathbf{y}=(y_p)_{p\in\mathcal{D}}$  be  defined as in
\eqref{eq:def-x}.
 Then, for all $\omega\in\mathcal{D}$, $\ell\in\mathbb{Z}$, and $\tilde{\ell}=\frac{2\pi\ell}{d}$, we have that
\[
 \int_{-\pi}^{\pi}P_\mathcal{S}f(x)P_\mathcal{R}m(x-\tilde\ell)\mathbbm{e}^{-\mathbbm{i}x\omega}dx=\frac{2\pi}{d}\sum_{p\in\mathcal{D}} x_p y_{p-\ell} \mathbbm{e}^{-2\pi \mathbbm{i} \omega p/d},
\]
and as a consequence, 
\begin{equation}\label{eq: discreteT}
T_{\omega,\ell} = 
\frac{4\pi^2}{d^2} \bigg|\sum_{p\in\mathcal{D}} x_p y_{p-\ell} \mathbbm{e}^{-2\pi \mathbbm{i} \omega p/d} \bigg|^2.
\end{equation}
\end{lemma}
The matrix $\mathbf{T}$ depends on the vector $\mathbf{y}$ which is obtained by sampling the trigonometric polynomial $P_\mathcal{R}m$. By construction, $\mathbf{y}$ is not compactly supported, even under Assumption \ref{as: compact}. In Section \ref{sec: WDD}, we will apply a Wigner Deconvolution method based on \cite{Perlmutter2020} to invert our discretized measurements. In order to do this, we will need to use the vector $\mathbf{z}$ which is obtained by subsampling $m$ rather than $P_\mathcal{R}{m}.$ (By construction, $\mathbf{z}$ will be compactly supported under Assumption \ref{as: compact}, and under Assumption \ref{as: trigpoly}, we have $\mathbf{y}=\mathbf{z}$ and so this makes no difference.)  This motivates the following lemma which shows that $\mathbf{T}$ is well-approximated by the matrix
  $\mathbf{T'}=(T'_{\omega,\ell})_{\omega,\ell\in\mathcal{D}}$  obtained by replacing $\mathbf{y}$ with $\mathbf{z}$ in \eqref{eq: discreteT}, i.e.,
 \begin{equation}\label{eq: Tprime}
     T'_{\omega,\ell}=\frac{4\pi^2}{d^2}\bigg| \sum_{p\in\mathcal{D}} x_p z_{p-\ell} \mathbbm{e}^{-2\pi \mathbbm{i} \omega p/d}\bigg|^2.
 \end{equation}
 \begin{lemma}\label{lem TvsTprime} Let $\mathbf{T}$ and $\mathbf{T}'$ be the matrices defined in \eqref{eq:truncated-integral} and \eqref{eq: Tprime}. Then, under  Assumption \ref{as: trigpoly}, we have 
 \begin{equation*}
 \|\mathbf{T}-\mathbf{T'}\|_\infty=0,
 \end{equation*} and under Assumption   \ref{as: compact},
 \begin{equation*}
 \|\mathbf{T}-\mathbf{T'}\|_\infty\leq C_{f,m}\left(\frac{1}{r}\right)^{k-1}.
 \end{equation*}
 \end{lemma}
 \begin{proof}
Under Assumption \ref{as: trigpoly}, we have  $\mathbf{y}=\mathbf{z}.$ Thus by \eqref{eq: discreteT} and \eqref{eq: Tprime} we have   $\mathbf{T}=\mathbf{T}'$ and therefore the first claim is immediate.  
To prove the second claim, we will assume Assumption \ref{as: compact} holds and use arguments similar to those used in the proof of Lemma \ref{lem: bound on E}. Let 
 \begin{equation*}
     U_{\omega,\ell}= \frac{2\pi}{d} \sum_{p\in\mathcal{D}} x_p y_{p-\ell} \mathbbm{e}^{-2\pi \mathbbm{i} \omega p/d}\quad\text{and}\quad
     U'_{\omega,\ell}= \frac{2\pi}{d}\sum_{p\in\mathcal{D}} x_p z_{p-\ell} \mathbbm{e}^{-2\pi \mathbbm{i} \omega p/d}.
 \end{equation*}
 Then by Lemma \ref{lem: discretesum} we have 
 \begin{equation*}
     T_{\omega,\ell}=|U_{\omega,\ell}|^2 \quad\text{and}\quad
T'_{\omega,\ell}=|U'_{\omega,\ell}|^2.
 \end{equation*}
 By Lemma \ref{lem:Intbound} and the fact that $m$ is a continuous periodic function,  we see 
 \begin{align*}
     \|\mathbf{x}\|_\infty &\leq \|P_\mathcal{B}f\|_{L^\infty([-\pi,\pi])}\leq C_f,\\
      \|\mathbf{y}\|_\infty &\leq \|P_\mathcal{R}m\|_{L^\infty([-\pi,\pi])}\leq C_m,\quad\text{and}\\
     \|\mathbf{z}\|_\infty &\leq \|m\|_{L^\infty([-\pi,\pi])}\leq C_m. 
 \end{align*}
 Therefore,  
 \begin{equation*}
     |U_{\omega,\ell}|+|U'_{\omega,\ell}| \leq C_{f,m}.
 \end{equation*}
To bound $|U_{\omega,\ell}-U'_{\omega,\ell}|$, we may again apply  Lemma \ref{lem:Intbound}, to see 
 \begin{equation*}
     |U_{\omega,\ell}-U'_{\omega,\ell}|\leq 2\pi \|\mathbf{x}\|_\infty \|\mathbf{y}-\mathbf{z}\|_\infty \leq C_f \|m-P_\mathcal{R}m\|_{L^\infty([-\pi,\pi])} \leq C_{f,m} \left(\frac{1}{r}\right)^{k-1}.
 \end{equation*}
 Therefore, by the same reasoning as in the proof of Lemma \ref{lem: bound on E}, we have 
 \begin{equation*}
     |T_{\omega,\ell}-T'_{\omega,\ell}|\leq (|U_{\omega,\ell}|+|U'_{\omega,\ell}|)(|U_{\omega,\ell}-U'_{\omega,\ell}|)\leq  C_{f,m} \left(\frac{1}{r}\right)^{k-1}.
 \end{equation*}
 \end{proof}

\section{Wigner Deconvolution}\label{sec: WDD}

In this section, we will use a Wigner Deconvolution method based on \cite{Perlmutter2020} to recover $\mathbf{x}$ from the matrix $\mathbf{T}'$ defined in \eqref{eq: Tprime}.
In order to do this, we
 let $\mathbf{E}$ be the total  error matrix defined by 
\begin{equation*}
    \mathbf{E}\coloneqq \mathbf{Y}-\mathbf{T}'.
\end{equation*}
We note that $\mathbf{E}$  can be decomposed by
\begin{equation*}
    \mathbf{E}=(\mathbf{Z}-\mathbf{T}')+\bm{\eta},
\end{equation*}
where $(\mathbf{Z}-\mathbf{T}')$ is the error due to discretization and $\bm{\eta}$ is measurement noise. 
 Let $K$ and $L$ divide $d.$ 
 Let $\mathbf{E_{K,L}},$ $\mathbf{T'_{K,L}},$ and $\mathbf{\bm\eta_{K,L}}$ be the $K\times L$ matrices obtained by subsampling the columns of $\mathbf{E},$ $\mathbf{T'}$, and $\bm{\eta}$ as in \eqref{eqn: subsample}.
  Similarly to \cite{Perlmutter2020}, we introduce the quantities $\tilde{\mathbf{E}}$ and $\tilde{\mathbf{T}}$ defined by \begin{equation*}\tilde{\mathbf{E}}\coloneqq \mathbf{F_L}\mathbf{E_{K,L}}^T\mathbf{F_K}^T\quad\text{ and }\quad 
 \mathbf{\tilde{T}}\coloneqq \mathbf{F_L}(\mathbf{T'_{K,L}})^T\mathbf{F_K}^T.
 \end{equation*}
 Since $\sqrt{L}\mathbf{F_L}$ and $\sqrt{K}\mathbf{F_K}$ are unitary, we have 
\begin{equation*}
\|\tilde{\mathbf{E}}\|_F= \|\mathbf{F_L}\mathbf{E_{K,L}}^T\mathbf{F_K}^T\|_F\leq \frac{1}{\sqrt{KL}}\|\mathbf{E_{K,L}}\|_F\leq \|\mathbf{Z}-\mathbf{T'}\|_\infty + \frac{1}{\sqrt{KL}}\|\mathbf{\bm{\eta}_{K,L}}\|_F.
\end{equation*}
Therefore, Lemmas \ref{lem: bound on E} and \ref{lem TvsTprime} imply that under Assumption \ref{as: trigpoly} we have
\begin{equation}\label{eq: initial Etildebound} 
\|\tilde{\mathbf{E}}\|_F\leq C_{f,m}\left(\frac{1}{s}\right)^{k-1} + \frac{1}{\sqrt{KL}}\|\mathbf{\bm{\eta}_{K,L}}\|_F,
\end{equation} and that under Assumption \ref{as: compact} we have
\begin{equation}\label{eq: initial Etildeboundcompact}
\|\tilde{\mathbf{E}}\|_F\leq C_{f,m}\bigg(\left(\frac{1}{s}\right)^{k-1}+\left(\frac{1}{r}\right)^{k-1}\bigg) + \frac{1}{\sqrt{KL}}\|\mathbf{\bm{\eta}_{K,L}}\|_F.
\end{equation}
It follows from Theorem 4 of \cite{Perlmutter2020} that 
\begin{align}
\tilde{T}_{\ell,\omega}
&=
4\pi^2d
\sum_{q\in\left[\frac{d}{L}\right]_c}
\sum_{p\in\left[\frac{d}{K}\right]_c}\left(\mathbf{\mathbf{F_d}}\left({\widehat{\bf {\bf x}}}\circ S_{qL-\ell}\overline{{\widehat{\bf {\bf x}}}}\right)\right)_{\omega-pK}\left(\mathbf{F_{d}}\left(\widehat{{\bf {\bf z}}}\circ S_{\ell-qL}\overline{\widehat{{\bf {\bf z}}}}\right)\right)_{\omega-p K}+\tilde{E}_{\ell,\omega} \label{eq: doublealiasingfourier}\\
&=
\frac{4\pi^2}{d}
\sum_{q\in\left[\frac{d}{L}\right]_c}
\sum_{p\in\left[\frac{d}{K}\right]_c}\left(\mathbf{\mathbf{F_d}}\left({\bf {\bf x}}\circ S_{\omega-pK}\overline{{\bf {\bf x}}}\right)\right)_{\ell-qL}\left(\mathbf{F_{d}}\left({\bf {\bf z}}\circ S_{\omega-p K}\overline{{\bf {\bf z}}}\right)\right)_{qL-\ell}+\tilde{E}_{\ell,\omega}. \label{eq: doublealiasingspace}
\end{align}
In Sections \ref{sec: wddas1} and \ref{sec: wddas2}, we will be able to use  \eqref{eq: doublealiasingfourier} and \eqref{eq: doublealiasingspace} to recover a portion of the Fourier autocorrelation matrix $\widehat{\mathbf{x}}\widehat{\mathbf{x}}^*$.  (Note that \cite{Perlmutter2020} uses a different normalization of the discrete Fourier transform and consequently \eqref{eq: doublealiasingfourier} and \eqref{eq: doublealiasingspace} have different powers of $d$ than the corresponding equations there.)

\subsection{Wigner Deconvolution Under Assumption \ref{as: trigpoly}\label{sec: wddas1}}

In this subsection, we will assume our mask $\tilde{m}(x)$ satisfies Assumption \ref{as: trigpoly}, i.e., that it is a trigonometric polynomial with at most $\rho$ nonzero coefficients for some $\rho\leq r-1$. We also assume that $K=d,$ that $L$ divides $d$, and that $L=\rho+\kappa$ for some $2\leq \kappa\leq \rho$. 

Since $K=d,$ equation \eqref{eq: doublealiasingfourier} simplifies to
\begin{equation*}
\tilde{T}_{\ell,\omega}=4\pi^2d
\sum_{q\in\left[\frac{d}{L}\right]_c}\left(\mathbf{\mathbf{F_d}}\left({\bf {\bf \widehat{x}}}\circ S_{q L-\ell}\overline{{\bf {\bf \widehat{x}}}}\right)\right)_{\omega}\left(\mathbf{F_{d}}\left({\bf \widehat{{\bf z}}}\circ S_{\ell-q L}\overline{{\bf \widehat{{\bf z}}}}\right)\right)_{\omega}+\tilde{E}_{\ell,\omega}.
\end{equation*}
By construction, $\text{supp}(\widehat{\mathbf{z}})\subseteq[\rho+1]_c$.
Therefore, if $1-\kappa\leq \ell\leq \kappa-1$, we may use the same reasoning as in the proof of Lemma 10  of \cite{Perlmutter2020},
to see
\begin{equation*}
    {\bf \widehat{{\bf z}}}\circ S_{\ell-q L}\overline{{\bf \widehat{{\bf z}}}}=0
\end{equation*}
except for when $q=0.$ Thus,
\begin{equation}\label{eq:discrete-meas3}
\tilde{T}_{\ell,\omega}=
4\pi^2d\left(\mathbf{F_d}\left({\bf {\bf \widehat{x}}}\circ S_{-\ell}\overline{{\bf {\bf \widehat{x}}}}\right)\right)_{\omega}\left(\mathbf{F_d}\left({\bf \widehat{{\bf z}}}\circ S_{\ell}\overline{{\bf \widehat{{\bf z}}}}\right)\right)_{\omega}+\tilde{E}_{\ell,\omega}\quad\text{for all }|\ell|\leq \kappa-1.
\end{equation}

In order use \eqref{eq:discrete-meas3} to solve for $\left(\mathbf{F_d}\left({\bf {\bf \widehat{x}}}\circ S_{-\ell}\overline{{\bf {\bf \widehat{x}}}}\right)\right)_{\omega}$, we must divide by $\left(\mathbf{F_d}\left({\bf \widehat{{\bf z}}}\circ S_{\ell}\overline{{\bf \widehat{{\bf z}}}}\right)\right)_{\omega}$.
This motivates us to introduce a mask-dependent constant defined by   \begin{equation}\label{eq: mu1}\mu_1\coloneqq 
\min_{|p|\leq\kappa-1,q\in\mathcal{D}}\bigg|\left({\bf F_d}\left(\widehat{{\bf z}}\circ S_p\overline{\widehat{\bf z}}\right)\right)_q\bigg|.
\end{equation}
 Proposition 
 \ref{prop:mu_condition} shows that it is relatively simple to construct a trigonometric polynomial $\tilde{m}(x)$ such that $\mu_1$ is strictly positive. For a proof, please see Appendix \ref{app: mask design}.

\begin{proposition}
\label{prop:mu_condition} 
Assume that $\tilde{m}$ satisfies Assumption \ref{as: trigpoly}. Further assume
\begin{equation}
\left|\widehat{m}\left(-\frac{\rho}{2}\right)\right|>2\rho\left|\widehat{m}\left(-\frac{\rho}{2}+1\right)\right|\label{eqn: a0big}
\end{equation}
and 
\begin{equation}
\left|\widehat{m}\left(-\frac{\rho}{2}+1\right)\right|\geq \left|\widehat{m}\left(-\frac{\rho}{2}+2\right)\right|\geq \ldots\geq \left|\widehat{m}\left(\frac{\rho}{2}\right)\right|>0.
\label{eqn: decreasingamplitudes}
\end{equation}
Then  the mask-dependent constant $\mu_1$ defined as in \eqref{eq: mu1}
satisfies
 \begin{equation*}\mu_1\geq\frac{1}{2d} \left|\widehat{m}\left(-\frac{\rho}{2}\right)\right|\left|\widehat{m}\left(-\frac{\rho}{2}+\kappa-1|\right)\right|>0.\end{equation*} 
\end{proposition}
 
For the rest of this section, we will assume that $\mu_1$ is non-zero. Therefore, we may make a change of variables $\ell\rightarrow-\ell$ in \eqref{eq:discrete-meas3}
to  see that 
\begin{align*}
\left(\mathbf{F_d}\left({\bf {\bf \widehat{x}}}\circ S_{\ell}\overline{{\bf {\bf \widehat{x}}}}\right)\right)_{\omega}&=\frac{1}{4\pi^2d}\bigg(\frac{\tilde{T}_{-\ell,\omega}-\tilde{E}_{-\ell,\omega}}{(\mathbf{F_d}(\widehat{\mathbf{z}}\circ S_{-\ell}\widehat{\mathbf{z}}))_\omega}\bigg)\\
&=\frac{1}{4\pi^2d}\bigg(\frac{\tilde{T}_{-\ell,\omega}}{(\mathbf{F_d}(\widehat{\mathbf{z}}\circ S_{-\ell}\overline{\widehat{\mathbf{z}}}))_\omega}\bigg)- \frac{1}{4\pi^2d}\bigg(\frac{\tilde{E}_{-\ell,\omega}}{(\mathbf{F_d}(\widehat{\mathbf{z}}\circ S_{-\ell}\overline{\widehat{\mathbf{z}}}))_\omega}\bigg)
\end{align*}
for all $1-\kappa\leq \ell\leq \kappa-1$. 
Writing the above equation in column form, we have 
\begin{equation*}
\mathbf{F_d}\left({\bf {\bf \widehat{x}}}\circ S_{\ell}\overline{{\bf {\bf \widehat{x}}}}\right)=\frac{1}{4\pi^2d}\bigg(\frac{\tilde{\mathbf{T}}^T_{-\ell}}{\mathbf{F_d}(\widehat{\mathbf{z}}\circ S_{-\ell}\overline{\widehat{\mathbf{z}}})}\bigg)- \frac{1}{4\pi^2d}\bigg(\frac{\tilde{\mathbf{E}}^T_{-\ell}}{\mathbf{F_d}(\widehat{\mathbf{z}}\circ S_{-\ell}\overline{\widehat{\mathbf{z}}})}\bigg)
\end{equation*}
and so
\begin{equation}\label{eqn: diagbands1}
{\bf {\bf \widehat{x}}}\circ S_{\ell}\overline{{\bf {\bf \widehat{x}}}}
=\frac{1}{4\pi^2d}\mathbf{F_d}^{-1}\bigg(\frac{\tilde{\mathbf{T}}^T_{-\ell}}{\mathbf{F_d}(\widehat{\mathbf{z}}\circ S_{-\ell}\overline{\widehat{\mathbf{z}}}))}\bigg)
- \frac{1}{4\pi^2d}\mathbf{F_d}^{-1}\bigg(\frac{\tilde{\mathbf{E}}^T_{-\ell}}{\mathbf{F_d}(\widehat{\mathbf{z}}\circ S_{-\ell}\overline{\widehat{\mathbf{z}}}}\bigg),
\end{equation}
where, as mentioned in Section \ref{sec: intro}, the division of vectors is defined componentwise and $\mathbf{M}_j$ denotes the $j$-th column of a matrix $\mathbf{M}$.

  Let  $T_\kappa : \mathbb{C}^{d \times d} \rightarrow \mathbb{C}^{d \times d}$ be the restriction  operator defined for $\mathbf{M}\in\mathbb{C}^{d\times d}$ by \begin{equation*} T_\kappa(\mathbf{M})_{ij}
=\begin{cases}
M_{i,j} &\text{if } |i-j| \leq \kappa-1,\\
0 &\text{otherwise}.
\end{cases}
  \end{equation*} 
  Then, we may rewrite \eqref{eqn: diagbands1} 
 in matrix form as 
\begin{equation}\label{eq: angsynch1xxstar}
T_\kappa(\widehat{\mathbf{x}}\widehat{\mathbf{x}}^*) = \mathbf{X} + \tilde{\mathbf{N}},
\end{equation}
where the matrices $\mathbf{X}=(X_{i,j})_{i,j\in\mathcal{D}}$ and $\tilde{\mathbf{N}}=(\tilde{N}_{i,j})_{i,j\in\mathcal{D}}$ have entries  defined by
\begin{equation}\label{eq: defX}
    X_{i,j} = \begin{cases}\frac{1}{4\pi^2d}\left(\mathbf{F_d}^{-1}\left(\frac{\tilde{\mathbf{T}}^T_{i-j}}{\mathbf{F_d}(\widehat{\mathbf{z}}\circ S_{i-j}\overline{\widehat{\mathbf{z}}})}\right)\right)_{i}&\text{if }  |i-j|\leq \kappa-1,\\
       0&\text{otherwise},
       \end{cases}
\end{equation}
and \begin{equation*}
    \tilde{N}_{i,j} = \begin{cases}\frac{-1}{4\pi^2d}\left(\mathbf{F_d}^{-1}\left(\frac{\tilde{\mathbf{E}}^T_{i-j}}{\mathbf{F_d}(\widehat{\mathbf{z}}\circ S_{i-j}\overline{\widehat{\mathbf{z}}})}\right)\right)_{i}&\text{if }  |i-j|\leq \kappa-1,\\
       0&\text{otherwise}.
       \end{cases}
\end{equation*}

For a $d\times d$ matrix,  $\mathbf{M}=(M_{i,j})_{i,j\in\mathcal{D}}$,
let $R(\mathbf{M})=(R(M)_{i,j})_{i\in\mathcal{D}, j\in[2\kappa-1]_c}$ be the $d\times (2\kappa-1)$ matrix with entries defined by
\begin{equation*}
    R(M)_{i,j}=M_{i,i+j}.
\end{equation*}
Note that the columns of $R(\mathbf{M})$ are the diagonal bands of $\mathbf{M}$ which are near the main diagonal, and that in particular, the middle column, column zero, is the main diagonal. 
Since $\tilde{\mathbf{N}}$ is a banded matrix whose nonzero terms are within $\kappa$ of the main diagonal,
we see
\begin{equation*}
    \|\tilde{\mathbf{N}}\|_F = \|R(\tilde{\mathbf{N}})\|_F.
\end{equation*}  
 Therefore, 
 since
 $\frac{1}{\sqrt{d}}\mathbf{F_d}^{-1}$ is unitary, we may bound the $\ell^2$-norm of the columns of  $R(\tilde{\mathbf{N}})$  by
 \begin{equation*}
     \|R(\tilde{\mathbf{N}})_{j}\|_2 = \bigg\|\frac{1}{4\pi^2d}\mathbf{F_d}^{-1}\bigg(\frac{\tilde{\mathbf{E}}^T_{-j}}{\mathbf{F_d}(\widehat{\mathbf{z}}\circ S_{-j}\overline{\widehat{\mathbf{z}}})}\bigg)\bigg\|_2 \leq \frac{1}{4\pi^2d^{1/2}}\bigg\|\frac{\tilde{\mathbf{E}}^T_{-j}}{\mathbf{F_d}(\widehat{\mathbf{z}}\circ S_{-j}\overline{\widehat{\mathbf{z}}})}\bigg\|_2\leq \frac{1}{4\pi^2d^{1/2}\mu_1}\|\tilde{\mathbf{E}}^T_{-j}\|_2,
 \end{equation*}
 where $\mu_1$ is the mask-dependent constant defined in \eqref{eq: mu1}. 
  Therefore, by \eqref{eq: initial Etildebound} with $K=d$, we have
 \begin{equation}\label{eq: bound on N}
 \|\tilde{\mathbf{N}}\|_F=\|R(\tilde{\mathbf{N}})\|_F\leq C\frac{1}{d^{1/2}\mu_1}\|\mathbf{\tilde{E}}\|_F\leq C_{f,m}\frac{1}{d^{1/2}\mu_1}\bigg(
\left(\frac{1}{s}\right)^{k-1} + \frac{1}{\sqrt{dL}}\|\mathbf{\bm{\eta}_{d,L}}\|_F\bigg).
\end{equation}

Let $H:\mathbb{C}^{d\times d}\rightarrow\mathbb{C}^{d\times d}$ be the Hermitianizing operator 
\begin{equation}\label{eq: H}
    H(\mathbf{M})=\frac{\mathbf{M}+\mathbf{M}^*}{2}.
\end{equation}
 Since $T_\kappa( \mathbf{x}\mathbf{x}^*)$ is Hermitian, applying  $H$ to both sides of \eqref{eq: angsynch1xxstar}
 yields
\begin{equation}\label{eq: angsynch1}
    T_\kappa(\widehat{\mathbf{x}}\widehat{\mathbf{x}}^*)=
    \mathbf{A}+\mathbf{N},
\end{equation}
where 
\begin{equation} \label{eq: defA}
    \mathbf{A}\coloneqq H(\mathbf{X}) \text{ and } \mathbf{N}\coloneqq H({\bf \tilde{N}}).
\end{equation}
We note that by \eqref{eq: bound on N} and the triangle inequality, we have
\begin{equation}\label{eq: second bound on N}
 \|\mathbf{N}\|_F\leq\|\tilde{\mathbf{N}}\|_F\leq  C_{f,m}\frac{1}{d^{1/2}\mu_1}\bigg(
\left(\frac{1}{s}\right)^{k-1} + \frac{1}{\sqrt{dL}}\|\mathbf{\bm{\eta}_{d,L}}\|_F\bigg).
\end{equation}

\subsection{Wigner Deconvolution Under Assumption \ref{as: compact} \label{sec: wddas2}}

In this subsection, we assume $f(x)$ and $\tilde{m}(x)$ satisfy Assumption \ref{as: compact}, i.e., that $\text{supp}(f)\subseteq(-a,a)$ and $\text{supp}(\tilde{m})\subseteq(-b,b)$ with $a+b<\pi$. Note that, by construction,  this implies that the vector $\mathbf{z}$ defined in \eqref{eq:def-x} satisfies $\text{supp}(\mathbf{z})\subseteq [\delta+1]_c$, where $\delta=\lfloor\frac{bd}{\pi}\rfloor.$ We also assume that $L=d,$ that $K$ divides $d$ and that $K=\delta+\kappa$ for some $2\leq \kappa\leq \delta$. Furthermore, we let $s<2\kappa-1.$

Since $L=d,$ equation \eqref{eq: doublealiasingspace} simplifies to
\begin{equation*}
\tilde{T}_{\ell,\omega}= \frac{4\pi^2}{d}
\sum_{p\in\left[\frac{d}{K}\right]_c}\left(\mathbf{\mathbf{F_d}}\left({\bf {\bf x}}\circ S_{\omega-pK}\overline{{\bf {\bf x}}}\right)\right)_{\ell}\left(\mathbf{F_{d}}\left({\bf {\bf z}}\circ S_{\omega-p K}\overline{{\bf {\bf z}}}\right)\right)_{-\ell}+\tilde{E}_{\ell,\omega}.
\end{equation*}
Furthermore, if  $|\omega|\leq \kappa-1,$ then by the same reasoning as in Lemma 11 and Remark 1 of \cite{Perlmutter2020}, all terms in the above sum are zero except for the term corresponding to $p=0.$
 Therefore,
\begin{equation}
\tilde{T}_{\ell,\omega}=
\frac{4\pi^2}{d}
\left(\mathbf{\mathbf{F_d}}\left({\bf {\bf x}}\circ S_{\omega}\overline{{\bf {\bf x}}}\right)\right)_{\ell}\left(\mathbf{F_{d}}\left({\bf {\bf z}}\circ S_{\omega}\overline{{\bf {\bf z}}}\right)\right)_{-\ell}+\tilde{E}_{\ell,\omega}\quad\text{for all }|\omega|\leq \kappa-1.\label{eq:FYF_single_aliasing}
\end{equation}

The following lemma is a restatement of Lemma 3 of \cite{Perlmutter2020}, although we note that our result appears slightly different due to the fact that we use a different normalization of the discrete Fourier transform.
\begin{lemma}\label{lem: switchbackandforth} For all $\ell$ and $\omega,$ we have
\begin{equation*}
    \left(\mathbf{\mathbf{F_d}}\left({\bf {\bf x}}\circ S_{\omega}\overline{{\bf {\bf x}}}\right)\right)_{\ell} = d\mathbbm{e}^{2\pi \mathbbm{i} \omega\ell/d}
\left(\mathbf{\mathbf{F_d}}\left({\bf {\widehat{\bf x}}}\circ S_{-\ell}\overline{{\bf {\widehat{\bf x}}}}\right)\right)_{\omega}.
\end{equation*}
\end{lemma}

Applying Lemma \ref{lem: switchbackandforth} to \eqref{eq:FYF_single_aliasing}, we see that
\begin{equation}
\tilde{T}_{\ell,\omega}=
4\pi^2d
\left(\mathbf{\mathbf{F_d}}\left({\bf {\widehat{\bf x}}}\circ S_{-\ell}\overline{{\bf {\widehat{\bf x}}}}\right)\right)_{\omega}\left(\mathbf{F_{d}}\left({\bf \widehat{{\bf z}}}\circ S_{\ell}\overline{{\bf \widehat{{\bf z}}}}\right)\right)_{\omega}+\tilde{E}_{\ell,\omega}\label{eq:FYF_single_aliasing-modified}
\end{equation}
for all $|\omega|\leq \kappa-1$.
In order to solve for $\left(\mathbf{\mathbf{F_d}}\left({\bf {\widehat{\bf x}}}\circ S_{-\ell}\overline{{\bf {\widehat{\bf x}}}}\right)\right)_{\omega},$ we need to divide by $\left(\mathbf{F_{d}}\left({\bf \widehat{{\bf z}}}\circ S_{\ell}\overline{{\bf \widehat{{\bf z}}}}\right)\right)_{\omega}$. This motivates us to introduce a second mask-dependent constant given by 
\begin{equation}\label{eq: mu2}\mu_2\coloneqq \min_{\omega\in[2\kappa-1]_c,\ell\in[2s-1]_c} \bigg|
    \left(\mathbf{F_d}\left(\widehat{\mathbf{z}}\circ S_{\ell}\overline{\widehat{\mathbf{z}}}\right)\right)_\omega\bigg|.
\end{equation}
Proposition 
 \ref{prop: mu2} shows that, for any given $d$,  it is relatively simple to construct a mask $\tilde{m}(x)$ such that $\mu_2$ is strictly positive. For a proof please see Appendix \ref{app: mask design}.
 
\begin{proposition}\label{prop: mu2}
Assume that $\tilde{m}(x)$ satisfies Assumption \ref{as: compact}. Let $\mathbf{z}=(z_p)_{p\in\mathcal{D}}$ be the vector defined as in \eqref{eq:def-x} by $z_p=m\left(\frac{2\pi p}{d}\right)$, and let $\delta = \lfloor\frac{b}{\pi}d\rfloor.$ Let $\tilde{\delta}\leq \delta+1$ and assume that $\text{supp}(\mathbf{z})=\{n, n+1, \ldots, n+\tilde\delta-1\}$
for some $\kappa\leq\tilde{\delta}\leq \delta+1.$
Further assume that 
\begin{equation}\label{eqn: big z n}
    |z_n| >2\tilde{\delta} |z_{n+1}|
\end{equation}
and that 
\begin{equation}\label{eqn: zs decrease}
    |z_{n+1}|\geq |z_{n+2}|\geq\ldots |z_{n+\tilde\delta-1}|>0.
\end{equation}
Then the mask-dependent constant $\mu_2$ defined in \eqref{eq: mu2} satisfies \begin{equation*}\mu_2
 \geq \frac{1}{2d^2}|z_n||z_{n+\kappa-1}|>0.
\end{equation*}
\end{proposition}
\begin{remark}
Given any vector $\mathbf{z}=(z_p)_{p\in\mathcal{D}},$ one may construct, e.g., through spline interpolation, a function $\tilde{m}(x)$ such that $\tilde{m}\left(\frac{2\pi p}{d}\right)=z_p$ for all $p\in\mathcal{D}.$ 
\end{remark}

For the rest of this section, we will assume that $\mu_2$ is not equal to zero. 
Therefore, we may make a change of variables $\ell\rightarrow-\ell$ in \eqref{eq:FYF_single_aliasing-modified}
to  see that 
\begin{align*}
\left(\mathbf{F_d}\left({\bf {\bf \widehat{x}}}\circ S_{\ell}\overline{{\bf {\bf \widehat{x}}}}\right)\right)_{\omega}&=\frac{1}{4\pi^2d}\bigg(\frac{\tilde{T}_{-\ell,\omega}-\tilde{E}_{-\ell,\omega}}{(\mathbf{F_d}(\widehat{\mathbf{z}}\circ S_{-\ell}\overline{\widehat{\mathbf{z}}}))_\omega}\bigg)\nonumber\\
&=\frac{1}{4\pi^2d}\bigg(\frac{\tilde{T}_{-\ell,\omega}}{(\mathbf{F_d}(\widehat{\mathbf{z}}\circ S_{-\ell}\overline{\widehat{\mathbf{z}}}))_\omega}\bigg)- \frac{1}{4\pi^2d}\bigg(\frac{\tilde{E}_{-\ell,\omega}}{(\mathbf{F_d}(\widehat{\mathbf{z}}\circ S_{-\ell}\overline{\widehat{\mathbf{z}}}))_\omega}\bigg).
\end{align*}
Now, recall that $s\leq 2\kappa-1$, and let $\mathbf{B}\coloneqq(B_{\omega,\ell}),  \mathbf{C}\coloneqq(C_{\omega,\ell})$, and $\mathbf{D}\coloneqq(D_{\omega,\ell})$ be $(2\kappa-1)\times (2s-1)$ matrices with entries defined by 
\begin{equation}\label{eq: defU}
B_{\omega,\ell}=\left(\mathbf{F_d}\left({\bf {\bf \widehat{x}}}\circ S_{\ell}\overline{{\bf {\bf \widehat{x}}}}\right)\right)_{\omega},\quad
C_{\omega,\ell} =\frac{1}{4\pi^2d}\bigg(\frac{\tilde{T}_{-\ell,\omega}}{(\mathbf{F_d}(\widehat{\mathbf{z}}\circ S_{-\ell}\overline{\widehat{\mathbf{z}}}))_\omega}\bigg),\quad 
D_{\omega,\ell}
= \frac{-1}{4\pi^2d}\bigg(\frac{\tilde{E}_{-\ell,\omega}}{(\mathbf{F_d}(\widehat{\mathbf{z}}\circ S_{-\ell}\overline{\widehat{\mathbf{z}}}))_\omega}\bigg)
\end{equation}
for $\omega\in[2\kappa-1]_c$ and $\ell\in[2s-1]_c$
so that 
\begin{equation*}
    \mathbf{B}=\mathbf{C}+\mathbf{D}.
\end{equation*}
Note that 
\begin{equation}\label{eq: Vbound}
    \|\mathbf{D}\|_F\leq \frac{1}{4\pi^2d\mu_2}\|\tilde{\mathbf{E}}\|_F, 
\end{equation}
where $\mu_2$ is the mask-dependent constant defined in \eqref{eq: mu2}.

Next observe that we may factor $\mathbf{B}=\mathbf{WV}$, where $\mathbf{V}\coloneqq(V_{j,k})_{j\in\mathcal{S},k\in[2s-1]_c}$ is the $s\times(2s-1)$ matrix with entries defined by $V_{j,k}=(\widehat{\mathbf{x}}\circ S_k\overline{\widehat{\mathbf{x}}})_j$
and
$\mathbf{W}\coloneqq(W_{j,k})_{j\in[2\kappa-1]_c,k\in\mathcal{S}}$ is the $(2\kappa-1)\times s$ partial Fourier matrix with entries $W_{j,k}=(\mathbf{F_d})_{j,k}.$ 
Since $s\leq 2\kappa-1$, we may  let $\mathbf{W}^\dagger\coloneqq(\mathbf{W}^*\mathbf{W})^{-1}\mathbf{W}^*$ be the pseudoinverse of $\mathbf{W}$ and see
\begin{equation*}
    \mathbf{V}=\mathbf{W}^\dagger\mathbf{C}+ \mathbf{W}^\dagger\mathbf{D}.
\end{equation*}
Now, let $\Lambda:\mathbb{C}^{s\times(2s-1)}\rightarrow\mathbb{C}^{d\times d}$ be the lifting operator defined by 
\begin{equation*}
    (\Lambda(M))_{i,j}=M_{i,j-i}.
\end{equation*}
Note that the columns of $\mathbf{M}$ are  diagonal bands  of $\Lambda(M)$ with the middle column on the main diagonal. 
By construction, we have $T_{2s-1} (\widehat{\mathbf{x}}\widehat{\mathbf{x}}^*)=\Lambda(\mathbf{V})$. Therefore, since $T_{2s-1}(\widehat{\mathbf{x}}\widehat{\mathbf{x}}^*)$ is Hermitian,  we have
\begin{equation*}
    T_{2s-1}(\widehat{\mathbf{x}}\widehat{\mathbf{x}}^*)=H(\Lambda(\mathbf{V})),
\end{equation*}
where $H$ is the Hermitianizing operator introduced in \eqref{eq: H}. Therefore,
\begin{equation}\label{eq: defXcompact} T_{2s-1} (\widehat{\mathbf{x}}\widehat{\mathbf{x}}^*)=
\mathbf{A}+\mathbf{N},
\end{equation}
where
 \begin{equation}\label{eq: defAcompact}
     \mathbf{A}\coloneqq H(\Lambda(\mathbf{W}^\dagger\mathbf{C}))\quad\text{and}\quad \mathbf{N}\coloneqq H(\Lambda(\mathbf{W}^\dagger\mathbf{D})).
 \end{equation}
Since $H$ is contractive, \eqref{eq: Vbound} implies
\begin{equation*}
    \|\mathbf{N}\|_F\leq\|\Lambda(\mathbf{W}^\dagger\mathbf{D})\|= \|\mathbf{W}^\dagger\mathbf{D}\|_F\leq \frac{1}{\sigma_{\min}(\mathbf{W})}\|\mathbf{D}\|_F\leq  \frac{1}{4\pi^2d\mu_2\sigma_{\min}(\mathbf{W})}\|\tilde{\mathbf{E}}\|_F,
\end{equation*}
where $\sigma_{\min}(\mathbf{W})$ is the smallest singular value of $\mathbf{W}$.
Combining this with \eqref{eq: initial Etildeboundcompact} yields
\begin{equation}\label{eq: Ntildefinalbound}
    \|\mathbf{N}\|_F\leq  C_{f,m}\frac{1}{d\mu_2\sigma_{\min}(\mathbf{W})} \bigg(\left(\frac{1}{s}\right)^{k-1}+\left(\frac{1}{r}\right)^{k-1} + \frac{1}{\sqrt{Kd}}\|\mathbf{\bm{\eta}_{K,d}}\|_F\bigg).
\end{equation}

\section{Convergence Guarantees of Algorithms \ref{Big Algorithm} and \ref{Big Algorithm Compact}} \label{sec: convergence theorems}

In this section, we will provide  convergence guarantees  for  Algorithms \ref{Big Algorithm} and \ref{Big Algorithm Compact}. Specifically, we will prove Theorem \ref{thm: angular synch} which guarantees that we can reconstruct $f(x)$ from a noisy Fourier autocorrelation matrix. Corollaries \ref{cor: trigpoly} and \ref{cor: compact}, which guarantee the convergence of our algorithms, will then follow immediately from \eqref{eq: angsynch1}, \eqref{eq: second bound on N}, \eqref{eq: defXcompact}, and \eqref{eq: Ntildefinalbound}, which are   
proved in Section \ref{sec: WDD}.

For the rest of this section, we will assume that there exists $1\leq \gamma\leq d$ such that 
\begin{equation} \label{eq: angsynchsetup}T_\gamma (\widehat{\mathbf{x}}\widehat{\mathbf{x}}^*)=\mathbf{A}+\mathbf{N}.
\end{equation}
Here, $\mathbf{A}=(A_{i,j})_{i,j\in\mathcal{D}}$ is a known approximation of the partial Fourier autocorrelation matrix $T_\gamma (\widehat{\mathbf{x}}\widehat{\mathbf{x}}^*)$ and $\mathbf{N}\in\mathbb{C}^{d\times d}$ is an arbitrary noise matrix. We 
note that, under Assumption \ref{as: trigpoly},  equation \eqref{eq: angsynch1} shows that \eqref{eq: angsynchsetup} holds with $\gamma=\kappa$. Similarly, under Assumption \ref{as: compact}, equation \eqref{eq: defXcompact} shows that \eqref{eq: angsynchsetup} holds with $\gamma=2s-1.$ We also remark that  \eqref{eq: second bound on N} and \eqref{eq: Ntildefinalbound} provide bounds on $\|\mathbf{N}\|_F$ in these cases. 
We will also assume for the remainder of this section that there exists $\beta<\gamma/2$ such that  $f$ belongs to the class of functions with $\beta$ Fourier decay introduced in Definition \ref{def: Fourier decay}.

By construction,  the discrete Fourier transform of the vector $\mathbf{x}$ defined in \eqref{eq:def-x} satisfies
\begin{align*}
\widehat{x}_n=
\widehat{f}(n)  \text{ for all }n\in \mathcal{S},
\end{align*}
and so the square magnitudes of the Fourier coefficients of $f$ lie on the main diagonal of the matrix $T_\gamma(\widehat{\mathbf{x}}\widehat{\mathbf{x}}^*).$ Therefore, we view $a_n\coloneqq\sqrt{|A_{n,n}|}$  as an approximation of $|\widehat{x}_n|$. More specifically, Lemma 3 of \cite{IVW16} 
 shows that
\begin{equation}\label{la:mag-error}
\Big|a_n-|\widehat{f}(n)|\Big|^2\leq 3 \|\mathbf{N}\|_\infty.
\end{equation} 
 For each $n\in\mathcal{S}$, the greedy entry selection algorithm, Algorithm \ref{alg: greedy entry selection}, outputs a sequence $\{n_\ell\}_{\ell=0}^b$, where $n_0=\argmax_{n\in\mathcal{S}}a_n$ and  $n_b=n$. 
Given that sequence, we define 
\begin{equation}\label{eq: definealpha}
\alpha_n\coloneqq \sum_{l=0}^{b-1} \arg\left(A_{n_{\ell+1},n_{\ell}}\right).
\end{equation}
To understand this definition, we let
\begin{equation}\label{eq: deftau}
    \theta_0\coloneqq\arg(\widehat{f}(n_0))
\quad\text{and}\quad
\tau_n\coloneqq\sum_{l=0}^{b-1} \arg\left((\widehat{\mathbf{x}}\widehat{\mathbf{x}}^*)_{n_{\ell+1},n_{\ell}}\right).
\end{equation}
By construction, $\tau_n=
    \arg\big(\widehat{f}(n)\big)-\theta_0$. Therefore
\begin{equation*}
\mathbbm{e}^{-\mathbbm{i}\theta_0}\widehat{f}(n)=|\widehat{f}(n)|\mathbbm{e}^{\mathbbm{i}\tau_n}
\end{equation*}
for all $n\in\mathcal{S}$. (Note that $n_0$ does not depend on $n$.)
Since $\mathbf{A}$ is a noisy approximation of (a portion of) $\widehat{\mathbf{x}}\widehat{\mathbf{x}}^*$, we intuitively view $\alpha_n$ as a noisy approximation of $\tau_n$ (up to a phase shift $\theta_0$). Lemma \ref{la:phase-error} will show  that this intuition  is correct when $|\widehat{f}(n)|$ is sufficiently large.
Therefore, in light of \eqref{la:mag-error}, we  define a trigonometric polynomial, $f_e(x)$, which estimates $f(x)$
by
\begin{equation}\label{eqn: fe}
f_e(x)\coloneqq    \sum_{n\in \mathcal{S}} a_n\mathbbm{e}^{\mathbbm{i}\alpha_n} \mathbbm{e}^{\mathbbm{i}nx}.
\end{equation} 
The following theorem shows that $f_e(x)$ is a good approximation of $f(x)$.
\begin{theorem} \label{thm: angular synch}Assume that $f(x)$ has $\beta$ Fourier decay for some $\beta<\gamma/2.$
 For $n\in \mathcal{S}$, let $\alpha_n$ be defined as in \eqref{eq: definealpha}, let $a_n=\sqrt{A_{n,n}}$, and  let  $f_e(x)$ be the trigonometric polynomial defined as in \eqref{eqn: fe}. 
Then,
 
\begin{equation*}
\min_{\theta\in[0,2\pi]} \|\mathbbm{e}^{\mathbbm{i}\theta} f - f_e\|_{L^2([-\pi,\pi])}^2 \leq C\,s\left(\frac{d}{\gamma}\right)^2\|\mathbf{N}\|_\infty
+ C_f\left(\frac{1}{s}\right)^{2k-2}.
\end{equation*}
\end{theorem}

Before proving Theorem \ref{thm: angular synch}, we recall that $\gamma=\kappa$ under Assumption \ref{as: trigpoly} and $\gamma=2s-1$ under Assumption \ref{as: compact}. Therefore, 
\eqref{eq: second bound on N}, \eqref{eq: Ntildefinalbound}, and the fact that $\|\mathbf{N}\|_\infty\leq\|\mathbf{N}\|_F,$
immediately lead to the following corollaries. 
\begin{corollary}[Convergence Guarantees for Algorithm \ref{Big Algorithm}] \label{cor: trigpoly}
Let $s+r<d,$  let $K=d,$  and let $L$ divide $d$.
Assume that $f(x)$ and $\tilde{m}(x)$ satisfy Assumption \ref{as: trigpoly}, that $\rho\leq r-1$, and  that $L=\rho+\kappa$ for some $2\leq \kappa\leq \rho.$ Then the trigonometric polynomial $f_e(x)$ output by Algorithm \ref{Big Algorithm} satisfies
\begin{align*}
&\min_{\theta\in[0,2\pi]} \|\mathbbm{e}^{\mathbbm{i}\theta} f - f_e\|_{L^2([-\pi,\pi])}^2 
\leq   C_{f,m}\bigg(
\frac{sd^{3/2}}{\kappa^2\mu_1}\bigg(
\left(\frac{1}{s}\right)^{k-1} + \frac{1}{\sqrt{dL}}\|\mathbf{\bm{\eta}_{d,L}}\|_F\bigg)
+ \left(\frac{1}{s}\right)^{2k-2}\bigg),
\end{align*}
where  $\mu_1$ is the mask-dependent constant defined in \eqref{eq: mu1}. Moreover, if $s>d/2,$ then 
\begin{align*}
\min_{\theta\in[0,2\pi]} \|\mathbbm{e}^{\mathbbm{i}\theta} f - f_e\|_{L^2([-\pi,\pi])}^2 
\leq   C_{f,m}\bigg(
\frac{1}{\kappa^2\mu_1}
\left(\frac{1}{d}\right)^{k-7/2} + \frac{d^2}{\kappa^2L^{1/2}\mu_1}\|\mathbf{\bm{\eta}_{d,L}}\|_F
+ \left(\frac{1}{d}\right)^{2k-2}\bigg).
\end{align*}
\end{corollary}

\begin{corollary}[Convergence Guarantees for Algorithm \ref{Big Algorithm Compact}]\label{cor: compact}
Let $s+r<d,$ let $L=d,$  and let $K$ divide $d$.
Assume $f(x)$ and $\tilde{m}(x)$ satisfy Assumption \ref{as: compact} and let $\delta=\lfloor\frac{bd}{\pi}\rfloor.$ Further, assume that $K=\delta+\kappa$ for some $2\leq \kappa\leq \delta$ and that $s<2\kappa-1.$ Then the trigonometric polynomial $f_e(x)$ output by Algorithm \ref{Big Algorithm Compact}, satisfies

\begin{align*}
\min_{\theta\in[0,2\pi]} &\|\mathbbm{e}^{\mathbbm{i}\theta} f - f_e\|_{L^2([-\pi,\pi])}^2\\
&\leq C_{f,m}\bigg(\frac{d}{s\mu_2\sigma_{\min}(\mathbf{W})} \bigg(\left(\frac{1}{s}\right)^{k-1}+\left(\frac{1}{r}\right)^{k-1} + \frac{1}{\sqrt{Kd}}\|\mathbf{\bm{\eta}_{K,d}}\|_F\bigg)
+ \left(\frac{1}{s}\right)^{2k-2}\bigg),
\end{align*}
where  $\mu_2$ is the mask-dependent constant defined in \eqref{eq: mu2}.
Moreover, if $s,r\geq \frac{db}{2\pi}$, then 

\begin{align*}
\min_{\theta\in[0,2\pi]} &\|\mathbbm{e}^{\mathbbm{i}\theta} f - f_e\|_{L^2([-\pi,\pi])}^2\\
&\leq C_{f,m}\bigg(\frac{1}{\mu_2\sigma_{\min}(\mathbf{W})b^{k-1}d^k}  + \frac{d^{1/2}}{K^{1/2}\mu_2\sigma_{\min}(\mathbf{W})}\|\mathbf{\bm{\eta}_{K,d}}\|_F
+ \left(\frac{1}{bd}\right)^{2k-2}\bigg).
\end{align*}
\end{corollary}

In order to prove Theorem \ref{thm: angular synch}, we need the following lemma which provides us with an estimate of $\|\mathbbm{e}^{-\mathbbm{i}\theta_0}P_\mathcal{S}f-f_e\|_{L^2([-\pi,\pi])}$ as well as the uniform convergence of Fourier series. 

\begin{lemma}\label{lem: truncated fourier error}Assume that $f(x)$ has $\beta$ Fourier decay for some $\beta<\gamma/2.$
 For $n\in \mathcal{S}$, let $\alpha_n$ be defined as in \eqref{eq: definealpha}, let $a_n=\sqrt{A_{n,n}}$, and  let  $f_e(x)$ be the trigonometric polynomial defined as in \eqref{eqn: fe} by \newline 
 $f_e(x)=    \sum_{n\in \mathcal{S}} a_n\mathbbm{e}^{\mathbbm{i}\alpha_n} \mathbbm{e}^{\mathbbm{i}nx}
$. Then, 
\begin{align*}
\Big\|\mathbbm{e}^{-\mathbbm{i}\theta_0}P_\mathcal{S}f-f_e\Big\|^2_{L^2([-\pi,\pi])}&\leq C\,s\left(\frac{d}{\gamma}\right)^2\|\mathbf{N}\|_\infty.
\end{align*}
\end{lemma}

In order to prove Lemma \ref{lem: truncated fourier error}, we need the following lemma, which is a modification of  \cite[Lemma 4]{IVW16}. It shows that $\alpha_n$ is a good approximation of $\tau_n$ for all $n$ such that $|\widehat{f}(n)|$ is sufficiently large. For a proof, please see Appendix \ref{app: pf La phase-erro}.

\begin{lemma}\label{la:phase-error} Suppose that $f$ has $\beta$ Fourier decay for some $\beta\leq\gamma/2$, and let $L_f$ be the set of indices corresponding to large Fourier coefficients
defined by 
\begin{equation}\label{eq: defLf}
    L_f\coloneqq \{n\in \mathcal{S}: |\widehat{f}(n)|^2\geq 48 \|\mathbf{N}\|_\infty\}.
\end{equation} 
Let $n\in L_f$, and let $\tau_n$ and $\alpha_n$ be as in \eqref{eq: definealpha}
 and \eqref{eq: deftau}. Then
\begin{equation*}
|\mathbbm{e}^{\mathbbm{i}\tau_n}-\mathbbm{e}^{\mathbbm{i}\alpha_n}|\leq 
\frac{4\pi d}{\gamma}\frac{\|\mathbf{N}\|_\infty}{|\widehat{f}(n)|^2}.
\end{equation*}
\end{lemma}

\begin{proof}[The Proof of Lemma \ref{lem: truncated fourier error}]

Recall that $\widehat{\mathbf{x}}_n = \widehat{f}(n)$  for all $n\in\mathcal{S}$,
and let $\mathbf{\widehat{x}}_{|\mathcal{S}}$ be a vector of length $s$ obtained by restricting $\mathbf{\widehat{x}}$ to indices in $\mathcal{S}.$
Define vectors $\mathbf{u}=(u_n)_{n\in\mathcal{S}}$ and $\mathbf{v}=(v_n)_{n\in\mathcal{S}}$
by 
\begin{equation*}
    u_n=a_n\mathbbm{e}^{\mathbbm{i}\alpha_n} \quad\text{and}\quad v_n=|\widehat{f}(n)|\mathbbm{e}^{\mathbbm{i}\alpha_n}.
\end{equation*}
 By Parsevals identity, we see
\begin{align*}
\Big\|\mathbbm{e}^{-\mathbbm{i}\theta_0}P_\mathcal{S}f(x)-\sum_{n\in \mathcal{S}}a_n\mathbbm{e}^{\mathbbm{i}\alpha_n}\mathbbm{e}^{\mathbbm{i}n x}\Big\|_{L^2([-\pi,\pi])}
&=\Big\|\mathbbm{e}^{-\mathbbm{i}\theta_0}\sum_{n\in \mathcal{S}}\widehat{f}(n)\mathbbm{e}^{\mathbbm{i}n x}-\sum_{n\in \mathcal{S}}u_n\mathbbm{e}^{\mathbbm{i}n x}\Big\|_{L^2([-\pi,\pi])}\\
&\leq 
\sqrt{2\pi}\left\|\mathbbm{e}^{-\mathbbm{i}\theta_0}{\widehat{\bf{x}}}_{|\mathcal{S}}-{\bf{u}}\right\|_{\ell_2}\\
&\leq 
\sqrt{2\pi}\left\|\mathbbm{e}^{-\mathbbm{i}\theta_0}{\widehat{\bf{x}}_{|\mathcal{S}}}-{\bf{v}}\right\|_{\ell_2}+
\sqrt{2\pi}\|{\bf{u}}-{\bf{v}}\|_{\ell_2}\\
&\eqqcolon I_1 + I_2.
\end{align*}
To estimate $I_2$, we recall  \eqref{la:mag-error} and note
\begin{align}\label{eq:I2}
I_2^2&=2\pi\sum_{n\in \mathcal{S}}|u_n-v_n|^2=2\pi{\sum_{n\in \mathcal{S}}\Big|a_n\mathbbm{e}^{\mathbbm{i}\alpha_n}-|\widehat{f}(n)|\mathbbm{e}^{\mathbbm{i}\alpha_n}\Big|^2}=2\pi{\sum_{n\in \mathcal{S}}\Big|a_n-|\widehat{x}_n|\Big|^2}\le 6\pi {s\|\mathbf{N}\|_\infty}.
\end{align}
Using  Lemma~\ref{la:phase-error} and the fact that $|\mathbbm{e}^{\mathbbm{i}\tau_n}-\mathbbm{e}^{\mathbbm{i}\alpha_n}|\leq 2$, we have 
\begin{align*}\nonumber
I_1^2&=2\pi\sum_{n\in \mathcal{S}}|\widehat{f}(n)|^2|\mathbbm{e}^{\mathbbm{i}\tau_n}-\mathbbm{e}^{\mathbbm{i}\alpha_n}|^2\\
&\nonumber\leq 
C\sum_{n\in \mathcal{S}\setminus L_{f}}|\widehat{f}(n)|^2
+C \sum_{n\in L_{f}} \left(\frac{d}{\gamma}\right)^{2}\,\|\mathbf{N}\|_\infty^2\,|\widehat{f}(n)|^{-2}\\
&\nonumber\leq
C\,s\,\|\mathbf{N}\|_\infty
+
C\,\sum_{n\in L_{f}}\left(\frac{d}{\gamma}\right)^{2}\|\mathbf{N}\|_\infty\\
&\leq 
 C\,s\left(\frac{d}{\gamma}\right)^2\|\mathbf{N}\|_\infty,
\end{align*}
where $L_f$ is the set of indices corresponding to large Fourier coefficients introduced in \eqref{eq: defLf}.
Combining this with \eqref{eq:I2} yields 
\begin{align*}
\Big\|\mathbbm{e}^{-\mathbbm{i}\theta_0}P_\mathcal{S}f(x)-\sum_{n\in \mathcal{S}}a_{n}\mathbbm{e}^{\mathbbm{i}n x}\mathbbm{e}^{\mathbbm{i}\alpha_n}\Big\|^2_{L^2([-\pi,\pi])}&\leq C\,s\left(\frac{d}{\gamma}\right)^2\|\mathbf{N}\|_\infty
\end{align*}
as desired.
\end{proof}

Theorem \ref{thm: angular synch} now follows readily via Lemma \ref{lem:Intbound} which estimates $\|f-P_\mathcal{S}f\|_{L^2([-\pi,\pi])}^2$.
\begin{proof}[The Proof of Theorem \ref{thm: angular synch}] Let $\theta_0=\arg(\widehat{f}(n_0)).$ Then we get 
\begin{align*}
\min_{\theta\in[0,2\pi]}\Big\|\mathbbm{e}^{\mathbbm{i}\theta} f(x) &- \sum_{n\in \mathcal{S}} a_n\mathbbm{e}^{\mathbbm{i}\alpha_n}\mathbbm{e}^{\mathbbm{i}nx}\Big\|_{L^2([-\pi,\pi])}\\
&\leq \min_{\theta\in[0,2\pi]}\bigg(\Big\|\mathbbm{e}^{\mathbbm{i}\theta}f(x)-\mathbbm{e}^{\mathbbm{i}\theta}P_\mathcal{S} f(x)\Big\|_{L^2([-\pi,\pi])} + \Big\|\mathbbm{e}^{\mathbbm{i}\theta} P_\mathcal{S}f(x)-\sum_{n\in \mathcal{S}} a_n\mathbbm{e}^{\mathbbm{i}\alpha_n} \mathbbm{e}^{\mathbbm{i}nx}\Big\|_{L^2([-\pi,\pi])}\bigg)\\
&\leq  \|f(x)-P_\mathcal{S} f(x)\|_{L^2([-\pi,\pi])} +  \Big\|\mathbbm{e}^{-\mathbbm{i}\theta_0}P_\mathcal{S}f(x)-\sum_{n\in \mathcal{S}} a_n\mathbbm{e}^{\mathbbm{i}\alpha_n} \mathbbm{e}^{\mathbbm{i}nx}\Big\|_{L^2([-\pi,\pi])}.
\end{align*}
By Lemma \ref{lem: truncated fourier error}, we know that
\begin{equation*}
\Big\|\mathbbm{e}^{-\mathbbm{i}\theta_0}P_\mathcal{S}f(x)-\sum_{n\in \mathcal{S}}a_n\mathbbm{e}^{\mathbbm{i}\alpha_n}\mathbbm{e}^{\mathbbm{i}n x}\Big\|_{L^2([-\pi,\pi])}^2 \leq C\,s\left(\frac{d}{\gamma}\right)^2\|\mathbf{N}\|_\infty.
\end{equation*}
Therefore, we conclude by applying Lemma \ref{lem:Intbound} to see
\begin{align*}
\|f-P_\mathcal{S} f\|_{L^2([-\pi,\pi])}^2&\leq
2\pi \|f-P_\mathcal{S} f\|_{L^\infty([-\pi,\pi])}^2
\leq C_f\left(\frac{1}{s}\right)^{2k-2}.
\end{align*}
\end{proof}
\textbf{}
\begin{algorithm}
\begin{raggedright}
\textbf{Inputs}
\par\end{raggedright}
\begin{enumerate}
\item \begin{raggedright}
Trigonometric polynomial mask $\tilde{m}$ satisfying Assumption \ref{as: trigpoly}.
\par\end{raggedright}
\item \begin{raggedright}
Matrix $\mathbf{Y}=(Y_{\omega,\ell})_{\omega\in\mathcal{D},\ell\in\mathcal{L}}$ of spectrogram measurements defined as in \eqref{eq:conti-meas2}.
\par\end{raggedright}

\end{enumerate}
\begin{raggedright}
\textbf{Steps}
\par\end{raggedright}
\begin{enumerate}
\item Define vector 
$\mathbf{z}=(z_p)_{p\in\mathcal{D}}$ by 
$z_{p} = \tilde{m}\left(\frac{2\pi p}{d}\right).
$
\item Let $\kappa=L-\rho$, and for $1-\kappa\leq \ell \leq \kappa-1$, estimate  
\begin{equation*}
\mathbf{F_d}\left({\bf {\bf \widehat{x}}}\circ S_{\ell}\overline{{\bf {\bf \widehat{x}}}}\right)\approx\frac{1}{4\pi^2Ld^2}\left(\frac{(\mathbf{F_L}\mathbf{Y}^T\mathbf{F_d}^T)_{-\ell}}{\mathbf{F_d}(\widehat{\mathbf{z}}\circ S_{-\ell}\overline{\widehat{\mathbf{z}}})}\right). 
\end{equation*}

\item Invert the Fourier transforms above to recover estimates of the  
vectors ${\bf {\bf \widehat{{\bf x}}}}\circ S_{\ell}\overline{{\bf {\bf \widehat{{\bf x}}}}}$. 
\item Organize these vectors into a banded matrix 
$\mathbf{X}=(X_{i,j})_{i,j\in\mathcal{D}}$ described as in \eqref{eq: defX}.
\item Hermitianize $\mathbf{X}$ to obtain the matrix $\mathbf{A}=(A_{i,j})_{i,j\in\mathcal{D}}$  as described in \eqref{eq: defA}. 
\item Estimate $|\widehat{f}(n)|\approx a_n=\sqrt{|A_{n,n}|}$.
\item For $n\in\mathcal{S}$, choose $\{n_\ell\}_{\ell=0}^b$ according to Algorithm \ref{alg: greedy entry selection}. 
\item Approximate 
\begin{equation*}\arg\big(\widehat{f}(n)\big)\approx \alpha_n=\sum_{\ell=0}^{b-1} \arg\left(A_{n_{\ell+1},n_{\ell}}\right).
\end{equation*}
\end{enumerate} 
\begin{raggedright}
\textbf{Output}
\par\end{raggedright}
\begin{raggedright}
An approximation of $f$ given by 
\begin{equation*}
    f_e(x)=\sum_{n\in\mathcal{S}}a_n\mathbbm{e}^{\mathbbm{i}\alpha_n}\mathbbm{e}^{\mathbbm{i}nx}.
\end{equation*}
\par\end{raggedright}
\raggedright{}\textbf{\caption{\textbf{Signal Recovery with Trigonometric Polynomial Masks} \label{Big Algorithm}}
}
\end{algorithm}
\textbf{}
\begin{algorithm}
\begin{raggedright}
\textbf{Inputs}
\par\end{raggedright}
\begin{enumerate}
\item \begin{raggedright}
Compactly supported mask $\tilde{m}$ satisfying Assumption \ref{as: compact}.
\par\end{raggedright}
\item \begin{raggedright}
Matrix $\mathbf{Y}=(Y_{\omega,\ell})_{\omega\in\mathcal{K},\ell\in\mathcal{D}}$ of spectrogram measurements defined as in \eqref{eq:conti-meas2}.
\par\end{raggedright}
\end{enumerate}
\begin{raggedright}
\textbf{Steps}
\par\end{raggedright}
\begin{enumerate}
\item Define vector 
$\mathbf{z}=(z_p)_{p\in\mathcal{D}}$ by 
$z_{p} = \tilde{m}\left(\frac{2\pi p}{d}\right).
$
\item Let $\kappa=K-\delta$, and for $1-\kappa\leq \omega \leq \kappa-1, 1-s\leq \ell\leq s-1$ estimate   
\begin{equation*}
\mathbf{F_d}\left({\bf {\bf \widehat{x}}}\circ S_{\ell}\overline{{\bf {\bf \widehat{x}}}}\right)\approx\frac{1}{4\pi^2Kd^2}\left(\frac{(\mathbf{F_d}\mathbf{Y}^T\mathbf{F_K}^T)_{-\ell}}{(\mathbf{F_d}(\widehat{\mathbf{z}}\circ S_{-\ell}\overline{\widehat{\mathbf{z}}}))}\right). 
\end{equation*}

\item Form the matrix $\mathbf{C}$ according to \eqref{eq: defU}.

\item Compute  $\mathbf{V}=\mathbf{W}^\dag\mathbf{C}$, where $\mathbf{W}=((\mathbf{F_d})_{j,k})_{j\in[2\kappa-1]_c,k\in \mathcal{S}}$ is the $(2\kappa-1)\times s$ partial Fourier matrix.
 
\item Apply lifting operator $\Lambda$. 
\item Hermitianize $\Lambda(\mathbf{V})$ to obtain the matrix $\mathbf{A}=(A_{i,j})_{i,j\in\mathcal{D}}$  as described in \eqref{eq: defAcompact}. 
\item Estimate $|\widehat{f}(n)|\approx a_n=\sqrt{|A_{n,n}|}$.
\item For $n\in\mathcal{S}$, choose $\{n_\ell\}_{\ell=0}^b$ according to Algorithm \ref{alg: greedy entry selection}. 
\item Approximate 
\begin{equation*}\arg\big(\widehat{f}(n)\big)\approx \alpha_n=\sum_{\ell=0}^{b-1} \arg\left(A_{n_{\ell+1},n_{\ell}}\right).
\end{equation*}
\end{enumerate} 
\begin{raggedright}
\textbf{Output}
\par\end{raggedright}
\begin{raggedright} 
An approximation of $f$ given by 
\begin{equation*}
    f_e(x)=\sum_{n\in\mathcal{S}}a_n\mathbbm{e}^{\mathbbm{i}\alpha_n}\mathbbm{e}^{\mathbbm{i}nx}.
\end{equation*}
\par\end{raggedright}
\raggedright{}\textbf{\caption{\textbf{Signal Recovery with Compactly Supported Masks} \label{Big Algorithm Compact}}
}
\end{algorithm}
\textbf{}
\begin{algorithm}
\begin{raggedright}
\textbf{Inputs}
\par\end{raggedright}
\begin{enumerate}
\item \begin{raggedright}
Vector of amplitudes $\mathbf{a}=(a_n)_{n\in\mathcal{D}},\quad a_n=\sqrt{|A_{n,n}|}$.
\par\end{raggedright}
\item \begin{raggedright}
Entry $n\in\mathcal{S}$.\par\end{raggedright}
\end{enumerate} 
\begin{raggedright}
\textbf{Steps}
\par\end{raggedright}
\begin{enumerate}
\item Choose $n_0=\argmax_{n\in\mathcal{S}}a_n$.
\item Let $b=0$.
\item While: $|n-n_b|\geq \gamma$.
\subitem {\hspace{.1in} If}: $n>n_b$, let $n_{b+1}\leftarrow\argmax_{n_b+\gamma-\beta\leq m<n_b+\gamma} a_m$.
\subitem {\hspace{.1in} If}: $n<n_b$, let $n_{b+1}\leftarrow\argmax_{n_b-\gamma<m\geq n_b-\gamma+\beta} a_m$.
\subitem \hspace{.1in} $b\leftarrow b+1$.
\item $n_b\leftarrow n$.
\end{enumerate} 
\begin{raggedright}
\textbf{Output}
\par\end{raggedright}
\begin{raggedright}A sequence $\{n_\ell\}_{\ell=0}^b$, $|n_{\ell+1}-n_{\ell}|<2\beta$, $n_b=n$, $b\leq \frac{d}{\beta}$.
\par\end{raggedright}
\raggedright{}\textbf{\caption{\textbf{ Entry Selection}
\label{alg: greedy entry selection}}}
\end{algorithm}


\section{Empirical Evaluation}
\label{sec:Eval}
We now present numerical results demonstrating the efficiency and robustness of 
Algorithms \ref{Big Algorithm} and \ref{Big Algorithm Compact}.

\subsection{Empirical Evaluation of Algorithm \ref{Big Algorithm}}
\label{subsec:numerics_alg1}
We begin by 
investigating the empirical performance of Algorithm \ref{Big Algorithm} in recovering 
the following class of compactly supported $C^\infty$-smooth test functions,
\begin{equation}
	f(x) \coloneqq \sum_{j=1}^J\alpha_j \ \xi_{c_1, c_2}(x-\nu_j).
	\label{eq:testfun_numerics}
\end{equation}
Here $J \in \mathbb N$, $\alpha_j \in \mathbb C$, $\nu_j \in [-\pi, \pi]$, and 
$\xi_{c_1,c_2}$ denotes a $C^\infty$-smooth bump function with $\xi_{c_1,c_2}(x)>0$ 
in $(c_1,c_2)$ and $\xi_{c_1,c_2}(x) = 0$ for $x \notin [c_1, c_2]$. For the 
experiments below, we set $J=4$, $c_1=-\pi/5$, $c_2=\pi/5$, and choose $\alpha_j$ 
such that its real and complex components are both i.i.d. uniform random variables 
$\mathcal U[-1,1]$. The shifts $\nu_j$ are selected uniformly at random (without 
repetition) from the set 
$\left \{ -\nu_{{\rm max}}+j(2\nu_{{\rm max}}/(2J-1))\right \}_{j=0}^{2J-1}$ where  $\nu_{{\rm max}} = 0.9\pi - \max\{|c_1|, |c_2|\}$
so that ${\rm supp}(f) \subseteq [-\pi,\pi]$. A 
representative plot of (the real and imaginary parts of) such a test function is
provided in Fig. \ref{fig:testfnc_alg1}.  

\begin{figure}[htbp]
\captionsetup[subfigure]{justification=centering}
\centering
\begin{subfigure}[b]{0.45\textwidth}
\centering
\includegraphics[clip=true, trim = .85in 2.5in 0.85in 2.5in,scale=0.3]{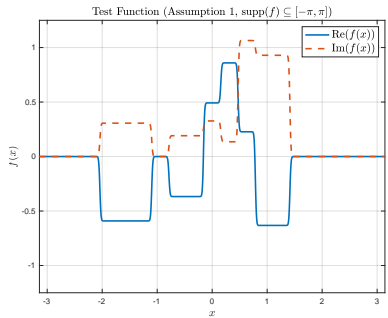}
\caption{Test Function 
(with ${\rm supp}(f)\subseteq [-\pi,\pi]$)}
\label{fig:testfnc_alg1}
\end{subfigure}
\hspace{.35in}
\begin{subfigure}[b]{0.45\textwidth}
\centering
\includegraphics[clip=true, trim = .85in 2.5in 0.85in 2.5in,scale=0.3]{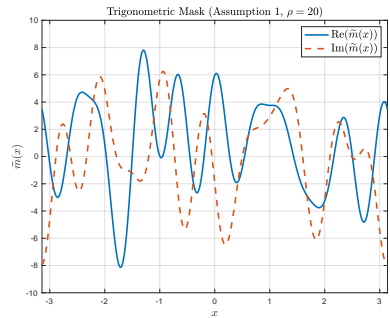}
\caption{Mask (Trigonometric Polynomial; 
$\rho=20$)}
\label{fig:mask_alg1}
\end{subfigure} 
\caption{Representative Test Function and Mask Satisfying Assumption $1$.}
\label{fig:testfnc_mask_alg1}
\end{figure}

To generate masks satisfying Assumption 1 (see Section 1.1), we choose the Fourier 
coefficients $\widehat{m}$ from a zero mean, unit variance i.i.d. complex Gaussian 
distribution and  empirically verify that the mask-dependent constant $\mu_1$ 
(as defined in (\ref{eq: mu1}) is strictly positive. Fig. \ref{fig:mask_alg1} plots such a 
(complex) trigonometric mask for $\rho = 20$, where $\rho+1$ is the (two-sided) 
bandwidth of the mask. Table \ref{tab:muvals_alg1} lists the empirically calculated 
$\mu_1$ values, and averaged over $100$ trials) for 
such masks. The left two columns of the table list $\mu_1$ for a fixed discretization size ($d=211$) and varying $\rho$; they show that $\mu_1$ is approximately constant for 
fixed $d$. The right two columns list $\mu_1$ values for fixed $\rho$ and varying 
$d$; they show $\mu_1$ decreases slowly with $d$ (roughly proportional to $1/d$). 
This verifies that constructing admissible (i.e., with $\mu_1\neq 0$) trigonometric 
masks as per Assumption 1 is indeed possible for reasonable values of $d$ and 
$\rho$.
\begin{table}[htbp]
\centering
\begin{tabular}{|c|c||c|c|}
\hline
($d=211,\rho)$ &    $\mu_1$ (Average over $100$ trials)    & ($d,\rho=50)$ & $\mu_1$ (Average over $100$ trials)   \\
\hline
\hline
  $(211,20)$     &  $1.957 \times 10^{-4}$ & $(111,50)$ & $4.825 \times 10^{-4}$\\
  $(211,40)$     &  $1.704 \times 10^{-4}$ & $(223,50)$ & $1.560 \times 10^{-4}$\\
  $(211,60)$     &  $1.563 \times 10^{-4}$ & $(447,50)$ & $6.199 \times 10^{-5}$\\
  $(211,80)$    &  $1.500 \times 10^{-4}$ & $(895,50)$ & $2.162 \times 10^{-5}$\\
  $(211,100)$    &  $1.530 \times 10^{-4}$ & $(1791,50)$ & $8.247 \times 10^{-6}$\\
 \hline
\end{tabular}
\caption{Empirically evaluated $\mu_1$ values (mask constant) for Algorithm \ref{Big Algorithm}. (Fourier coefficients of mask chosen as i.i.d. complex standard normal entries. Left two columns show $\mu_1$ values for fixed $d$, right two columns show $\mu_1$ values for fixed $\rho$.)}
\label{tab:muvals_alg1}
\end{table}

Finding closed form analytical expressions for the integral in (\ref{eq: restrict to pi}) is 
non-trivial. Therefore, we use numerical 
quadrature computations on an equispaced fine grid (of $10,001$ points) in $[-\pi, \pi]$ 
to generate phaseless measurements corresponding to (\ref{eq: restrict to pi}) under both Assumptions 1 and 2.

We now investigate the noise robustness of Algorithm \ref{Big Algorithm}. For the results shown in 
Fig. \ref{fig:alg1_noise} (where each data point is generated by averaging the results of 
$100$ trials), we add i.i.d. random (real) Gaussian noise to the phaseless measurements 
(\ref{eq: restrict to pi}) at desired signal to noise ratios (SNRs). In particular, the noise matrix 
$\mathbf{\bm{\eta}_{K,L}} \in \mathbbm R^{d\times L}$ in Section \ref{sec: WDD} is chosen to be i.i.d. 
$\mathcal N(\mathbf 0,\sigma^2\mathbf I)$. The variance $\sigma^2$ is chosen such
that 
\[ \mbox{SNR (dB)} = 10 \log_{10} \left( \frac{\Vert \mathbf{Z}\Vert_F^2}
		{dL \, \sigma^2}\right) \]
where $\mathbf{Z}$ denotes the corresponding matrix of 
perfect (noiseless) measurements. 
Errors in the recovered signal are also reported in dB with 
\[ \mbox{Error (dB)} = 10 \log_{10} \left( \frac{h \sum_{i=0}^N\vert f(x_i) - 
f_e(x_i) \vert^2}{h \sum_{i=0}^N\vert f(x_i) \vert^2}\right), \]
where $f$ and $f_e$ denote the true and recovered functions respectively, and $x_i$ denotes (equispaced) grid 
points in $[-\pi, \pi]$, i.e. $x_i = -\pi + hi$ with $h \coloneqq 2\pi/N$. Errors 
reported in this section use $N=2003$.

Fig. \ref{fig:alg1_noise} plots the error in recovering a test function using
Algorithm \ref{Big Algorithm} (for $d=257,\rho=32, \kappa=\rho-1$ and $(2\rho-1)d$ total measurements) over a wide range of SNRs. 
For reference, we also include
results using an improved reconstruction method 
based on Algorithm \ref{Big Algorithm}, as well as the popular HIO+ER alternating 
projection algorithm \cite{bauschke2002phase,fienup1982phase,marchesini2016alternating}. 
Refinements over Algorithm \ref{Big Algorithm} 
included use of an improved eigenvector-based magnitude estimation procedure in place 
of Step 6 (see \cite[Section 6.1]{IPSV18} for details), and (exponential) low-pass filtering\footnote{With filter order increasing with 
SNR; we used a $2$\textsuperscript{nd}-order filter at $10$dB SNR and a $12$\textsuperscript{th}-order 
filter at $60$dB SNR.} in the output Fourier partial sum 
reconstruction step of Algorithm~\ref{Big Algorithm}.
The HIO+ER algorithm implementation used the zero vector as an initial guess, 
although use of a random starting guess did not change the qualitative nature of the
results. As is common practice, (see for example \cite{fienup1982phase}) we
implemented the HIO+ER algorithm in blocks of eight HIO iterations
followed by two ER iterations in order to accelerate convergence of the
algorithm. To minimize computational cost while ensuring convergence (see Fig. \ref{fig:alg1_elbow_hio}), the
total number of HIO+ER iterations was limited to $30$.  
As we see, Algorithm \ref{Big Algorithm} compares well with the popular HIO+ER algorithm, with 
the improved method offering even better noise performance. Furthermore, this post-processing procedure does not significantly increase the computational cost. Fig. \ref{fig:alg1_etime} plots the execution time (in seconds, averaged over $100$ trials) to recover a test signal using $dL$ measurements, where $d$ is the discretization size, $L=2\rho-1$ and $\rho = \min\{ (d-5)/2, 2\lfloor\log_2(d)\rfloor\}$. Both Algorithm \ref{Big Algorithm} and its refined variant are essentially $\mathcal{O}(dL)$, where $dL$ is the number of measurements acquired, with Algorithm \ref{Big Algorithm} performing much faster than the HIO+ER procedure. Finally, we note that reconstruction error can be reduced by increasing the 
number of shifts $L$ acquired (and consequently, the total number of measurements). Fig. \ref{fig:alg1_shifts} plots the error in reconstructing a test signal discretized using $d=257$ points, $\kappa=\rho-1$ and $Ld = (2\rho-1)d$ measurements for different values of $\rho$ (and correspondingly $L$). As expected, we see that noise performance improves as $L$ increases. Additional numerical experiments studying the convergence behavior of Algorithm \ref{Big Algorithm} (in the absence of measurement errors) can be found in Appendix \ref{sec:appendix_numerics}.

\begin{figure}[htbp]
\centering
\begin{subfigure}[b]{0.245\textwidth}
\includegraphics[clip=true, trim = .8in 2.5in 0.85in 2.5in,scale=0.245]{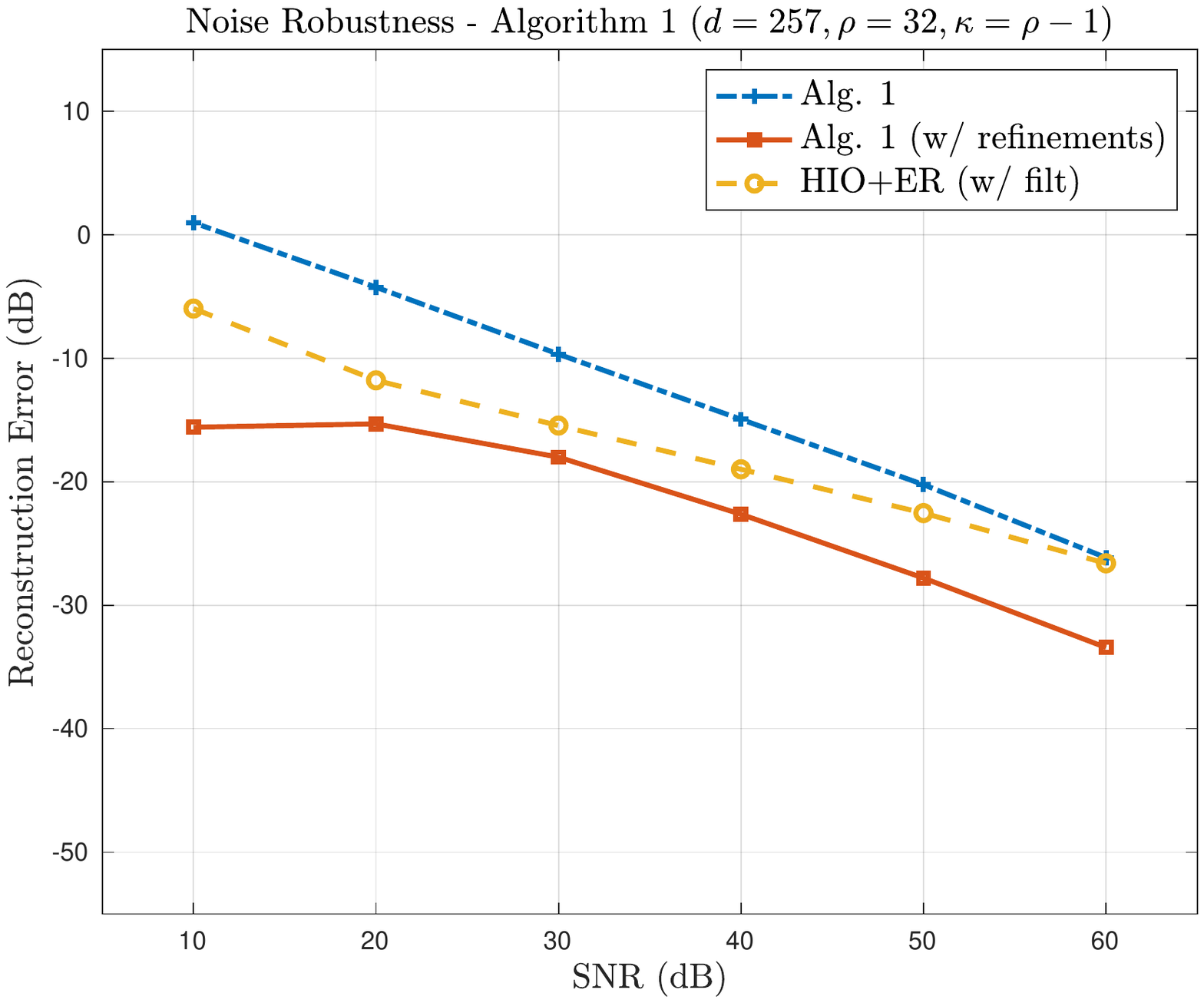}
\caption{Noise Robustness}
\label{fig:alg1_noise}
\end{subfigure}
\hfill
\begin{subfigure}[b]{0.245\textwidth}
\includegraphics[clip=true, trim = .8in 2.5in 0.85in 2.5in,scale=0.245]{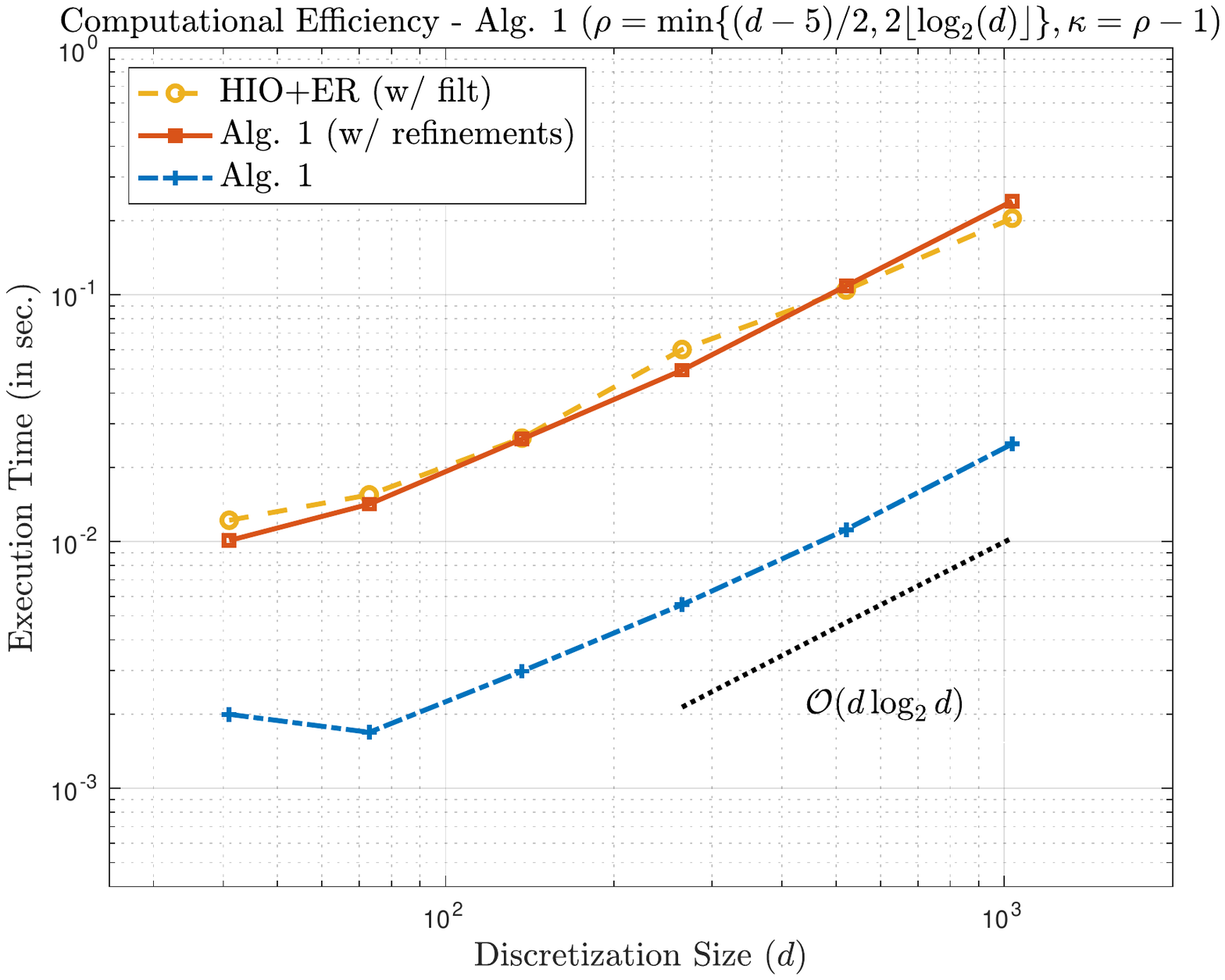}
\caption{Computational Cost}
\label{fig:alg1_etime}
\end{subfigure} 
\hfill
\begin{subfigure}[b]{0.245\textwidth}
\includegraphics[clip=true, trim = .8in 2.5in 0.85in 2.5in,scale=0.245]{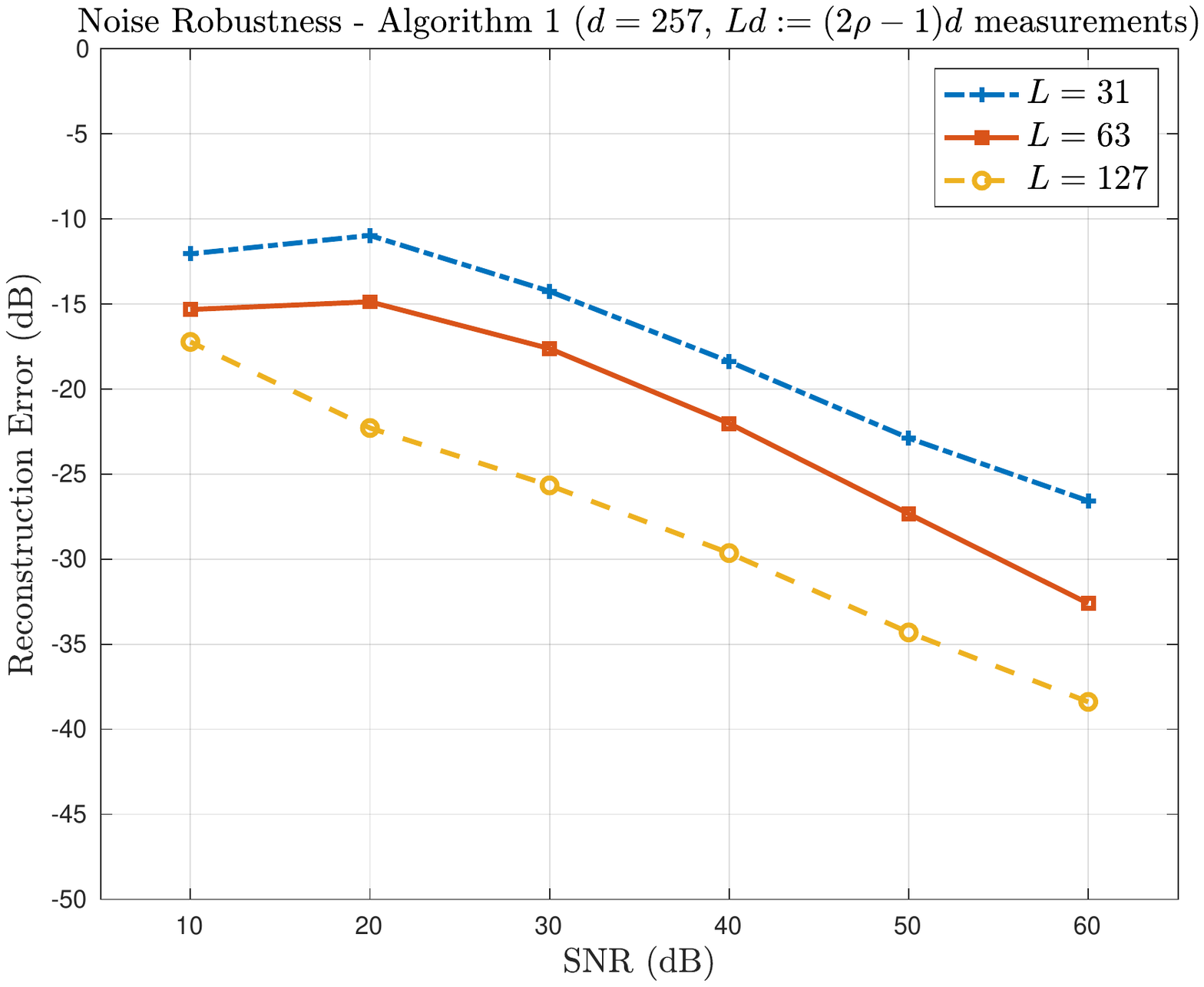}
\caption{Error vs. No. of Shifts}
\label{fig:alg1_shifts}
\end{subfigure}
\hfill
\begin{subfigure}[b]{0.245\textwidth}
\includegraphics[clip=true, trim = .8in 2.5in 0.85in 2.5in,scale=0.245]{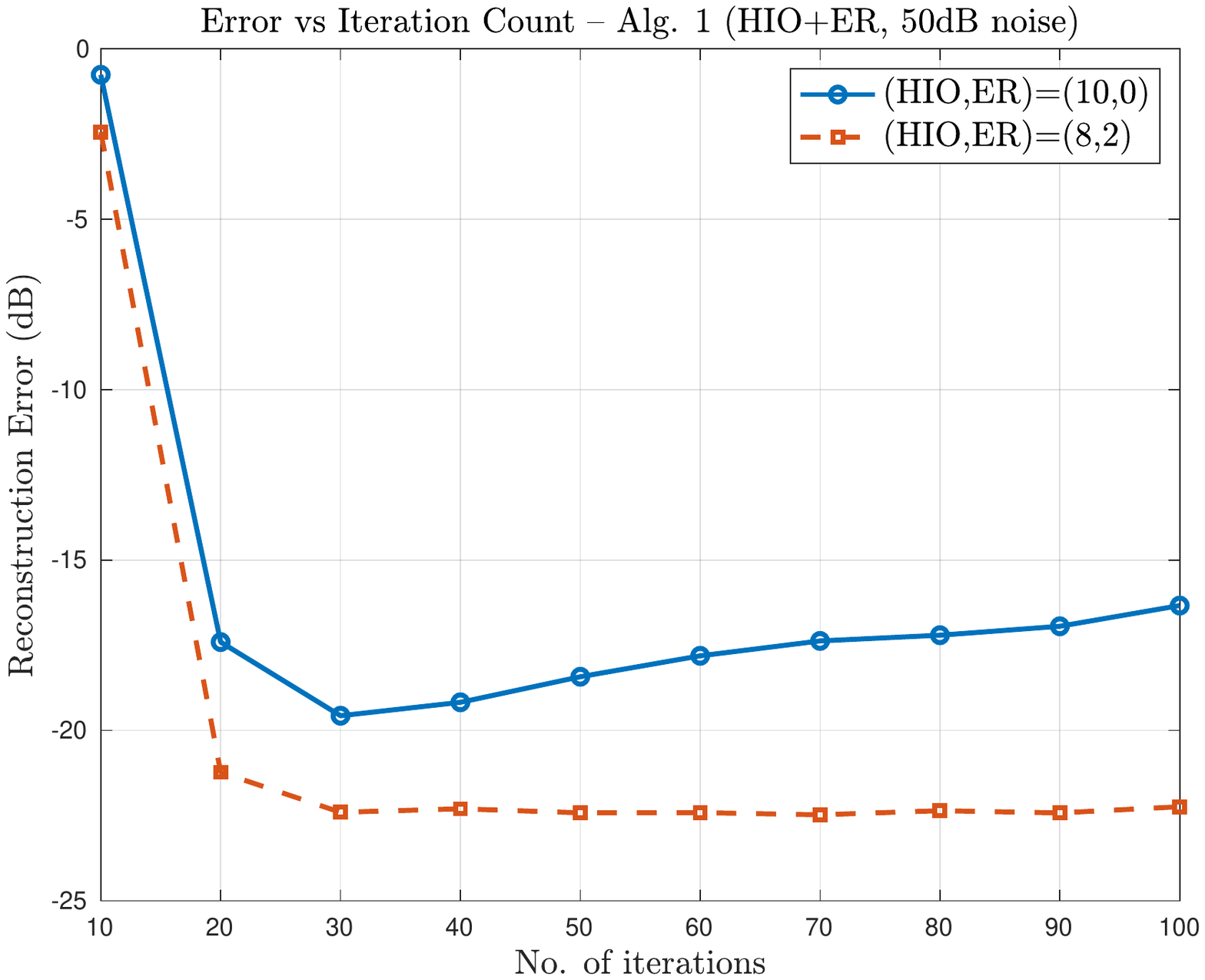}
\caption{HIO+ER Iterations\protect\footnotemark}
\label{fig:alg1_elbow_hio}
\end{subfigure}
\caption{Empirical Evaluation of Algorithm \ref{Big Algorithm} and Selection of HIO+ER Parameters for Comparison}
\label{fig:alg1_numerics}
\end{figure}
\footnotetext{The notation (HIO,ER)=($x$,$y$) in this figure denotes implementation of the HIO+ER algorithm in ``blocks" of $x$ iterations of the HIO algorithm followed by $y$ iterations of the ER algorithm. We choose $30$ total iterations of the red dashed plot in our implementations of the HIO+ER algorithm in this section.}

\subsection{Empirical Evaluation of Algorithm \ref{Big Algorithm Compact}}
\label{subsec:numerics_alg2}
We next present empirical simulations evaluating the robustness and efficiency 
of Algorithm \ref{Big Algorithm Compact}. As detailed in Assumption 2 (see Section 1.1), we recover compactly supported test functions with $\supp(f) \subseteq (-a,a)$ using compactly supported masks which satisfy $\supp(\widetilde m) \subseteq (-b,b)$, where $a+b<\pi$. For experiments in this section, we choose $b=3/4$ and $a=0.9(\pi-3/4)$. The test functions are generated as detailed in 
(\ref{eq:testfun_numerics}) of Section \ref{subsec:numerics_alg1}, as a (complex) weighted sum of shifted $C^\infty$-smooth bump functions, but with a maximum shift of 
$\nu_{{\rm max}} = a-b$. A representative test function is plotted in Fig. \ref{fig:testfnc_alg2}. The corresponding compactly supported masks are generated as the product of a trigonometric polynomial and a bump function 
using
\begin{equation}
    \widetilde{m}(x) = \xi_{-b.b} (x) \cdot \left( \sum_{p=-\rho/2}^{\rho/2} \widehat m(p) \mathbbm e^{\mathbbm i px/b} \right), 
    \label{eq:mask_alg2}
\end{equation}
where $\xi_{-b,b}$ is the $C^\infty$-smooth bump function described in Section \ref{subsec:numerics_alg1}, and the term in the parenthesis describes a (complex) $2b$-periodic trigonometric polynomial. A representative example of such as mask is provided in Fig. \ref{fig:mask_alg2} with $\rho=16$ and the coefficients $\widehat m$ chosen from a zero mean, unit variance i.i.d. complex Gaussian distribution. 

\begin{figure}[htbp]
\captionsetup[subfigure]{justification=centering}
\centering
\begin{subfigure}[b]{0.45\textwidth}
\centering
\includegraphics[clip=true, trim = .85in 2.5in 0.85in 2.5in,scale=0.3]{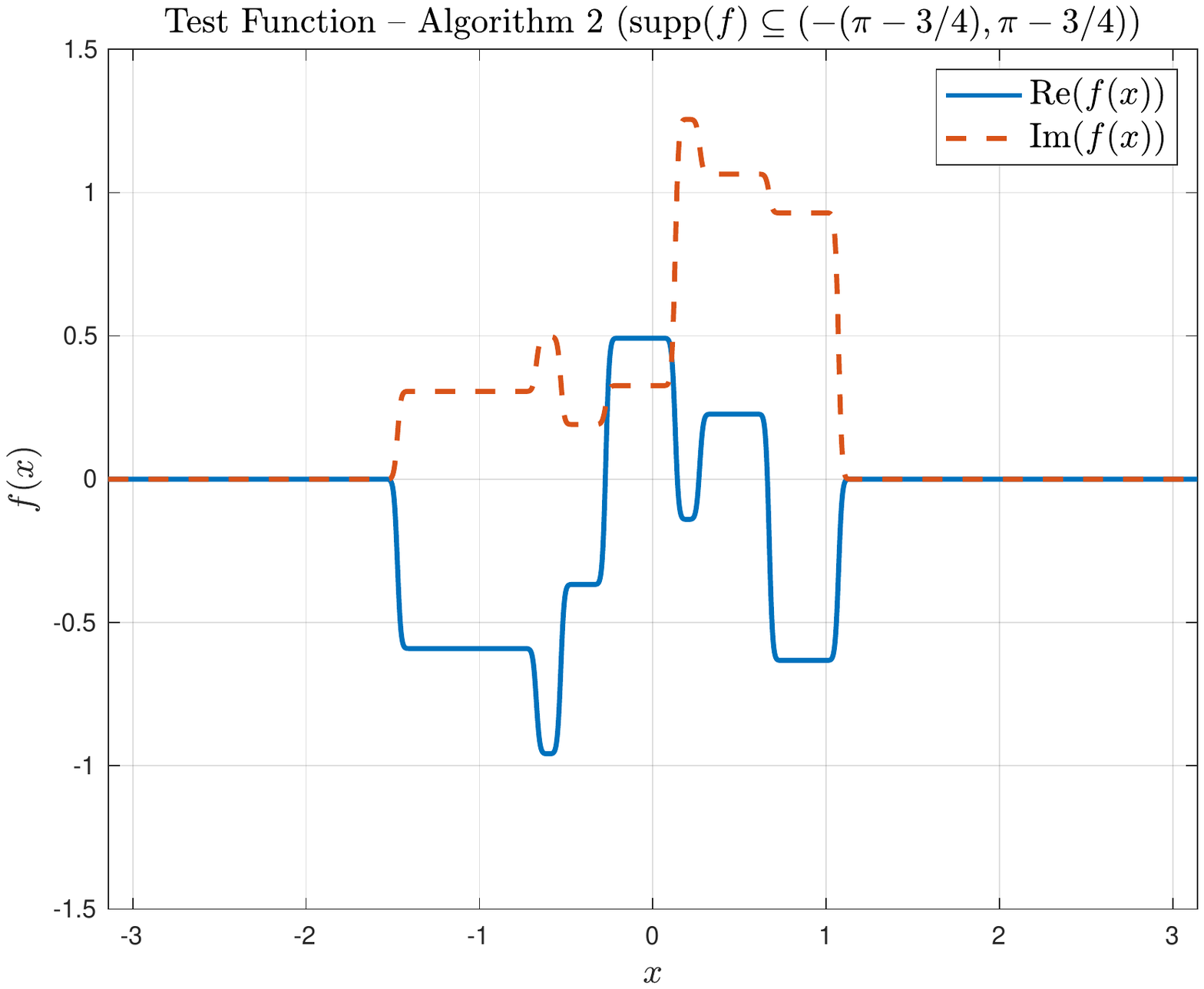}
\caption{Test Function ($a=\pi-3/4$; ${\rm supp}(f)\subseteq (-a,a))$}
\label{fig:testfnc_alg2}
\end{subfigure}
\hspace{.35in}
\begin{subfigure}[b]{0.45\textwidth}
\centering
\includegraphics[clip=true, trim = .85in 2.5in 0.85in 2.5in,scale=0.3]{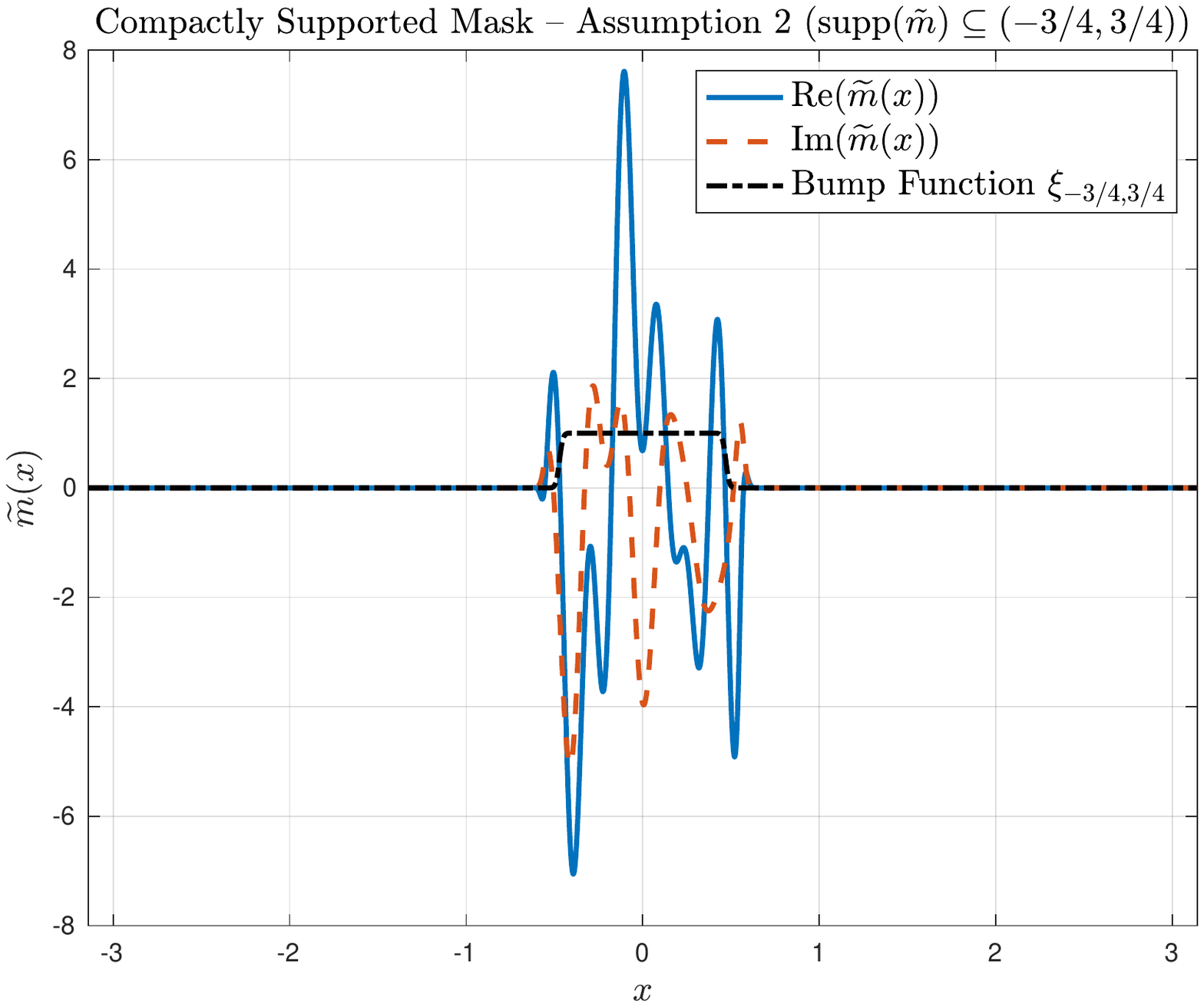}
\caption{Mask (${\rm 
supp}(\widetilde{m})\subseteq (-b,b) = (-3/4,3/4)$)}
\label{fig:mask_alg2}
\end{subfigure}
\caption{Representative Test Function and Mask Satisfying Assumption $2$.}
\label{fig:testfnc_mask_alg2}
\end{figure}
\begin{table}[htbp]
\centering
\begin{tabular}{|c|c||c|c|}
\hline
($d=189,\kappa)$ &    $\mu_2$ (Average over $100$ trials)    & ($d,\kappa=27)$ & $\mu_2$ (Average over $100$ trials)   \\
\hline
\hline
  $(189,3)$     &  $2.563 \times 10^{-3}$ & $(165,27)$ & $9.722 \times 10^{-5}$\\
  $(189,10)$     &  $2.873 \times 10^{-4}$ & $(223,27)$ & $8.866 \times 10^{-5}$\\
  $(189,31)$     &  $8.331 \times 10^{-5}$ & $(495,27)$ & $4.686 \times 10^{-5}$\\
  $(189,94)$    &  $2.642 \times 10^{-19}$ & $(1045,27)$ & $2.448 \times 10^{-5}$\\
 \hline
\end{tabular}
\caption{Empirically evaluated $\mu_2$ values (mask constant) for Algorithm \ref{Big Algorithm Compact}. The left two columns show $\mu_2$ values for fixed $d$, right two columns show $\mu_2$ values for fixed $\kappa$. Here, 
$\delta = \kappa+1$ and $s=\kappa-1$.}
\label{tab:muvals_alg2}
\end{table}

Representative values of the mask constant $\mu_2$ (as defined in (\ref{eq: mu2}) and averaged over $100$ trials) are listed in Table \ref{tab:muvals_alg2}. The first two columns list $\mu_2$ values for fixed discretization size $d$, while the last two columns list $\mu_2$ values for fixed $\kappa$. In both cases, we set $K=2\kappa+1$ and ensure that $K$ divides $d$. We note that $\kappa$ denotes the number of modes used in the Wigner deconvolution procedure (Step $2$) in Algorithm \ref{Big Algorithm Compact}. Since the masks constructed using (\ref{eq:mask_alg2}) are compactly supported and smooth, we expect the autocorrelation of their Fourier transforms (and the corresponding Fourier coefficients of this autocorrelation) to decay rapidly. Therefore, we expect $\mu_2$ to be small for large $\kappa$ values; indeed, this is seen in the last row of Table \ref{tab:muvals_alg2} where the $\mu_2$ value is essentially zero when $d=189,\kappa=94$. However, as the functions we expect to recover also exhibit rapid decay in Fourier coefficients, we only require a small number of their Fourier modes to ensure accurate reconstructions. Hence, small to moderate $\kappa$ values suffice. As seen in Table \ref{tab:muvals_alg2}, it is feasible to construct admissible masks (i.e., $\mu_2> 0$)  for such $(d,\kappa)$ pairs. 
Experiments have also been conducted with $\widetilde m$ chosen to be the bump function $\xi_{-b,b}$ and a (truncated) Gaussian, However, these experiments yield smaller mask constants $\mu_2$, which make the resulting  reconstructions more susceptible to noise. Selection of ``optimal" and physically realizable compactly supported masks is an open problem which we defer to future research. 

We note that due to the equivalence of (\ref{eq:FYF_single_aliasing}) and (\ref{eq:FYF_single_aliasing-modified}), the Wigner deconvolution step (Step $2$) in Algorithm \ref{Big Algorithm Compact} may be instead evaluated using (\ref{eq:FYF_single_aliasing}). While theoretical analysis of this equivalent procedure is more involved, it offers computational advantages since it does not require solving\footnote{We use the {\em Iterated Tikhonov} method (see \cite{buccini2017iterated}, \cite[Algorithm 3]{Perlmutter2020}) to invert the Vandermonde system in Step $4$ of Alg. \ref{Big Algorithm Compact}.} the Vandermonde system of Step $4$ in Algorithm \ref{Big Algorithm Compact}. The corresponding $\mu_2$ values for this procedure also follow the qualitative behavior in Table \ref{tab:muvals_alg2}. This variant of Algorithm \ref{Big Algorithm Compact} is used in generating some of the plots in Appendix \ref{sec:appendix_numerics}, while Fig. \ref{fig:alg2_numerics} provides a comparison of Algorithm \ref{Big Algorithm Compact} and this alternate implementation. 

\begin{figure}[htbp]
\centering
\begin{subfigure}[b]{0.32\textwidth}
\includegraphics[clip=true, trim = .85in 2.5in 0.85in 2.5in,scale=0.32]{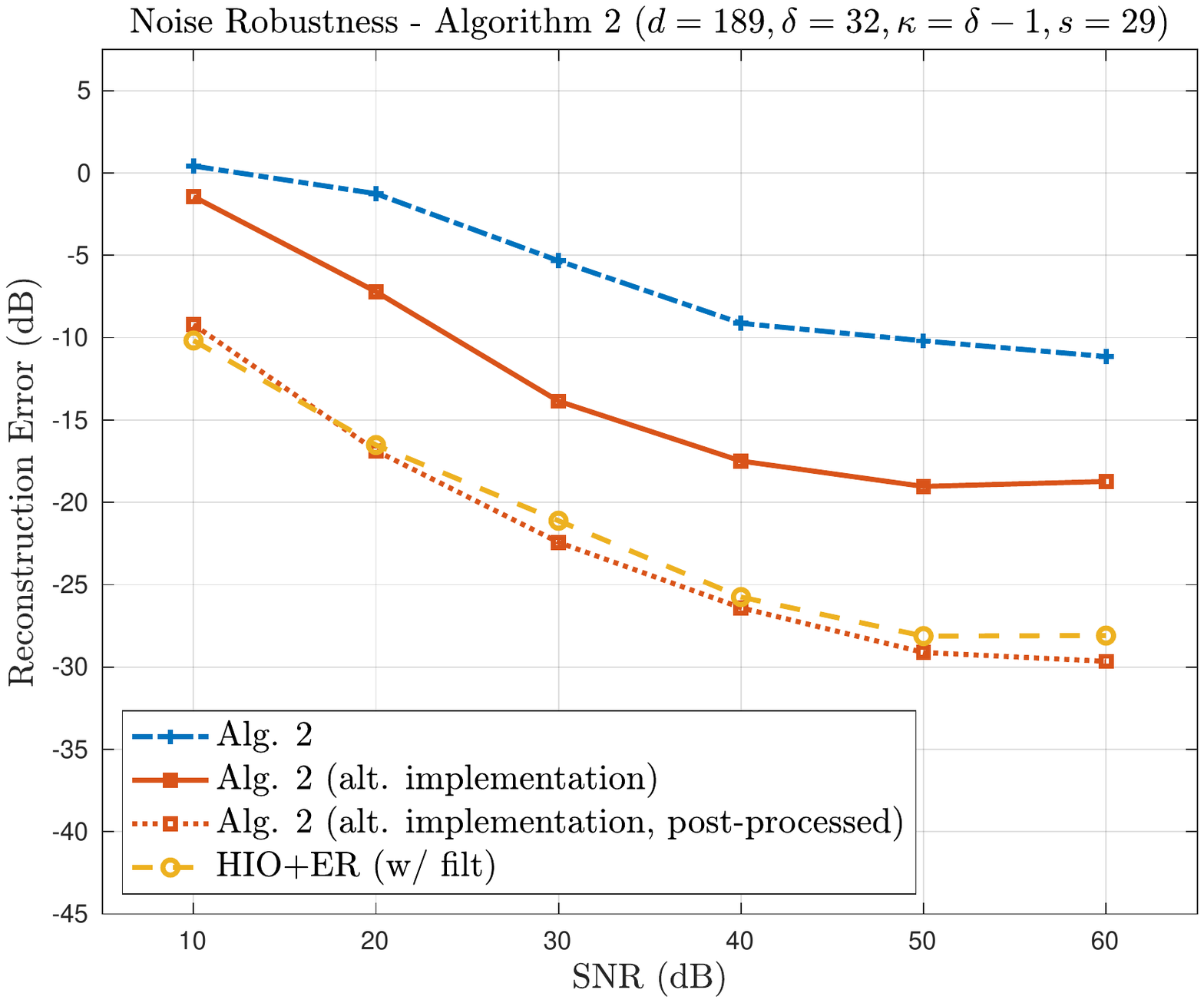}
\caption{Noise Robustness}
\label{fig:alg2_noise}
\end{subfigure}
\hfill
\begin{subfigure}[b]{0.32\textwidth}
\includegraphics[clip=true, trim = .85in 2.5in 0.85in 2.5in,scale=0.32]{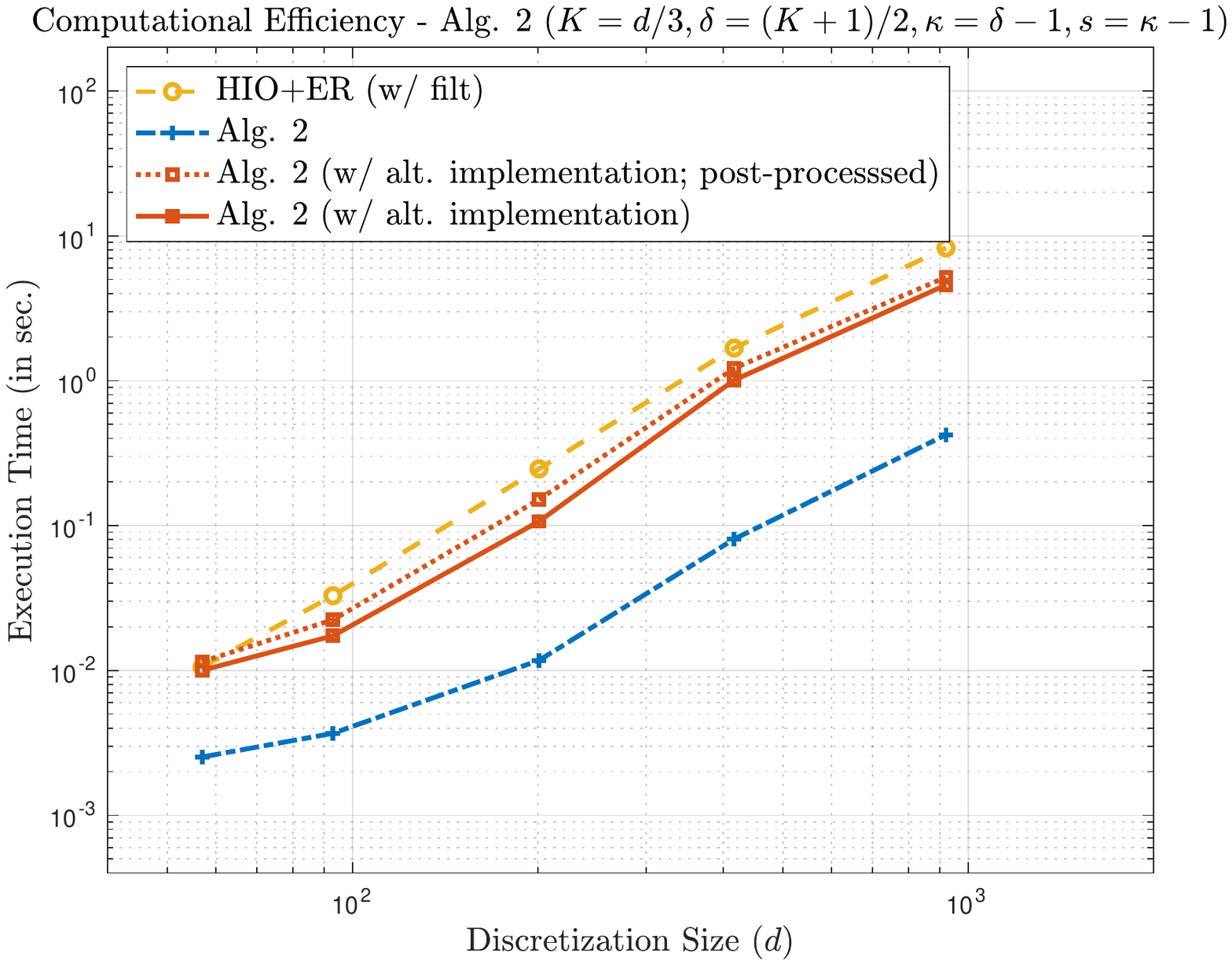}
\caption{Computational Cost}
\label{fig:alg2_etime}
\end{subfigure} 
\hfill
\begin{subfigure}[b]{0.32\textwidth}
\includegraphics[clip=true, trim = .85in 2.5in 0.85in 2.5in,scale=0.32]{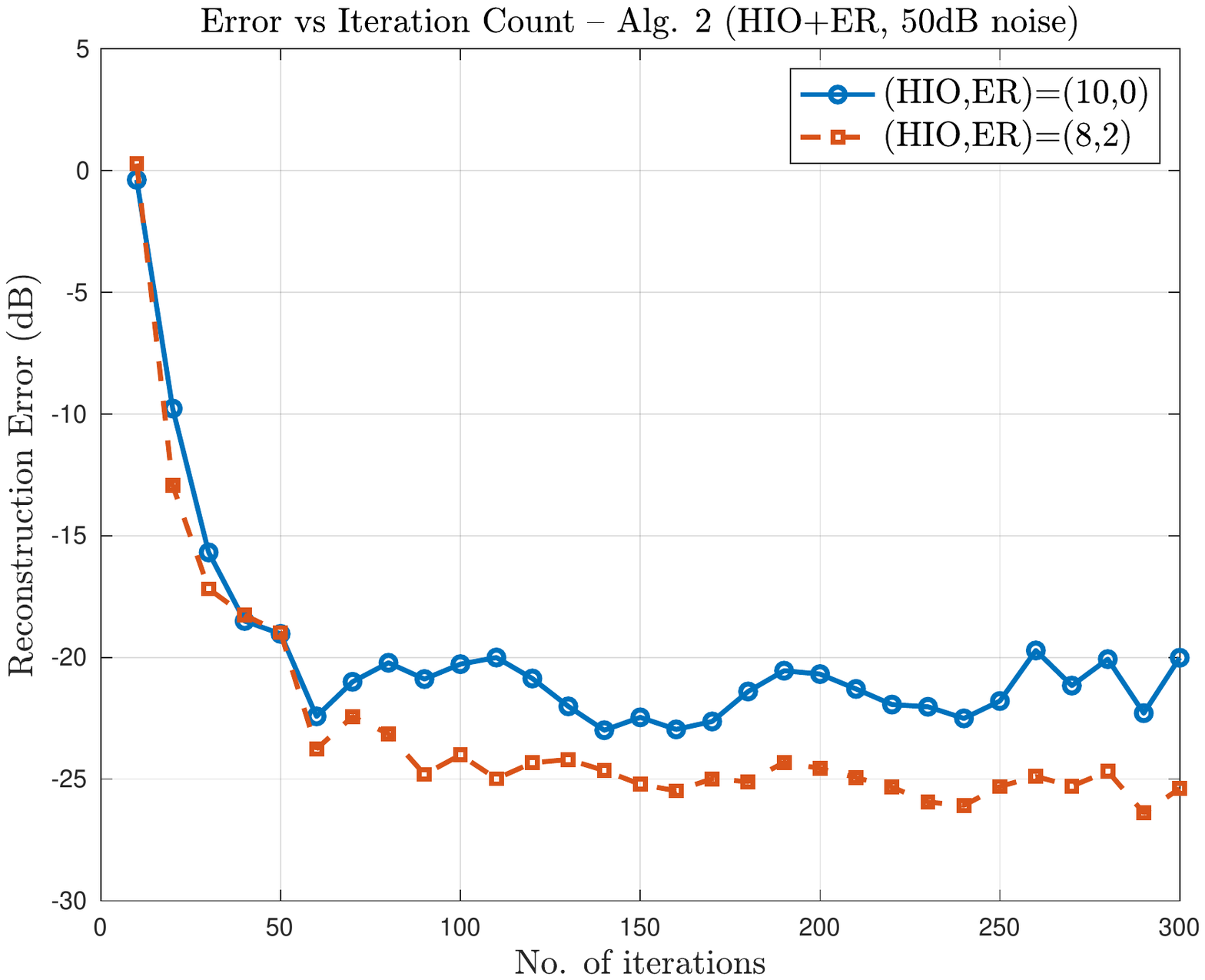}
\caption{HIO+ER Iterations}
\label{fig:alg2_elbow_hio}
\end{subfigure}
\caption{Empirical Evaluation of Algorithm \ref{Big Algorithm Compact} and Selection of HIO+ER Parameters for Comparison}
\label{fig:alg2_numerics}
\end{figure}

We now study the robustness and computational efficiency of Algorithm \ref{Big Algorithm Compact}. Fig. \ref{fig:alg2_noise} plots the error in recovering a test function (with each data point averaged over $100$ trials) for discretization size $d=189$, $\delta=32$, $\kappa=\delta-1$, $s=29$ and $d/3$ total measurements over a wide range of SNRs. 
For reference, we also include
results using the HIO+ER alternating 
projection algorithm, as well as the alternate implementation of Algorithm \ref{Big Algorithm Compact} (using (\ref{eq:FYF_single_aliasing}) to implement the Wigner deconvolution Step $2$). 
As in Section \ref{subsec:numerics_alg1}, the alternate implementation of Algorithm \ref{Big Algorithm Compact} and the HIO+ER implementations utilize (exponential) low-pass filtering.
The HIO+ER algorithm is implemented in blocks of eight HIO iterations
followed by two ER iterations in order to accelerate the convergence of the
algorithm, with a total of $100$ iterations used to ensure convergence while minimizing computational cost (see Fig. \ref{fig:alg2_elbow_hio}). 
The proposed method (especially the alternate implementation) compares well with the HIO+ER algorithm. 
Additionally, we also provide results using a post-processed implementation of Algorithm \ref{Big Algorithm Compact} using just $10$ iterations of HIO+ER. In this context, we can view the proposed method as an initializer which accelerates the convergence of alternating projection algorithms such as HIO+ER.
Finally, Fig. \ref{fig:alg2_etime}, which plots the execution time (in seconds, averaged over $100$ trials) to recover a test signal, shows that the proposed method in Algorithm \ref{Big Algorithm Compact} and its alternate implementation are computationally efficient, with all implementations running in $\mathcal{O}(dK)$ time where $dK$ is the number of measurements acquired.


\appendix
\section{The Proofs of Lemmas \ref{lem:Intbound} and \ref{lem: discretesum}}\label{app: proof of discretization lemmas}

\begin{proof}[The Proof of Lemma \ref{lem:Intbound}]
We first note that  $\|g\|_{L^\infty([-\pi,\pi])}<\infty$ since $g$ is a continuous periodic function. Next, we see that  since $g$ is $C^k$-smooth, we have $|\widehat{g}(\omega)| \leq C_g \left( \frac{1}{|\omega|} \right)^k$ 
for all
 $\omega \in \mathbb{Z} \setminus \{0\}$,  where  $C_g$ is a constant which depends on only $g$ and $k$.  As a result, we have 
\begin{equation*}
    \| P_{\mathcal{A}}g \|_{L^\infty([-\pi,\pi])} \leq \sum_{\omega \in\mathbb{Z}} |\widehat{g}(\omega)| \leq |\widehat{g}(0)| + 2C_g\sum_{m=1}^\infty \frac{1}{m^k}=C_g.
\end{equation*}
Similarly, 

$$\| g - P_{\mathcal{N}}g \|_{L^\infty([-\pi,\pi])} \leq \sum_{|\omega| \geq \frac{n+1}{2}} |\widehat{g}(\omega)| \leq  2C_g\sum_{|\omega| \geq \frac{n+1}{2}} \left(\frac{1}{|\omega|}\right)^\ell\leq C_g\left(\frac{1}{n}\right)^{\ell-1}.$$
The desired result now follows.
\end{proof}

\begin{proof}[The Proof of Lemma \ref{lem: discretesum}] Let $g\coloneqq  P_\mathcal{S}f$ and $h\coloneqq P_\mathcal{R}m$, where $P_\mathcal{S}$ and $P_\mathcal{R}$  are the Fourier projection operators defined as in \eqref{eq:Fourierprojection}. Since $g$ and $h$ are  
 trigonometric polynomials and $\mathcal{R} + \mathcal{S} \subseteq \mathcal{D}$, we may write
\begin{align*}
\int_{-\pi}^{\pi}g(x)h(x-\tilde\ell)\mathbbm{e}^{-\mathbbm{i}x\omega}dx
&= \sum_{m\in \mathcal{R}} \sum_{n\in\mathcal{S}} \widehat{g}(n) \widehat{h}(m) \mathbbm{e}^{-\mathbbm{i}m\tilde\ell}\int_{-\pi}^{\pi} \mathbbm{e}^{\mathbbm{i}(m+n-\omega)x}dx\nonumber\\
&=\sum_{m\in \mathcal{R}} \sum_{n\in\mathcal{S}} \widehat{g}(n) \widehat{h}(m) \mathbbm{e}^{-\mathbbm{i}m\tilde\ell} \frac{2\pi}{d}\sum_{p\in\mathcal{D}} \mathbbm{e}^{2\pi \mathbbm{i} p(n+m-\omega)/d}\nonumber\\
&= \frac{2\pi}{d}\sum_{p\in\mathcal{D}}\left(\sum_{n\in \mathcal{S}} \widehat{g}(n) \mathbbm{e}^{2\pi \mathbbm{i} pn/d}\right)\left(\sum_{m\in\mathcal{R}}\widehat{h}(m) \mathbbm{e}^{\left(\left(\frac{2\pi \mathbbm{i}m}{d}\right)\left(p-\ell \right)\right)}\right) \mathbbm{e}^{-2\pi \mathbbm{i} m p\omega/d}\nonumber\\
&= \frac{2\pi}{d}\sum_{p\in\mathcal{D}} g\left(\frac{2\pi p}{d}\right)h\left(\frac{2\pi (p-\ell)}{d}\right)    \mathbbm{e}^{-2\pi \mathbbm{i} p\omega/d}
\\
&=\frac{2\pi}{d}\sum_{p\in\mathcal{D}} x_p y_{p-\ell} \mathbbm{e}^{-2\pi \mathbbm{i} \omega p/d}\nonumber.
\end{align*}
\end{proof}

\section{The Proofs of Propositions \ref{prop:mu_condition} and \ref{prop: mu2} }\label{app: mask design}

\begin{proof}[The Proof of Proposition \ref{prop:mu_condition}] 
We first note that  \begin{equation*}\widehat{z}_q=\begin{cases}\widehat{m}(q)&\text{if } |q|\leq \rho/2,\\
0 &\text{if } |q|> \rho/2. \end{cases}
\end{equation*}
Therefore, for all $|p|\leq \kappa-1$, we have
\[
\left(\widehat{\mathbf{z}}\circ S_{p}\overline{\widehat{\mathbf{z}}}\right)_{q}=\begin{cases}
\widehat{m}(q)\widehat{m}(p+q) & \text{if }-\rho/2\leq q,\,p+q \leq \rho/2,\\
0 & \text{otherwise}.
\end{cases}
\]
For any $|p|\leq \kappa-1,$
let \begin{equation*}
    \mathcal{I}_{p}\coloneqq\{q\in\mathcal{D}:-\rho/2\leq q\leq \rho/2\quad\text{and   } -\rho/2\leq q+p\leq \rho/2\}.
\end{equation*} One may check that 
\begin{equation*}
    \mathcal{I}_{p} = 
     \begin{cases}
     \left[-\frac{\rho}{2}-p, \frac{\rho}{2}\right]\cap\mathbb{Z}&\text{if }p<0\\
     \left[-\frac{\rho}{2}, \frac{\rho}{2}-p\right]\cap\mathbb{Z} &\text{if }p\geq 0
    \end{cases}.
\end{equation*}
Therefore, making a simple change of variables in the case $p<0,$ we have that
\begin{align*}
\mathbf{F_{d}}\left(\widehat{\mathbf{z}}\circ S_{p}\overline{\widehat{\mathbf{z}}}\right)_{q}
&=\frac{1}{d}\sum_{\ell\in\mathcal{I}_{p}}\widehat{m}(\ell)\widehat{m}(p+\ell)\mathbbm{e}^{-2\pi\mathbbm{i}q\ell/d}=\frac{1}{d}\sum_{\ell=-\rho/2}^{\rho/2-|p|}\widehat{m}(\ell)\widehat{m}\left(\ell+|p|\right)\mathbbm{e}^{\mathbbm{i}\phi_{p,q,\ell}},
\end{align*} 
where $\mathbbm{e}^{\mathbbm{i}\phi_{p,q,\ell}}$ is a unimodular complex number depending on $p,q$ and $\ell.$
Using the assumptions \eqref{eqn: a0big} and \eqref{eqn: decreasingamplitudes},
we see that 
\begin{align*}
 \bigg|\frac{1}{d}\sum_{\ell=-\rho/2+1}^{\rho/2-|p|}\widehat{m}(\ell)\widehat{m}(\ell+|p|)\mathbbm{e}^{\mathbbm{i}\phi_{p,q,\ell}}\bigg|  
 &\leq  \frac{\rho}{d} \left|\widehat{m}\left(\frac{-\rho}{2}+1\right)\right|\left|\widehat{m}\left(\frac{-\rho}{2}+1+|p|\right)\right|\\
 &\leq\frac{1}{2d} \left|\widehat{m}\left(\frac{-\rho}{2}\right)\right|\left|\widehat{m}\left(\frac{-\rho}{2}+|p|\right)\right|.
\end{align*}
With this, we may use the reverse triangle inequality to see  
\begin{align*}
\bigg|\mathbf{F_{d}}\left(\widehat{\mathbf{z}}\circ S_{p}\overline{\widehat{\mathbf{z}}}\right)_{q}\bigg| &=\bigg|\frac{1}{d}\sum_{\ell=-\rho/2}^{\rho/2-|p|}\widehat{m}(\ell)\widehat{m}(\ell+|p|)\mathbbm{e}^{\mathbbm{i}\phi_{p,q,\ell}}\bigg|\\
&\geq \frac{1}{d} \bigg|\widehat{m}\left(\frac{-\rho}{2}\right)\bigg|\left|\widehat{m}\left(\frac{-\rho}{2}+|p|\right)\right| - \frac{1}{d}\bigg|\sum_{\ell=-\rho/2+1}^{\rho/2-|p|}\widehat{m}(\ell)\widehat{m}(\ell+|p|)\mathbbm{e}^{\mathbbm{i}\phi_{p,q,\ell}}\bigg|\\
&\geq\frac{1}{2d} \left|\widehat{m}\left(\frac{-\rho}{2}\right)\right|\left|\widehat{m}\left(\frac{-\rho}{2}+|p|\right)\right|\\
&\geq\frac{1}{2d} \left|\widehat{m}\left(\frac{-\rho}{2}\right)\right|\left|\widehat{m}\left(\frac{-\rho}{2}+\kappa-1|\right)\right|.
\end{align*}
\end{proof}

\begin{proof}[The Proof of Proposition \ref{prop: mu2}]
First, we note that by applying Lemma \ref{lem: switchbackandforth}, and setting $p=\omega, q=\ell$, we have 
\begin{equation*}\mu_2= \inf_{\omega\in[2\kappa-1]_c,\ell\in[2s-1]_c} |
    (\mathbf{F_d}(\widehat{\mathbf{z}}\circ S_{\ell}\overline{\widehat{\mathbf{z}}}))_\omega|=\frac{1}{d}\inf_{\omega\in[2\kappa-1]_c,\ell\in[2s-1]_c} |
    (\mathbf{F_d}(\mathbf{z}\circ S_{\omega}\overline{\mathbf{z}}))_\ell|=\frac{1}{d}\inf_{p\in[2\kappa-1]_c,q\in[2s-1]_c} |
    (\mathbf{F_d}(\mathbf{z}\circ S_{p}\overline{\mathbf{z}}))_q|.
\end{equation*} 
For $|p|\leq \kappa-1$, we have
\[
\left(\mathbf{z}\circ S_{p}\overline{\mathbf{z}}\right)_{q}=\begin{cases}
z_q\overline{z_{p+q}} & \text{if }n\leq q,p+q \leq n+\tilde\delta-1,\\
0 & \text{otherwise}.
\end{cases}
\]
For any $|p|\leq \kappa-1,$
let \begin{equation*}
    \mathcal{I}_{p}\coloneqq\{q\in\mathcal{D}:n\leq q\leq n+\tilde\delta-1\quad\text{and}\quad n\leq q+p\leq n+\tilde{\delta}-1\}.
\end{equation*} One may check that 
\begin{equation*}
    \mathcal{I}_{p} = \begin{cases}
     [n-p, n+\tilde\delta-1]\cap\mathbb{Z}&\text{if }p<0,\\
      [n, n+\tilde\delta-1-p]\cap\mathbb{Z}&\text{if }p\geq 0.
    \end{cases}
\end{equation*}
Therefore, making a simple change of variables  in the case $p<0,$ we have that in either case
\begin{align*}\left|
\mathbf{F_{d}}\left(\mathbf{z}\circ S_{p}\overline{\mathbf{z}}\right)_{q}\right|
&=\frac{1}{d}\bigg|\sum_{\ell\in\mathcal{I}_{p}}z_\ell\overline{z_{p+\ell}}\mathbbm{e}^{-2\pi\mathbbm{i}\ell q/d}\bigg|=\frac{1}{d}\bigg|\sum_{\ell=n}^{n+\tilde\delta-1-|p|}z_{\ell}\overline{z_{\ell+|p|}}\mathbbm{e}^{\mathbbm{i}\phi_{p,q,\ell}}\bigg|,
\end{align*} 
where $\mathbbm{e}^{\mathbbm{i}\phi_{p,q,\ell}}$ is a unimodular complex number depending on $p,q$ and $\ell.$
Using the assumptions \eqref{eqn: big z n} and \eqref{eqn: zs decrease}
we see that
\begin{align*}
 \bigg|\frac{1}{d}\sum_{\ell=n+1}^{n+\tilde\delta-1-|p|}z_\ell\overline{z_{\ell+|p|}}\mathbbm{e}^{\mathbbm{i}\phi_{p,q,\ell}}\bigg|
 \leq \frac{\tilde{\delta} }{d} \left|z_{n+1}\right|\left|z_{n+1+|p|}\right|
 \leq\frac{1}{2d} |z_n||z_{n+|p|}|.
\end{align*}
With this,
\begin{align*}
\left|F_{d}\left(\mathbf{z}\circ S_{p}\overline{\mathbf{z}}\right)_{q}\right| &=\bigg|\frac{1}{d}\sum_{\ell=n}^{n+\tilde\delta-1-|p|}z_{\ell}\overline{z_{\ell+|p|}}\mathbbm{e}^{\mathbbm{i}\phi_{p,q,\ell}}\bigg|
\geq \frac{1}{d} |z_n||z_{n+|p|}| - \bigg|\frac{1}{d}\sum_{\ell=n+1}^{n+\tilde\delta-1-|p|}z_\ell
\overline{z_{\ell+|p|}}\mathbbm{e}^{\mathbbm{i}\phi_{p,q,\ell}}\bigg|\\
&\geq\frac{1}{2d}|z_n||z_{n+|p|}|
\geq\frac{1}{2d}|z_n||z_{n+\kappa-1}|.
\end{align*}

\end{proof}

\section{The Proof of Lemma \ref{la:phase-error}} 
\begin{proof}\label{app: pf La phase-erro}

 Our proof requires the following sublemma which shows that, if $n\in L_f$, then Algorithm \ref{alg: greedy entry selection} used in the definition of $\alpha_n$ will only select indices $n_\ell$ corresponding to large Fourier coefficients. 
\begin{lemma}\label{lem: allbig}
Let $n\in L_f$, and let $n_0,\ldots, n_b$ be the sequence of indices as introduced in the definition of $\alpha_n$. Then
\begin{equation*}|\widehat{f}(n_\ell)|\geq \frac{|\widehat{f}(n)|}{2} 
\end{equation*} for all $0\leq \ell\leq b$.
\end{lemma}

\begin{proof}When $\ell=b$, the claim is immediate from the fact that $n_b=n$. 
For all $0\leq \ell\leq b-1$, the definition of $n_\ell$ implies that there exists an interval $I_\ell$ of length $\beta$, which is centered at some point $a$ with $|a|\leq |n|$, such that
\begin{equation*}
    a_{n_\ell}=\max_{m\in I_\ell} a_m.
\end{equation*} 
Letting $\epsilon=\sqrt{3\|\mathbf{N}\|_\infty}$, we see that  by \eqref{la:mag-error} and Remark \ref{la:n-sequence}
\begin{align*}
    |\widehat{f}(n_\ell)|&\geq a_{n_\ell}-\epsilon
    =\max_{m\in I_\ell} a_m-\epsilon
    \geq \max_{m\in I_\ell} |\widehat{f}(m)|-2\epsilon
    \geq |\widehat{f}(n)|-2\epsilon.
\end{align*}
The result now follows from noting that $\epsilon<\frac{|\widehat{f}(n)|}{4}$ for all $n\in L_f$.  \end{proof}

With Lemma \ref{lem: allbig} established, we may now prove Lemma \ref{la:phase-error}.
Let $n\in L_f$ and let $n_0,\ldots n_b$ be the sequence describe in the definition of $\alpha_n$. 
For $0\leq \ell\leq b-1$, let  $t_{\ell} \coloneqq \widehat{f}(n_{\ell+1})\overline{\widehat{f}(n_{\ell})}$,
$a'_{\ell} \coloneqq \widehat{f}(n_{\ell+1})\overline{\widehat{f}(n_{\ell})}+N_{n_{\ell+1},n_{\ell}}$, and 
 $N'_{\ell}\coloneqq N_{n_{\ell+1},n_{\ell}}$.
Consider the triangle with sides $a'_{\ell}$, $t_{\ell}$, and $N'_{\ell}$ with angles $\theta_{\ell} =|\arg(a'_{\ell})-\arg(t_{\ell})|$ and $\phi_{\ell} = |\arg(a'_{\ell})-\arg(N'_{\ell})|$, as
illustrated in Figure~\ref{fig:Triangle}.
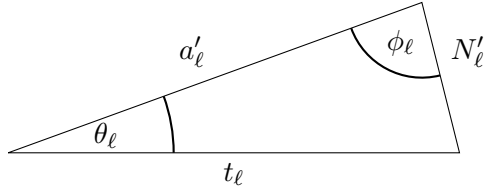
\begin{figure}[h!]
\centering \begin{tikzpicture}
\RectTri[black]{(0,0)}{(6,0)}{(5.5,2)}
\end{tikzpicture}
\caption{Triangle in the complex domain. }\label{fig:Triangle}
\end{figure}
By the law of sines and Lemma \ref{lem: allbig}, we get that 
\begin{equation}\label{eq:sin-estimate}
|\sin(\theta_{\ell})|= \left|\frac{N_{\ell}'}{t_{\ell}}\sin(\phi_{\ell})\right|\leq \frac{\|\mathbf{N}\|_\infty}{|\widehat{f}(n_{\ell})||\widehat{f}(n_{\ell+1})|}\leq \frac{4\|\mathbf{N}\|_\infty}{|\widehat{f}(n)|^2}
\end{equation}
for all $0\leq \ell\leq b$. By  the definition of $L_f$ and  Lemma \ref{lem: allbig}, we have that for all $\ell$ 
 \begin{equation*}
     |N'_\ell|\leq \|\mathbf{N}\|_\infty \leq \frac{|\widehat{f(n)}|^2}{4}\leq  |\widehat{f}(n_\ell)||\widehat{f}(n_{\ell+1})|=|t_\ell|. 
 \end{equation*} Therefore,
$0\leq \theta_\ell\leq \frac{\pi}{2}$, and so by  \eqref{eq:sin-estimate}, we have
\begin{equation*}
|\theta_{\ell}|\leq \frac{\pi}{2}\,|\sin(\theta_{\ell})|\leq 2\pi\,\frac{\|\mathbf{N}\|_\infty}{|\widehat{f}(n)|^2}.
\end{equation*}

By definition $\tau_{n}=\sum_{\ell=0}^{b-1} \arg(t_{\ell})$ and $
\alpha_{n}=\sum_{l=0}^{b-1} \arg(a'_{\ell})
$. Therefore,  
we have
\begin{align*}
|\mathbbm{e}^{\mathbbm{i}\tau_{n}}-\mathbbm{e}^{\mathbbm{i}\alpha_{n}}| \leq |\alpha_{n}-\tau_{n}|&=\bigg|\sum_{\ell=0}^{b-1} \arg(a'_{\ell})-\arg(t_{\ell}) \bigg|
=\bigg|\sum_{\ell=0}^{b-1} \theta_{\ell}\bigg| \leq 2\pi b\frac{\|\mathbf{N}\|_\infty}{|\widehat{f}(n)|^2}.
\end{align*}
From the definition of $n_\ell$, we have
 \begin{equation*}
 |n_\ell-n_{\ell-1}|\geq \gamma-\beta\geq \frac{\gamma}{2}
 \end{equation*}
 for all $1\leq \ell\leq b-1$. Therefore, the path length $b$ is bounded by 
 \begin{equation*}
     b\leq \frac{|n-n_0|}{\min |n_\ell-n_{\ell-1}|}\leq  \frac{2d}{\gamma}.
 \end{equation*}
 Thus, we have 
\begin{align*}
|\mathbbm{e}^{\mathbbm{i}\tau_{n}}-\mathbbm{e}^{\mathbbm{i}\alpha_{n}}|  \leq 2\pi b\frac{\|\mathbf{N}\|_\infty}{|\widehat{f}(n)|^2}
\leq \frac{4\pi d}{\gamma}\frac{\|\mathbf{N}\|_\infty}{|\widehat{f}(n)|^2}
\end{align*}
as desired.

\end{proof}

\section{Additional Numerical Simulations using Algorithms \ref{Big Algorithm} and \ref{Big Algorithm Compact}}
\label{sec:appendix_numerics}
\begin{figure}[htbp]
\centering
\begin{subfigure}[b]{0.245\textwidth}
\includegraphics[clip=true, trim = .85in 2.5in 0.85in 2.5in,scale=0.245]{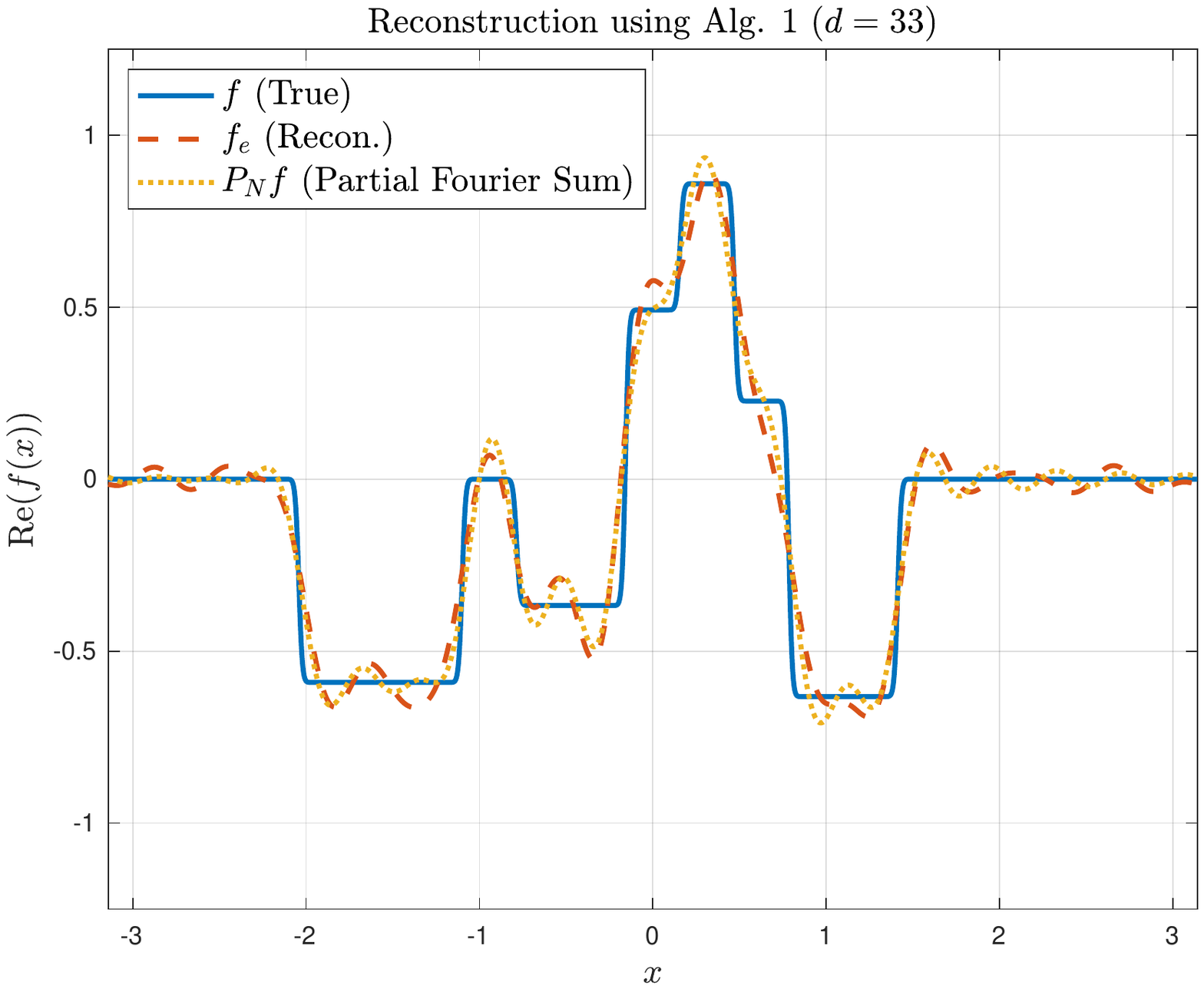}
\caption{$d=33$}
\end{subfigure}
\hfill
\begin{subfigure}[b]{0.245\textwidth}
\includegraphics[clip=true, trim = .85in 2.5in 0.85in 2.5in,scale=0.245]{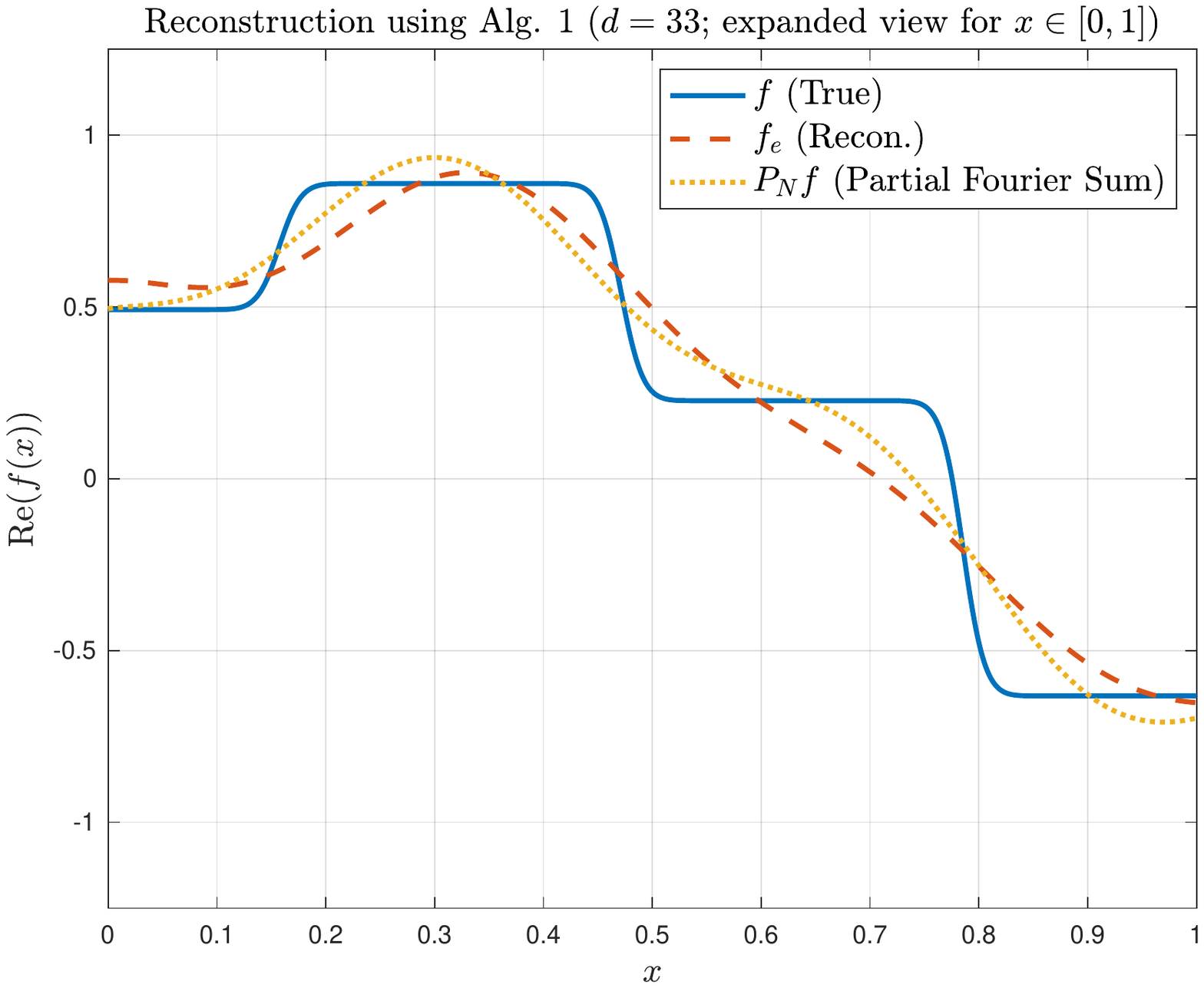}
\caption{$d=33$ (zoom)}
\end{subfigure}
\hfill
\begin{subfigure}[b]{0.245\textwidth}
\includegraphics[clip=true, trim = .85in 2.5in 0.85in 2.5in,scale=0.245]{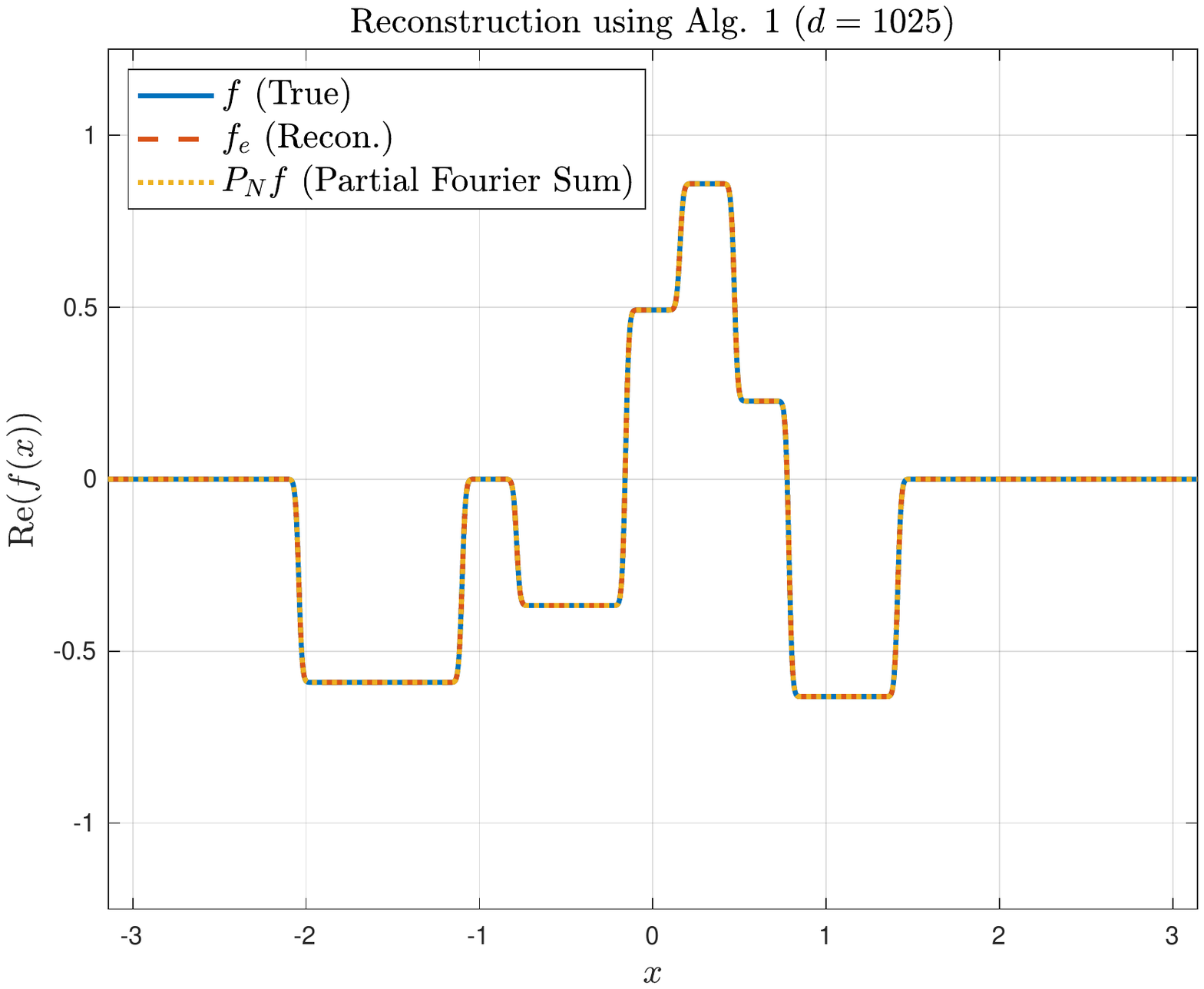}
\caption{$d=1025$}
\end{subfigure} 
\hfill
\begin{subfigure}[b]{0.245\textwidth}
\includegraphics[clip=true, trim = .85in 2.5in 0.85in 2.5in,scale=0.245]{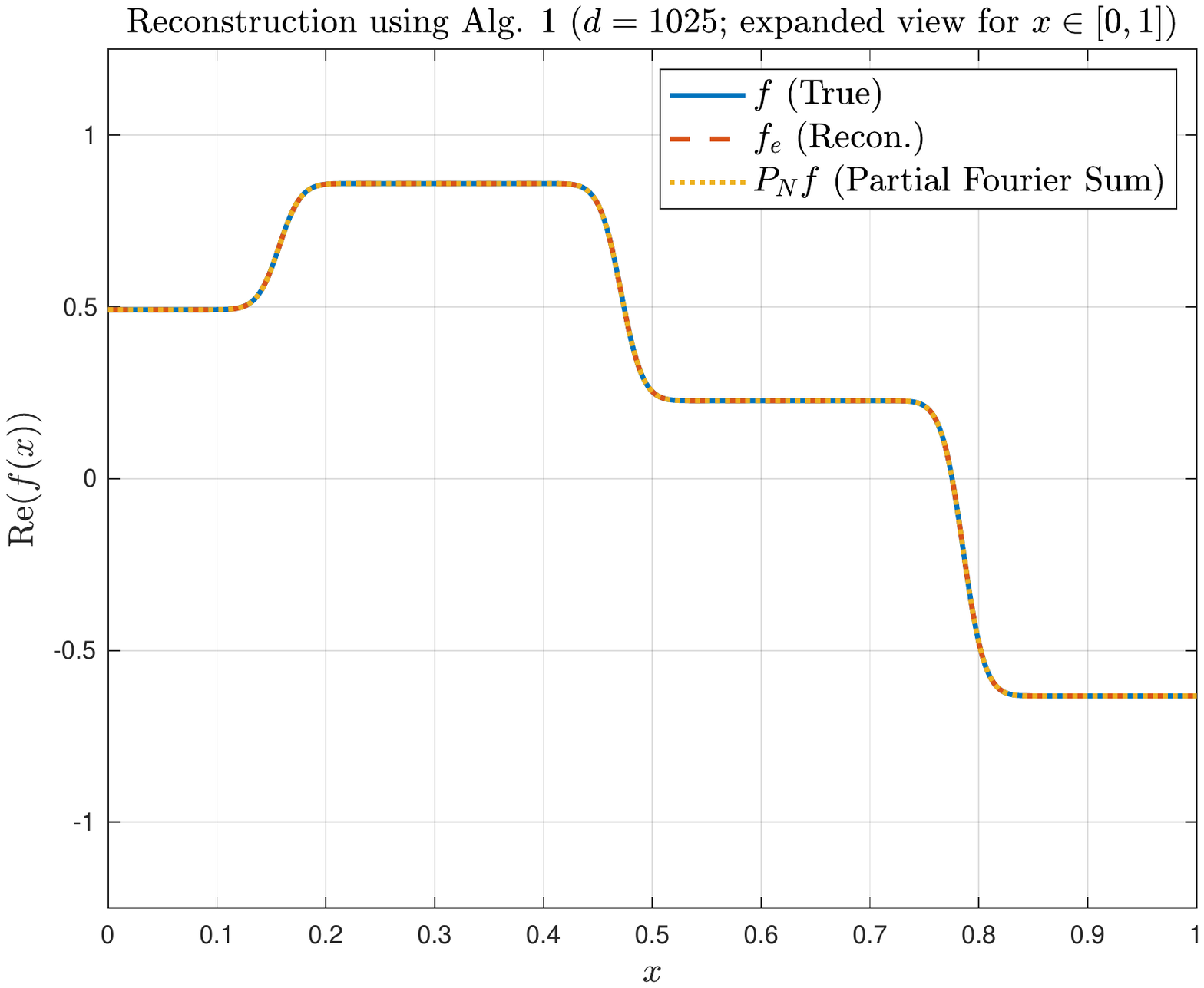}
\caption{$d=1025$ (zoom)}
\end{subfigure}
\caption{Evaluating the convergence behavior of Algorithm \ref{Big Algorithm}. Figure plots reconstructions of the real 
part of the test function at $d=33$ and $d=1025$ (along with an expanded view of the reconstruction in $[0,1]$) on a discrete equispaced grid in 
$[-\pi, \pi]$ of $7003$ points; we set $\rho=\min \{ (d-5)/2, 16\lfloor\log_2(d)\rfloor\}$ and $\kappa=\rho-1$.}
\label{fig:conv_alg1}
\end{figure}
In this section, we provide additional numerical simulations studying the empirical convergence behavior of Algorithms \ref{Big Algorithm} and \ref{Big Algorithm Compact}. We start with a study of the convergence behavior of Algorithm \ref{Big Algorithm}. Here, we reconstruct the same test function 
using different discretization sizes $d$ (with $\rho$ chosen to be $\min \{ (d-5)/2, 16\lfloor\log_2(d)\rfloor\}$ and $\kappa=\rho-1$), where the total number of 
phaseless measurements used is $Ld=(2\rho-1)d$. Fig. \ref{fig:conv_alg1} plots representative reconstructions (of the real part of the test function) for two choices of $d$ ($d=33$ and $d=1025$). We note that the (smooth) test function illustrated in the figure has several sharp and closely separated gradients, making the reconstruction process challenging. This is evident in the partial Fourier sums ($P_Nf$) plotted for reference alongside the reconstructions from Algorithm \ref{Big Algorithm} ($f_e$). For small $d$ and $\rho$, we observe  oscillatory behavior similar to that seen in the Gibbs phenomenon. Nevertheless, we see that the proposed algorithm closely tracks the performance of the partial Fourier sum, with reconstruction quality improving significantly as $d$ (and $\rho$) increases. 
\begin{figure}[htbp]
\centering
\begin{subfigure}[b]{0.245\textwidth}
\includegraphics[clip=true, trim = .85in 2.5in 0.85in 2.5in,scale=0.245]{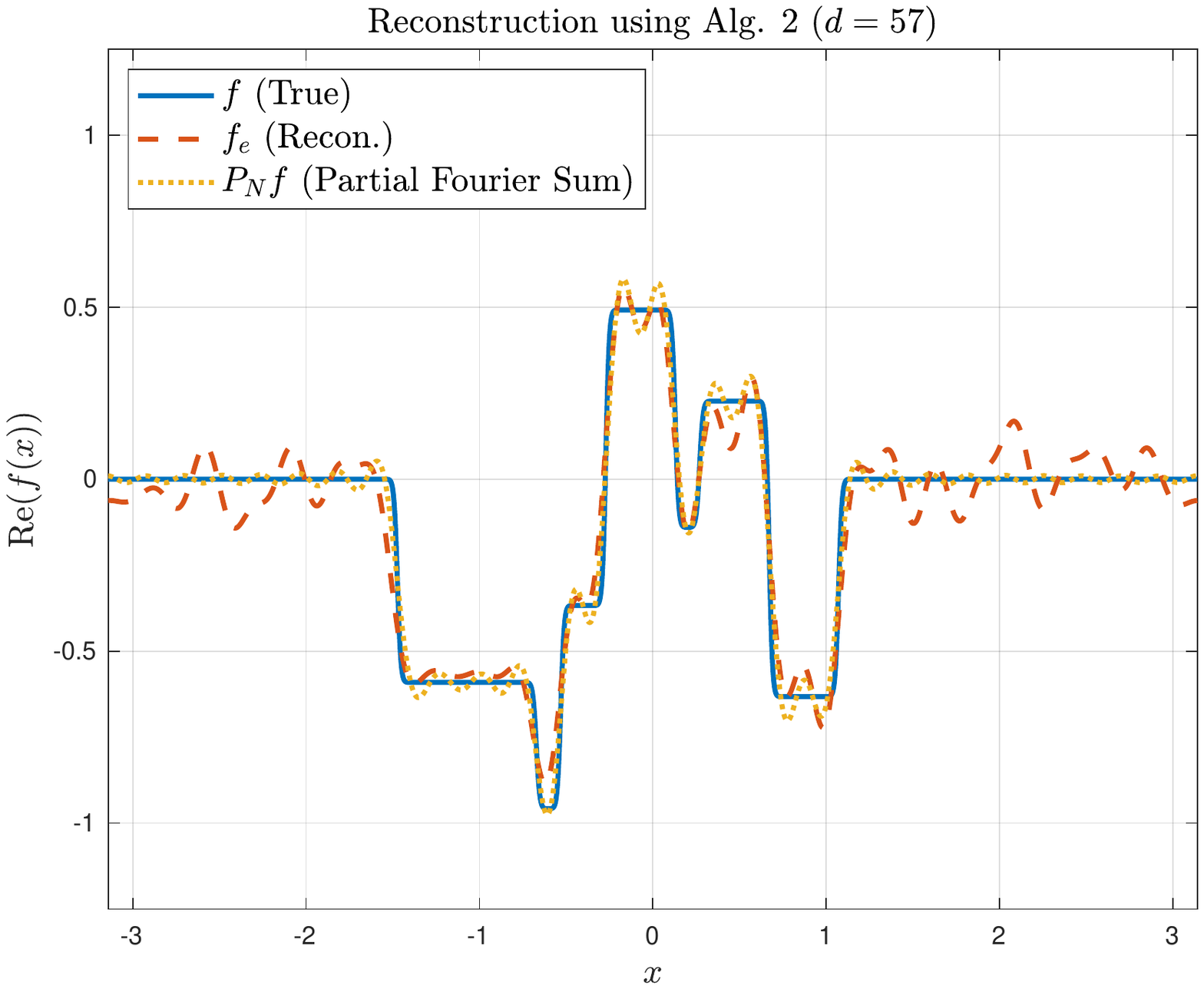}
\caption{$d=57$}
\end{subfigure}
\hfill
\begin{subfigure}[b]{0.245\textwidth}
\includegraphics[clip=true, trim = .85in 2.5in 0.85in 2.5in,scale=0.245]{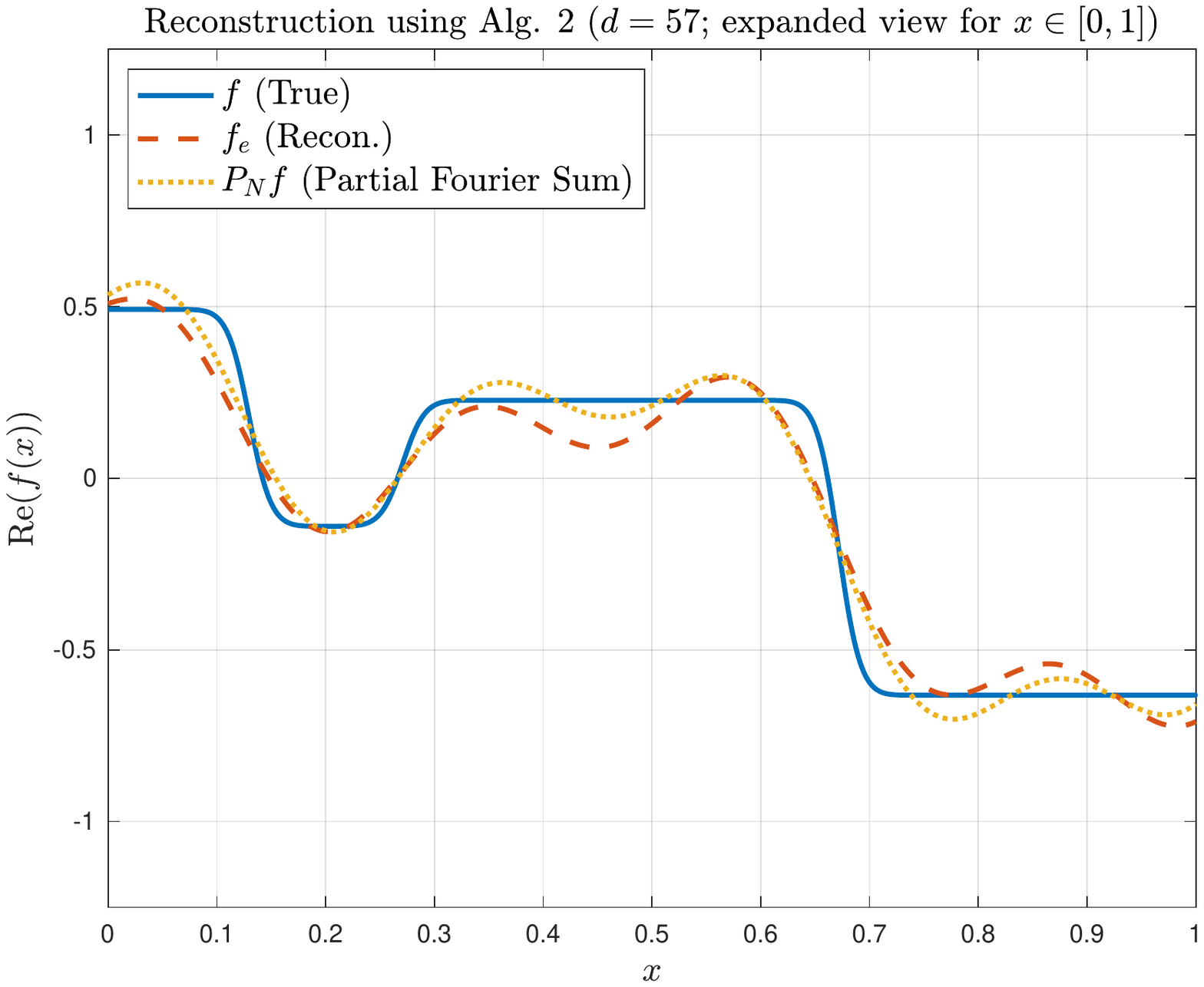}
\caption{$d=57$ (zoom)}
\end{subfigure}
\hfill
\begin{subfigure}[b]{0.245\textwidth}
\includegraphics[clip=true, trim = .85in 2.5in 0.85in 2.5in,scale=0.245]{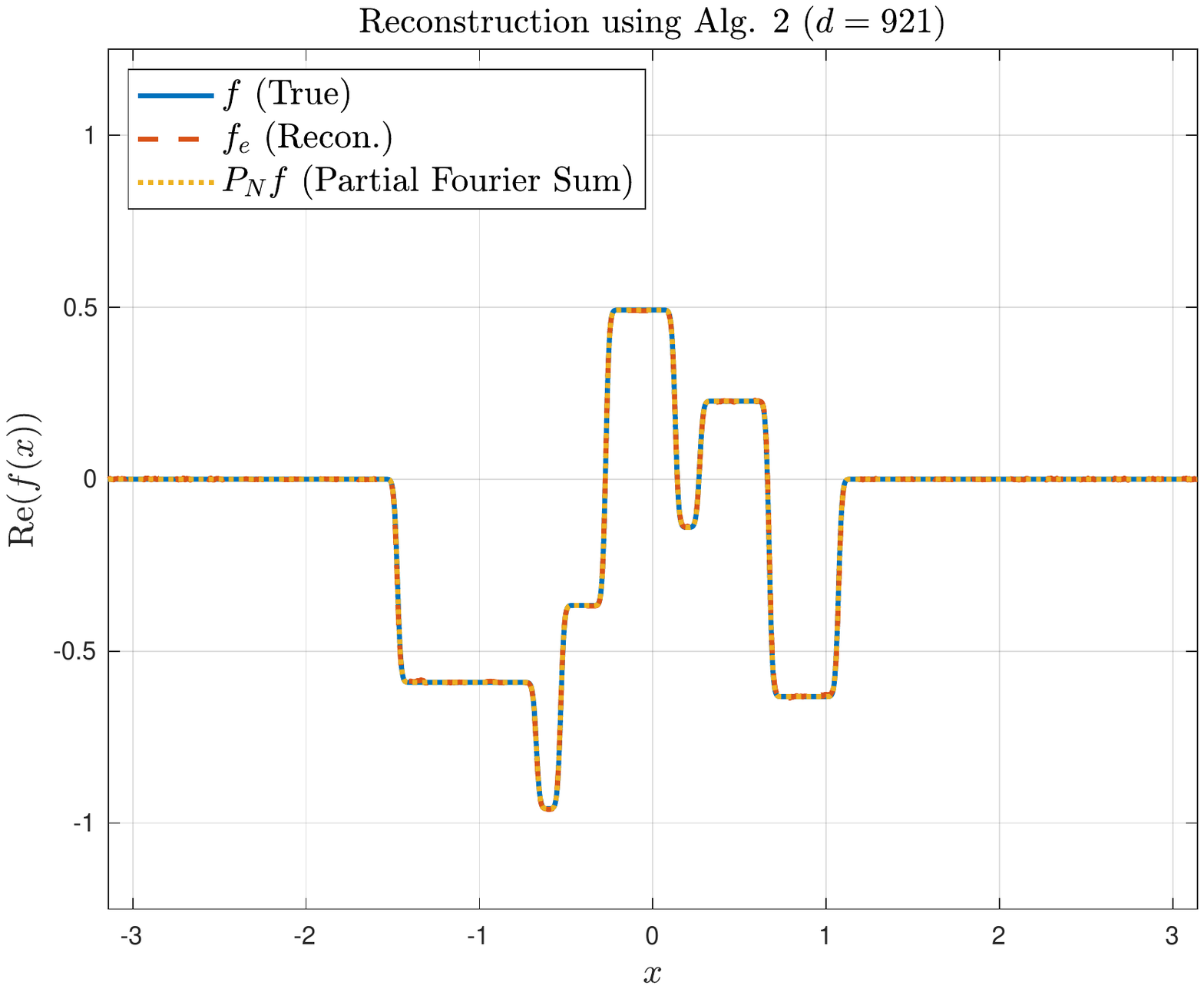}
\caption{$d=921$}
\end{subfigure} 
\hfill
\begin{subfigure}[b]{0.245\textwidth}
\includegraphics[clip=true, trim = .85in 2.5in 0.85in 2.5in,scale=0.245]{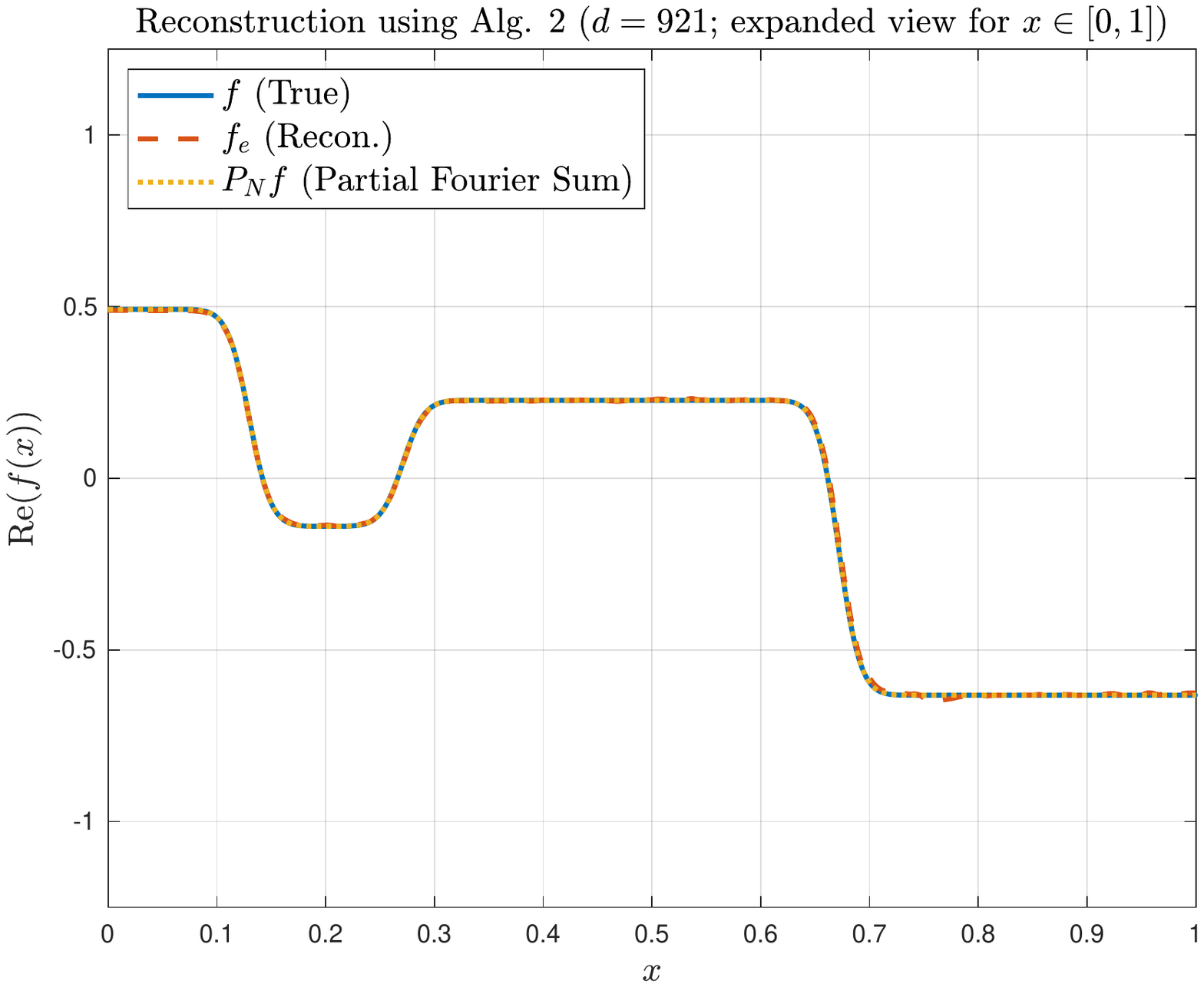}
\caption{$d=921$ (zoom)}
\end{subfigure}
\caption{Evaluating the convergence behavior of Algorithm \ref{Big Algorithm Compact}. Figure plots reconstructions of the real 
part of the test function at $d=57$ and $d=921$ (along with an expanded view of the reconstruction in $[0,1]$) on a discrete equispaced grid in 
$[-\pi, \pi]$ of $7003$ points; we set $K=d/3,\delta=(K+1)/2$ and $\kappa=\delta-1$.}
\label{fig:conv_alg2}
\end{figure}
%

We next evaluate the convergence behavior of Algorithm\footnote{using the alternate implementation -- with (\ref{eq:FYF_single_aliasing}) utilized in place of (\ref{eq:FYF_single_aliasing-modified}) in Step $2$ of the Algorithm -- as described in Section \ref{sec:Eval}} \ref{Big Algorithm Compact} by reconstructing the same test function 
using different discretization sizes $d$ (with $K=d/3$, $\delta=(K+1)/2$, $\kappa=\delta-1$ and $s=\kappa-1$). Fig. \ref{fig:conv_alg2} plots representative reconstructions (of the real part of the test function) for two choices of $d$ ($d=57$ and $d=921$). As in Fig. \ref{fig:conv_alg1}, we note that the (smooth) test function has several sharp and closely separated gradients, making the reconstruction process challenging. Again, the partial Fourier sums ($P_Nf$) plotted alongside the reconstructions from Algorithm \ref{Big Algorithm Compact} ($f_e$) exhibit Gibbs-like oscillatory behavior for small $d$ and $\kappa$. Nevertheless, we see that the proposed algorithm closely tracks the performance of the partial Fourier sum, with reconstruction quality improving significantly as $d$ (and $\delta,\kappa$) increases.

\end{document}